\documentclass[a4paper,11pt]{article}
\usepackage{array}
\usepackage{theorem}
\usepackage{amsmath,amscd,amssymb}
\usepackage{latexsym}
\usepackage{epsfig}
\usepackage{xypic}
\theorembodyfont{\sl}

\newtheorem{lemma}{Lemma}[section]

\newtheorem{proposition}[lemma]{Proposition}
\newtheorem{theorem}[lemma]{Theorem}
\newtheorem{corollary}[lemma]{Corollary}
\newtheorem{definition}[lemma]{Definition}

\theorembodyfont{\rmfamily}
\newtheorem{remark}[lemma]{Remark}
\newtheorem{example}[lemma]{Example}

\newcommand{\CC}{\mathbb C}

\newcommand{\NN}{\mathbb N}

\newcommand{\PP}{\mathbb P}
\newcommand{\QQ}{\mathbb Q}
\newcommand{\RR}{\mathbb R}

\newcommand{\VV}{\mathbb V}
\newcommand{\WW}{\mathbb W}

\newcommand{\ZZ}{\mathbb Z}

\newcommand{\cC}{\mathcal C}
\renewcommand{\cD}{\mathcal D}

\newcommand{\cF}{\mathcal F}

\renewcommand{\cH}{\mathcal H}

\renewcommand{\cL}{\mathcal L}
\newcommand{\cM}{\mathcal M}
\newcommand{\cN}{\mathcal N}
\newcommand{\cO}{\mathcal O}

\newcommand{\cS}{\mathcal S}

\newcommand{\cV}{\mathcal V}

\newcommand{\cX}{\mathcal X}

\newcommand{\cZ}{\mathcal Z}

\newcommand{\isoto}{\stackrel{\sim}{\to}}

\newcommand{\To}{\longrightarrow}
\newcommand{\Mapsto}{\mapstochar\longrightarrow}

\newcommand{\tensor}{\otimes}
\newcommand{\emb}{\hookrightarrow}

\renewcommand{\Tilde}{\widetilde}
\renewcommand{\Hat}{\widehat}
\renewcommand{\Bar}{\overline}

\newcommand{\cross}{\times}

\newcommand{\imic}{\cong}

\newcommand{\ratmap}{\dasharrow}

\newcommand{\GL}{\mathop{\mathrm {GL}}\nolimits}

\newcommand{\PGL}{\mathop{\mathrm {PGL}}\nolimits}

\newcommand{\SL}{\mathop{\mathrm {SL}}\nolimits}
\newcommand{\SO}{\mathop{\mathrm {SO}}\nolimits}
\newcommand{\Sp}{\mathop{\mathrm {Sp}}\nolimits}
\newcommand{\SU}{\mathop{\mathrm {SU}}\nolimits}
\newcommand{\Orth}{\mathop{\null\mathrm {O}}\nolimits}
\newcommand{\Ref}{\mathop{\null\mathrm {Ref}}\nolimits}

\newcommand{\Proj}{\mathop{\null\mathrm {Proj}}\nolimits}

\newcommand{\re}{\mathop{\mathrm {Re}}\nolimits}
\newcommand{\im}{\mathop{\mathrm {Im}}\nolimits}

\renewcommand{\Im}{\mathop{\mathrm {Im}}\nolimits}

\newcommand{\rk}{\mathop{\mathrm {rk}}\nolimits}
\newcommand{\rank}{\mathop{\mathrm {rank}}\nolimits}

\newcommand{\pr}{\mathop{\mathrm {pr}}\nolimits}
\newcommand{\hcf}{\mathop{\mathrm {hcf}}\nolimits}

\newcommand{\reg}{\mathop{\mathrm {reg}}\nolimits}
\newcommand{\sing}{\mathop{\mathrm {sing}}\nolimits}
\newcommand{\Hom}{\mathop{\mathrm {Hom}}\nolimits}
\newcommand{\Aut}{\mathop{\mathrm {Aut}}\nolimits}

\newcommand{\id}{\mathop{\mathrm {id}}\nolimits}

\newcommand{\divv}{\mathop{\null\mathrm {div}}\nolimits}

\newcommand{\sm}{\mathop{\null\mathrm {sm}}\nolimits}
\newcommand{\Hilb}{\mathop{\null\mathrm {Hilb}}\nolimits}
\newcommand{\Km}{\mathop{\null\mathrm {Km}}\nolimits}
\newcommand{\NS}{\mathop{\null\mathrm {NS}}\nolimits}
\newcommand{\Pic}{\mathop{\null\mathrm {Pic}}\nolimits}
\newcommand{\Mon}{\mathop{\null\mathrm {Mon}}\nolimits}

\newcommand{\Def}{\mathop{\null\mathrm {Def}}\nolimits}
\newcommand{\VSP}{\mathop{\null\mathrm {VSP}}\nolimits}
\newcommand{\Gr}{\mathop{\null\mathrm {Gr}}\nolimits}

\newcommand{\supp}{\mathop{\null\mathrm {Supp}}\nolimits}
\newcommand{\Bdiv}{\mathop{\null\mathrm {Bdiv}}\nolimits}
\newcommand{\trdeg}{\mathop{\null\mathrm {tr.deg}}\nolimits}
\newcommand{\bdT}{\mathbf T}

\newcommand{\bdv}{\mathbf v}
\newcommand{\bdw}{\mathbf w}

\newcommand{\bdN}{\mathbf N}

\newcommand{\latt}[1]{{\langle{#1}\rangle}}
\renewcommand{\div}{\mathop{\mathrm {div}}\nolimits}
\newcommand{\Kthree}{\mathop{\mathrm {K3}}\nolimits}

\newcommand{\qedsymbol}{\mbox{$\Box$}}
\newcommand{\qed}{\unskip\nobreak\hfil\penalty50\hskip1em\hbox{}\nobreak
\hfill\qedsymbol\parfillskip=0pt\finalhyphendemerits=0}
\newenvironment{proof}{\begin{ProofwCaption}{Proof}}{\end{ProofwCaption}}
\newenvironment{ProofwCaption}[1]
 {\addvspace\theorempreskipamount \noindent{\it #1.}\rm}
 {\qed \par \addvspace\theorempostskipamount}

\begin{document}

\title{Moduli of $\Kthree$ Surfaces and Irreducible Symplectic Manifolds}
\author{V.~Gritsenko, K.~Hulek and G.K.~Sankaran}
\date{}
\maketitle
\begin{abstract}
\noindent
The name ``K3 surfaces" was coined by A. Weil in 1957 when he formulated
a research programme for these surfaces and their moduli.
Since then, irreducible holomorphic symplectic manifolds have been
introduced as a higher dimensional analogue of $\Kthree$ surfaces.
In this paper we present a review of this theory starting from the
definition of $\Kthree$ surfaces and going as far as the global Torelli theorem
for irreducible holomorphic symplectic manifolds as recently proved by
M. Verbitsky.

For many years the last open question of Weil's programme was that of
the geometric type of the moduli spaces of polarised $\Kthree$ surfaces.
We explain how this problem has been solved.
Our method uses algebraic geometry, modular forms and Borcherds
automorphic products. We collect and discuss the relevant facts
from the theory of  modular forms with respect to the orthogonal group
$\Orth(2,n)$. We also give a detailed description of quasi pull-back
of automorphic Borcherds products. This part contains previously
unpublished  results. 
\end{abstract}

\tableofcontents

\section{Introduction}\label{sect:introduction}

The history of $\Kthree$ surfaces goes back to classical algebraic
geometry. In modern complex algebraic geometry a \emph{K3 surface} is
a compact complex surface $S$ whose canonical bundle is trivial,
i.e.\ $K_S\cong\cO_S$, and whose irregularity $q(S)=h^1(S,\cO_S)=0$.
These two facts immediately imply that
$H^0(S,\Omega^1_S)=H^0(S,T_S)=H^2(S,T_S)=0$.

The easiest examples of algebraic $\Kthree$ surfaces are smooth quartic
surfaces in $\PP^3$. Further examples are complete intersections of
type $(2,3)$ in $\PP^4$ and of type $(2,2,2)$ in $\PP^5$. Another
classical example is the \emph{Kummer surfaces}, i.e.\ the
(desingularised) quotient of a $2$-dimensional torus $A$ by the
involution $\iota\colon  x \mapsto -x$.

The modern name ``$\Kthree$ surface'' was coined by A.~Weil in his
famous ``Final report on research contract AF 18(603)-57'' \cite[Vol
  II, pp 390--395]{Wei}. In his comments on this report \cite[Vol II,
  p 546]{Wei} Weil writes: ``Dans la seconde partie de mon rapport, il
s'agit des vari\'et\'es k\"ahl\'eriennes dites K3, ainsi nomm\'ees en
l'honneur de Kummer, Kodaira, K\"ahler et de la belle montagne K2 au
Cachemire.''

In that report the following conjectures, due to Andreotti and Weil,
were stated:
\begin{itemize}
\item[{\rm (i)}] $\Kthree$ surfaces form one family;
\item[{\rm (ii)}] all $\Kthree$ surfaces are K\"ahler;
\item[{\rm (iii)}] the period map is surjective;
\item[{\rm (iv)}] a form of global Torelli theorem holds.
\end{itemize}
Weil also remarked that the structure of the moduli space of
(polarised) $\Kthree$ surfaces must be closely related to the theory
of automorphic forms. 

By now all of these questions have been answered positively and much
progress has been made in understanding the moduli spaces of $\Kthree$
surfaces. Conjecture (i) was proved by Kodaira \cite[Part I, theorem
  19]{Kod}.  Conjecture (ii) was first shown by Siu
\cite[Section 14]{Siu}.  Nowadays it can be derived from the more
general theorem, due to Buchdahl and Lamari, that a compact complex
surface is K\"ahler if and only if its first Betti number is even
(\cite[Theorem IV.3.1]{BHPV}, \cite{Bu},\cite{La}).  Surjectivity of
the period map (conjecture (iii)) was proved for special $\Kthree$
surfaces in various papers by Shah \cite{Sha1}, \cite{Sha2},
\cite{Sha3} and by Horikawa \cite{Hor}.  Kulikov \cite{Kul} gave a
proof for projective $\Kthree$ surfaces, with some points clarified by
Persson and Pinkham \cite{PP}. The general result for K\"ahler
$\Kthree$ surfaces was proved by Todorov \cite{Tod} and Looijenga
\cite{Lo}.  The strong Torelli theorem (i.e.\ conjecture (iv)) for
algebraic $\Kthree$ surfaces was first proved by Piatetskii-Shapiro
and Shafarevich~\cite{P-SS} with amendments by Rapoport and Shioda. It
was proved for K\"ahler $\Kthree$ surfaces (and hence for all) by
Burns and Rapoport \cite{BR}.  A detailed and supplemented account of
the original proof was written by Looijenga and Peters \cite{LP}.
Finally, Friedman \cite{Fr} gave a proof using degenerations.

It should be noted, though, that Weil's definition of $\Kthree$
surface was different from the standard definition used nowadays. For
him a surface was $\Kthree$ if its underlying differentiable structure
was that of a quartic surface in $\PP^3$. Using results from
Seiberg-Witten theory one can indeed show that any compact complex
surface diffeomorphic to a quartic is a $\Kthree$ surface.

The crucial ingredient is the fact that the plurigenera of a compact
complex surface are invariant under orientation preserving
diffeomorphism: see \cite[Theorem IX.9.4 and Theorem
  IX.9.5]{BHPV}. This was formulated as a question in
\cite[p. 244]{OV} and became known as the Van de Ven conjecture.  A
first (partial) proof of this was published by Brussee \cite{Br} (see
also \cite{FM}), and an elegant argument using Seiberg-Witten
invariants can be found in D\"urr's thesis \cite{Du}. It follows that
the only surfaces which admit an orientation preserving diffeomorphism
to a quartic are $\Kthree$ surfaces or possibly tori, but the latter
are excluded because they are not simply connected. If a surface is
diffeomorphic to a quartic surface, but with an orientation changing
diffeomorphism, then its signature is $-16$ and its Euler number
equals $c_2(S)=24$. By the Thom-Hirzebruch index theorem
$c_1^2(S)=96$, but that contradicts the Miyaoka-Yau inequality, so
such a surface cannot exist.

Moduli spaces of polarised $\Kthree$ surfaces are a historically old
subject, studied by the classical Italian geometers (starting with the
spaces of double sextics and plane quartics in $\PP^3$).  Mukai
extended the classical constructions and unirationality results for
the moduli spaces $\cF_{2d}$ parametrising polarised $\Kthree$
surfaces of degree $2d$ to many more cases, going as high as
$d=19$. See \cite{Mu1}, \cite{Mu2}, \cite{Mu3}, \cite{Mu4} and
\cite{Mu5} for details.  Kondo \cite{Ko1} proved that the moduli
spaces $\cF_{2p^2}$ are of general type for sufficiently large prime
numbers~$p$, but without giving an effective bound.  Finally, it was
shown in \cite{GHSK3} that the moduli spaces $\cF_{2d}$ are of general
type for $d > 61$ and $d= 46,50,54,57,58,60$, and later it was noticed
by A.~Peterson \cite{PeSa} that this proof also works for $d=52$: see
Theorem~\ref{thm:K3gt} below. For an account of these results see also
Voisin's Bourbaki expos\'e \cite{Vo1}.

In higher dimension $\Kthree$ surfaces can be generalised in two ways,
namely to Calabi-Yau varieties or to irreducible symplectic manifolds
(or hyperk\"ahler manifolds). In fact, together with tori, these three
classes of manifolds are the building blocks of all compact K\"ahler
manifolds with trivial first Chern class (over $\RR$), also called
Ricci flat manifolds: every such manifold admits a finite \'etale cover which
is a product of tori, Calabi-Yau manifolds and irreducible symplectic
manifolds (see \cite{Be}, \cite{Bog2}).  The first examples of
irreducible symplectic manifolds, by now regarded as classical, were
studied by Beauville in \cite{Be}, namely the Hilbert schemes of
points on $\Kthree$ surfaces and generalised Kummer varieties (and
their deformations). Two further classes of examples were later found
by O'Grady \cite{OG1}, \cite{OG2}.  The theory of irreducible
symplectic manifolds, which started from work of Bogomolov,
Beauville and Fujiki, was significantly advanced by Huybrechts \cite{Huy1} who,
among other results, proved surjectivity of the period map. It was
noticed early on by Debarre \cite{Deb} and Namikawa \cite{Nam} that
the obvious generalisations of the Torelli theorem are
false. Nevertheless, one can use the period map to exhibit moduli
spaces of polarised irreducible symplectic manifolds as dominant
finite to one covers of quotients of type~IV domains by arithmetic
groups. This was used in \cite{GHSsymp} and \cite{GHSdim21} to obtain
general type results for many of these moduli spaces
(Theorem~\ref{thm:splitgt} and Theorem~\ref{thm:21gt} below).

Very recently Verbitsky \cite{Ver} has announced a Torelli theorem for
irreducible symplectic manifolds.  The consequences of Verbitsky's
result for the moduli problem of polarised irreducible symplectic
manifolds were worked out in detail by Markman \cite{Mar4}. We also refer
the reader to Huybrecht's forthcoming Bourbaki talk \cite{Huy2}.

The theory of $\Kthree$ surfaces and irreducible symplectic manifolds
is a fascinating and vast subject.  We have started this introduction
by giving the definition of $\Kthree$ surfaces in complex geometry
(which also allows for non-algebraic surfaces).  The notion of
$\Kthree$ surface also makes perfect sense in algebraic geometry over
arbitrary fields: a $\Kthree$ surface is an irreducible smooth
algebraic surface $S$ defined over a field $k$ with trivial canonical
bundle and irregularity $q=h^1(S,\cO_S)=0$.  In this article we shall,
however, concentrate on the theory over the complex numbers, and
especially on moduli problems.  We are fully aware that we thus
exclude a large area of interesting questions.  Concerning $\Kthree$
surfaces in positive characteristic we refer the reader to the papers
by Artin \cite{Ar}, Ogus \cite{Og} and Nygaard \cite{Ny} for a first
introduction.  Another aspect, which we will not touch upon, is the
arithmetic of $\Kthree$ surfaces.  An introduction to this field can
be found in Sch\"utt's survey paper \cite{Sch}.

In this article we mainly survey known
results. Theorem~\ref{thm:qpbcusp} is new, however (special cases
occur in the literature) and so is Proposition~\ref{prop:nonsplitgt}.

\section{Periods of $\Kthree$ surfaces and the Torelli
  theorem}\label{sect:K3periods}
In this section we will discuss the period domain and the period map
for $\Kthree$ surfaces, and the local and global Torelli theorems. The
case of $\Kthree$ surfaces presents several special features and the
existence of a straightforward global Torelli theorem is one of them.

\subsection{Lattices}\label{subsect:lattices}
Let $S$ be a $\Kthree$ surface. By Noether's formula $b_2(S)=22$, and
since $S$ is simply connected (being diffeomorphic to a quartic
hypersurface), $H^2(S,\ZZ)$ is a free $\ZZ$-module of rank $22$.  The
intersection form defines a non-degenerate symmetric bilinear form on
$H^2(S,\ZZ)$, giving it the structure of a lattice, which has
signature $(3,19)$ by the Hodge index theorem. As the intersection
form is even and unimodular, this lattice is, independently of the
surface $S$, the \emph{K3 lattice}
\begin{equation}\label{K3lattice}
L_{\Kthree}:=3U \oplus 2E_8(-1)
\end{equation}
where $U$ is the hyperbolic plane (the unique even unimodular lattice
of signature $(1,1)$) and $E_8$ is the unique even unimodular positive
definite rank $8$ lattice.  For any lattice $L$ and $m\in\ZZ$, the
notation $L(m)$ indicates that the quadratic form is multiplied by
$m$, so $E_8(-1)$ is negative definite. 

It was proven by Siu~\cite{Siu} that every $\Kthree$ surface is
K\"ahler. Today, it is known that a compact complex surface is
K\"ahler if and only if its first Betti number is even \cite[Theorem
  IV.3.1]{BHPV}, which immediately implies Siu's result. Hence
$H^2(S,\CC)$ has a Hodge decomposition
\begin{equation*}
H^2(S,\CC)=H^{2,0}(S) \oplus H^{1,1}(S) \oplus H^{0,2}(S).
\end{equation*}
Since $K_S\imic\cO_S$ it follows that $H^{2,0}(S)=H^0(S,K_S)$ is
$1$-dimensional and thus $H^{1,1}(S)$ has dimension $20$.

Since the second cohomology has no torsion we can, via the universal
coefficient theorem, consider $H^2(S,\ZZ)$ as a lattice in
$H^2(S,\CC)$.  The \emph{N\'eron-Severi group} is the intersection
$\NS(S)=H^2(S,\ZZ) \cap H^{1,1}(S)$, which, in this case, coincides
with the Picard group $\Pic(S)$. The \emph{transcendental lattice} is
defined as the smallest lattice whose complexification contains a
generator $\omega$ of $H^0(S,K_S)$.  If the intersection pairing on
$\NS(S)$ is nondegenerate, e.g.\ if $S$ is projective, then $T(S)$ is
the orthogonal complement in $H^2(S,\ZZ)$ of the N\'eron-Severi group.

Note that for a general $\Kthree$ surface $S$ we have no algebraic
classes, i.e.\ $T(S)=H^2(S,\ZZ)$.  The \emph{Picard number} $\rho(S)$
of $S$ is the rank of the N\'eron-Severi group $\NS(S)$.

For future use we also need the \emph{K\"ahler cone} of a $\Kthree$
surface, which is the cone of classes of K\"ahler $(1,1)$-forms. This
lives in $H^{1,1}(S,\RR)=H^{1,1}(S) \cap H^2(S,\RR)$.  The restriction
of the intersection product to $H^{1,1}(S,\RR)$ has signature
$(1,19)$.  Let
\begin{equation*}
\cC_S\subset \{ x \in H^{1,1}(S,\RR) \mid (x,x) >0 \}
\end{equation*}
be the connected component that contains one (and hence
all) K\"ahler classes. This is called the \emph{positive cone} of $S$.

A class in $H^2(S,\ZZ)$ is called \emph{effective} if it is represented
by an effective divisor.  By a \emph{nodal} class we mean the class
$\delta$ of an effective divisor $D$ of self-intersection $D^2=-2$.
We denote by $\Delta$ the set of all nodal classes.  Every nodal class
$\delta \in \Delta$ defines a reflection
\begin{equation*}
s_\delta\colon  H^2(S,\ZZ) \to H^2(S,\ZZ)
\end{equation*}
given by $s_\delta(x)=x+(x,\delta)\delta$ and called the
\emph{Picard-Lefschetz reflection} defined by $\delta$. We shall
denote the $\RR$- and $\CC$-linear extensions of $s_\delta$ by the
same symbol. Clearly $s_\delta$ is the reflection in the hyperplane
$H_\delta$ orthogonal to $\delta$.  The set of effective classes on
$S$ is the semi-group generated by the integral points in the closure
of $\cC_S$ and the nodal classes.  The connected components of the set
$\cC_S \setminus \bigcup_{\delta \in \Delta} H_\delta$ are called the
\emph{chambers} of $\cC_S$. The chamber
\begin{equation}\label{kaehlercone}
\cC_S^+= \{ x \in \cC_S \mid (x,\delta)>0 \mbox{ for all effective } \delta\in\Delta \}
\end{equation}
is equal to the K\"ahler cone of $S$ according to \cite[Corollary VIII.3.9]{BHPV}.

\subsection{Markings and the period map}\label{subsect:markings}

A \emph{marking} of a $\Kthree$ surface is an isometry $\phi\colon
H^2(S, \ZZ) \to L_{\Kthree}$ and we refer to a pair $(S,\phi)$ as a
\emph{marked} $\Kthree$ surface. An isomorphism between marked
$\Kthree$ surfaces is an isomorphism between the surfaces that
commutes with the markings. If $\omega$ is a non-zero $2$-form on
$S$ then $\CC \phi(\omega)= \phi(H^{2,0}(S))$ is a line in the complex
vector space $L_{\Kthree}\otimes \CC$.

For any indefinite lattice $L$ we define
\begin{equation}\label{generalperioddomain}
\Omega_L=\{ [x] \in \PP(L_{\Kthree} \otimes \CC) \mid (x,x)=0, \,
(x,\bar{x}) >0 \}.
\end{equation}
In the case of the $\Kthree$ lattice, $\Omega=\Omega_{L_{\Kthree}}$ is a
connected complex manifold of dimension $20$, called the \emph{period
  domain of K3 surfaces}.  Since $(\omega,\omega)=0$ and
$(\omega,\bar{\omega})>0$ it follows that $[\phi(\omega)] \in
\Omega$. This is the \emph{period point} of the marked $\Kthree$
surface $(S,\phi)$.

Let $p\colon \cS \to U$ be a flat family of $\Kthree$ surfaces over
some sufficiently small contractible open set $U$.  If $\phi_0 \colon
H^2(S_0,\ZZ) \to L_{\Kthree}$ is a marking then this can be extended
to a marking $\phi_U \colon R^2p_{*} \ZZ_U \to (L_{\Kthree})_U$ where
$(L_{\Kthree})_U$ is the constant sheaf with fibre $L_{\Kthree}$ on
$U$. This defines a holomorphic map $\pi_U\colon U \to \Omega$, called
the \emph{period map} defined by the family $p\colon \cS \to U$.

\subsection{The Torelli theorem}\label{subsect:K3torelli}

The Torelli problem asks how much information about an algebraic
variety can be reconstructed from its Hodge structure.  In the case of
$\Kthree$ surfaces this means whether one can recover a $\Kthree$
surface $S$ from a period point.  This question can be made precise in
different ways. As we shall see, one can prove a very strong form of
the Torelli theorem in the case of $\Kthree$ surfaces.

We start by discussing the local Torelli theorem. For this let
$p\colon \cS \to U$ be a representative of the Kuranishi family (or
versal deformation) of $S$. Since $H^0(S,T_S)=H^2(S,T_S)=0$ the base
space of the Kuranshi family is smooth of dimension
$h^1(S,T_S)=h^{1,1}(S)=20$. Choosing any marking of the central fibre
of the Kuranishi family defines a marking for the entire family (we
shall choose $U$ sufficiently small and contractible) and hence a
period map $\pi\colon U \to \Omega$.

\begin{theorem}\label{thm:localtorelliK3}\emph{(Local Torelli)}
The base space of the Kuranishi family is smooth of dimension $20$. It
is universal for all points in a sufficiently small neighbourhood of
the origin. The differential of the period map is an isomorphism and
thus the period map is a local isomorphism.
\end{theorem}
\begin{proof}
See \cite[Theorem VIII.7.3]{BHPV}.
\end{proof}

In order to discuss the global Torelli theorem we need the notion of
Hodge isometry.  Let $S$ and $S'$ be $\Kthree$ surfaces. An
isomorphism of $\ZZ$-modules $\Phi\colon H^2(S,\ZZ) \to H^2(S',\ZZ)$
is called a \emph{Hodge isometry} if it is an isometry and if its
$\CC$-linear extension $\Phi_{\CC}\colon H^2(S,\CC) \to H^2(S',\CC)$
preserves the Hodge decomposition. It is moreover called
\emph{effective} if it maps $\cC_S$ to $\cC_S'$ and maps effective
classes to effective classes.

\begin{proposition}\label{prop:effectiveHodgeisometry}
Let $S$ and $S'$ be $\Kthree$ surfaces. Then the following are
equivalent for a Hodge isometry $\Phi\colon H^2(S,\ZZ) \to
H^2(S',\ZZ)$:
\begin{enumerate}
\item[\rm(i)] $\Phi$ is effective,
\item[\rm(ii)] $\Phi(\cC^+_S) \subset \cC^+_{S'}$, i.e.\ $\Phi$ maps
  the K\"ahler cone of $S$ into that of $S'$.
\item[\rm(iii)] $\Phi$ maps one element of the K\"ahler cone of $S$
  into the K\"ahler cone of the surface $S'$.
\end{enumerate}
\end{proposition}
\begin{proof}
See \cite[Proposition VIII.3.11]{BHPV}.
\end{proof}

The crucial result for the theory of moduli of $\Kthree$ surfaces is
the following.

\begin{theorem}\label{thm:strongtorelli}\emph{(Strong Torelli)}
Let $S$ and $S'$ be two $\Kthree$ surfaces and suppose $\Phi\colon
H^2(S',\ZZ) \to H^2(S,\ZZ)$ is an effective Hodge isometry.  Then
there is a unique isomorphism $f\colon S \to S'$ that induces $\Phi$,
i.e.\ such that $\Phi=f^*$.
\end{theorem}
A proof of this theorem can be found in \cite[Sections
  VIII.7--VIII.11]{BHPV}. Very roughly speaking the idea is to prove
the Torelli theorem for projective Kummer surfaces first. The second
step then consists of showing that the period points of marked Kummer
surfaces are dense (in the complex topology) in the period domain
$\Omega$. The final step is then to prove the Torelli theorem for all
$\Kthree$ surfaces by approximating them by Kummer surfaces and taking
suitable limits in the Barlet topology.

The following weaker form of the Torelli theorem is still useful.

\begin{theorem}\label{thm:weaktorelli}\emph{(Weak Torelli theorem)}
Two $\Kthree$ surfaces $S$ and $S'$ are isomorphic if and only if
there is a Hodge isometry $H^2(S',\ZZ) \to H^2(S,\ZZ)$.
\end{theorem}
\begin{proof}
Assume that $\Phi\colon H^2(S',\ZZ) \to H^2(S,\ZZ)$ is a Hodge
isometry. Let $W_S$ be the group of isometries of $H^2(S,\ZZ)$
generated by the Picard-Lefschetz reflections. This group acts on the
positive cone $\cC_S$ properly discontinuously. The closure of the
K\"ahler cone in the positive cone is a fundamental domain for this
action (cf. \cite[Proposition VIII.3.10]{BHPV}).  Hence we can choose
an element $w \in W_S$ such that for a suitable choice of sign $\pm w
\circ \Phi$ is an effective Hodge isometry by
Proposition~\ref{prop:effectiveHodgeisometry} and thus is induced by
an isomorphism $f\colon S \to S'$ by Theorem~\ref{thm:strongtorelli}.
\end{proof}

\subsection{The universal family of marked $\Kthree$ surfaces}\label{subsect:kuranishi}

For each $\Kthree$ surface $S$ we choose a representative $p\colon \cS
\to U$ of the Kuranishi family with $U$ contractible and sufficiently
small such that the following hold:
\begin{enumerate}
\item[\rm(i)] $p\colon  \cS \to U$  is the Kuranishi family for each
  point $s \in U$.
\item[\rm(ii)] If $\phi\colon  R^2 p_*\ZZ_S \to (L_{\Kthree})_U$ is a
  marking, then the associated period map $\pi\colon  U \to \Omega$ is
  injective.
\end{enumerate}

We consider all marked Kuranishi families, i.e.\ pairs $(p\colon \cS
\to U,\phi)$, having the above properties.  We can glue the various
copies of $U$ by identifying those points where the marked $\Kthree$
surfaces are isomorphic.  This defines a space $M_1$ all of whose
points have neighbourhoods isomorphic to some $U$. Hence $M_1$ is a
$20$-dimensional analytic manifold, but possibly not Hausdorff. It is
also possible to show (cf. \cite[Theorem VIII.10.6]{BHPV}) that one
can glue the Kuranishi families to obtain a global family of $\Kthree$
surfaces over $M_1$.

It turns out that the space $M_1$ is indeed not Hausdorff. This
follows from an example due to Atiyah, also referred to as Atiyah's
flop and nowadays crucial in higher-dimensional birational geometry.
Consider, for example, the following $1$-parameter family $S_t$ of
quartic (and hence $\Kthree$) surfaces in $\PP^3$, which is given in
affine coordinates by
\begin{equation*}
x^2(x^2-2) + y^2(y^2-2) + z^2(z^2-2) = 2t^2.
\end{equation*}
Let the parameter $t$ vary in a small neighbourhood $B$ of the
origin.  For $t \neq 0$ these surfaces are smooth, whereas the surface
$S_0$ has an ordinary double point at $(0,0,0)$. This is also the only
singularity, again an ordinary double point, of the total space $\cS$.
Blowing up this node one obtains a smooth $3$-fold $\Tilde{\cS}$
which contains a quadric $\PP^1 \times \PP^1$ as exceptional divisor
$E$.  The proper transform $\hat{S_0}$ of $S_0$ is a smooth $\Kthree$
surface intersecting the exceptional divisor $E$ in a rational curve
of bidegree $(1,1)$.  This is a nodal curve on $\hat{S_0}$. The two
rulings on $E$ can each be contracted giving rise to smooth
$3$-dimensional spaces $p_1\colon \cS_1 \to B$ and $p_2\colon \cS_2
\to B$.  These families are by construction identical over $B
\setminus 0$. They are, however, not identical over all of $B$, since
the identity on $\cS \setminus S_0$ would, otherwise, have to extend
to an automorphism of the total space acting non-trivially on the
tangent cone of the double point, which is clearly impossible. Now
choose a marking for $p_1$. This also defines a marking for $p_2$.
Since the families coincide outside the origin, the period maps
$\pi_1$ and $\pi_2$ also coincide away from $0$.  However, the
markings differ for the central fibre (namely by the Picard-Lefschetz
reflection defined by the nodal curve on $\hat{S_0})$. This shows that
$M_1$ cannot be Hausdorff.  In fact all non-Hausdorff behaviour comes
from the existence of different resolutions of double points in
families: see \cite[Theorem 1${}^\prime$ and Section 7]{BR}.

There is another formulation of the Torelli theorem, which we will now
describe.  For this we consider
\begin{equation*}
K\Omega=\{ (\kappa,[\omega]) \in (L_{\Kthree} \otimes\RR) \times
\Omega \mid (\kappa,\re(\omega))=(\kappa,\im(\omega))=0, \,
(\kappa,\kappa) > 0 \}
\end{equation*}
and define $E(\kappa,\omega)$ as the oriented $3$-dimensional space
spanned by the ordered basis $\{\kappa, \re(\omega),\im(\omega)\}$.
By mapping each point $(\kappa,[\omega])$ to the space
$E(\kappa,\omega)$ we obtain a fibration
\begin{equation*}
\Pi\colon  K\Omega \to \Gr^+(3,L_{\Kthree} \otimes\RR)
\end{equation*}
over the Grassmannian $\Gr^+(3,L_{\Kthree} \otimes\RR)$ of oriented
$3$-planes in $L_{\Kthree} \otimes\RR $ on which the form
$(\quad,\quad)$ is positive definite. This is an $\SO(3,\RR)$-fibre
bundle and the projection $\Pi$ is equivariant with respect to the
action of the orthogonal group $\Aut(L_{\Kthree} \otimes\RR )\cong
\Orth(3,19)$. We define
\begin{equation*}
(K\Omega)^0:=\{(\kappa,[\omega]) \in K\Omega \mid (\kappa,d) \neq 0 \mbox{ for } d \in L_{\Kthree},
(d,d)=-2, (\omega,d)=0 \}.
\end{equation*}
This is an open subset of $K\Omega$.

For every point $\omega \in \Omega$ we consider the cone
\begin{equation*}
C_{\omega} =\{x \in L_{\Kthree} \otimes\RR \mid (x,\omega)=0, \, (x,x)>0 \}.
\end{equation*}
This has two connected components and varies differentiably with
$\omega$. Since $\Omega$ is connected and simply connected, we can
globally choose one of these components, say $C_{\omega}^+$. Let
$(S,\phi)$ be a marked $\Kthree$ surface and let $\kappa \in
H^{1,1}(S,\RR)$ be a K\"ahler class. Then we say that $(S,\kappa)$, or
more precisely $((S,\phi),\kappa)$, is a \emph{marked pair} if
$\phi_{\CC}(\kappa) \in C_{\omega}^+$, where $\omega$ is the period
point defined by $(S,\phi)$.

Let $M_2'$ be the real analytic vector bundle with fibre
$H^{1,1}(S_t)$ over the base $M_1$ of the universal family of marked
$\Kthree$ surfaces and let $M_2 \subset M_2'$ be the subset of
K\"ahler classes. This is open by \cite[Theorem 15]{KS}. In
particular, $M_2$ is real analytic of dimension $60=40+20$, where $40$
is the real dimension of the base $M_1$ and $20$ is the dimension of
the fibre.  We can now define a real-analytic map
\begin{equation*}
\pi_2\colon  M_2 \to (K\Omega)^0
\end{equation*}
by mapping $\kappa \in H^{1,1}(S_t), t \in M_1$ to $\pi_2(\kappa)=
(\phi_{\CC}(\kappa), \pi(t))$.  This is called the \emph{refined period
  map}.  In this way we obtain the obvious commutative diagram
\begin{equation*}
\begin{CD}
M_2  @> \pi_2 >>
{(K\Omega)}^0  \\
@VV{}V   @VV{}V\\\
M_1 @> \pi_1 >>  \Omega.
\end{CD}
\end{equation*}

The Torelli theorem can now be reformulated as follows.
\begin{theorem}\label{thm:refinedtorelli}
The map $\pi_2$ is injective (and thus $M_2$ is Hausdorff).
\end{theorem}

Using this one can finally prove that the period map is surjective.
\begin{theorem}\label{thm:surjectivitytorelli}\emph{(Surjectivity of the period map)}
The refined period map $\pi_2$ is surjective. In particular, every
point of $\Omega$ appears as the period point of some marked $\Kthree$
surface.
\end{theorem}
\begin{proof}
This is proven in \cite[Section VIII.14]{BHPV}.
\end{proof}
The surjectivity of the period map was one of the major questions in
the theory of $\Kthree$ surfaces.  A.~Todorov \cite{Tod} was the first
to give a proof; the argument given in \cite{BHPV} is due to Looijenga.

In view of Theorem~\ref{thm:weaktorelli} and
Theorem~\ref{thm:surjectivitytorelli} we now have
\begin{theorem}\label{thm:moduliunpolarised}
The set $\Orth(L_{\Kthree}) \backslash \Omega$ is in $1:1$
correspondence with the set of isomorphism classes of $\Kthree$
surfaces.
\end{theorem}
We can thus think of $\Orth(L_{\Kthree}) \backslash \Omega$ as the
``moduli space of $\Kthree$ surfaces". It must be pointed out,
however, that the action of the group $\Orth(L_{\Kthree})$ is not well
behaved. In particular, it is not properly discontinuous. We shall now
turn to the case of polarised $\Kthree$ surfaces where we shall see
that the situation is much better.

\subsection{Moduli spaces of polarised $\Kthree$
  surfaces}\label{subsect:modulipolK3}

A \emph{polarisation} on a $\Kthree$ surface is an ample line bundle
$\cL$. Since the irregularity of $\Kthree$ surfaces is $0$ and the
Picard group has no torsion we can identify a line bundle $\cL$ with
its first Chern class $h=c_1(\cL) \in H^2(S,\ZZ)$.  An ample line
bundle $\cL$ is nef and big, and conversely a nef and big line bundle
on a $\Kthree$ surface is ample if there are no $(-2)$-curves $C$ with
$h.C=0$. This can be seen from the description of the K\"ahler cone
\eqref{kaehlercone} and Nakai's criterion, or from Reider's theorem
\cite[Theorem IV.11.4]{BHPV} or from \cite{S-D}.  Throughout, we shall
only consider primitive polarisations, i.e.\ we shall assume that the
class $h$ is non-divisible in the $\Kthree$ lattice.  The
\emph{degree} of a polarisation is $\deg(\cL)=h^2=2d$. Note that the
degree is always even.

We denote by $\latt{-2d}$ the rank $1$ lattice whose generator has
square $-2d$.

\begin{lemma}\label{lem:L2d}
Suppose $h\in L_{\Kthree}$ is a primitive vector with $h^2=2d>0$. Then
the orthogonal complement $L_h=h^{\perp}_{L_{\Kthree}}$ of $h$ is
isometric to~$L_{2d}$, where $L_{2d}$, the lattice $L_{2d}$ associated
with $\Kthree$ surfaces with a polarisation of degree $2d$, is defined
by
\begin{equation}\label{defineL2d}
L_{2d} = 2U \oplus 2E_8(-1) \oplus \latt{-2d}.
\end{equation}
\end{lemma}
\begin{proof}
It follows from Eichler's criterion (see Lemma~\ref{lem:eichler}
and Example~\ref{ex:K3} below) that there is a unique
$\Orth(L_{\Kthree})$-orbit of primitive vectors $h$ of given degree in
the $\Kthree$ lattice $ L_{\Kthree}=3U \oplus 2E_8(-1)$. Hence we can
assume that $h$ is in one of the copies of the hyperbolic plane and
that $h=e+df$ where $e,\ f$ are a basis of a copy of $U$ with
$e^2=f^2=0$ and $(e,f)=1$. The structure of $L_h$ is clear in this case.
\end{proof}

The lattice $L_{2d}$ is an even lattice of signature $(2,19)$. The
period domain $\Omega_{2d}=\Omega_{L_{2d}}$ has two connected
components, $\cD_{2d}$ and $\cD'_{2d}$, interchanged by complex
conjugation. The domain $\cD_{2d}$ is a $19$-dimensional symmetric
homogeneous domain of type~IV: see~\cite[Appendix~6]{Sat}. One can also
describe $\Omega_{2d}$ as the intersection of the domain $\Omega$ with
the hyperplane $h^{\perp}_{L_{\Kthree}}$.

We shall fix $h\in L_{\Kthree}$ once and for all.  For each polarised
$\Kthree$ surface $(S,\cL)$ of degree $2d$ we can consider polarised
markings, i.e.\ markings $\phi\colon H^2(S,\ZZ) \to L_{\Kthree}$ with
$\phi(c_1(\cL))=h$.  Any two such markings differ by an element in the
group
\begin{equation}\label{defineOLh}
\Orth(L_{\Kthree},h) = \{g \in \Orth(L_{\Kthree}) \mid g(h)=h \}.
\end{equation}
This group leaves the orthogonal complement $L_{2d}$ of $h$ invariant
and hence is a subgroup of $\Orth(L_{2d})$.

For any lattice $L$ we denote by $L^\vee=\Hom (L, \ZZ)$ its dual
lattice. The \emph{discriminant group} $D(L)=L^\vee/L$ is a finite
abelian group of order $|\det L|$ and carries a discriminant quadratic
form $q_L$ (if $L$ is even) and a discriminant bilinear form $b_L$,
with values in $\QQ/2\ZZ$ and $\QQ/\ZZ$ respectively (see \cite
[Section~1.3]{Nik2}). The \emph{stable orthogonal group} of an even
lattice $L$ is defined as the kernel
\begin{equation}\label{definestablegroup}
\Tilde{\Orth}(L) = \ker (\Orth(L) \to \Orth(D(L)).
\end{equation}
If $L$ is indefinite, the subgroup $\Orth^+(L)$ is defined to be the
group of elements of real spinor norm $1$ (see \cite{GHScomm} for a
definition of the real spinor norm). We define
\begin{equation}\label{defineOtilde+}
\Tilde{\Orth}^+(L)=\Tilde{\Orth}(L) \cap\Orth^+(L)
\end{equation}
and generally, for any subgroup $G<\Orth(L)$, we use $G^+$ to denote
the kernel of the real spinor norm and $\Tilde G$ to denote the stable
subgroup, the kernel of $G\to \Orth(D(L))$.

For $h\in L_{\Kthree}$ with $h^2=2d$, it follows from Nikulin's theory
\cite[Corollary 1.5.2]{Nik2} that
\begin{equation*}
\Orth(L_{\Kthree},h) = \Tilde{\Orth}(L_{2d})
\end{equation*}
considered as subgroups of $\Orth(L_{2d})$.
The two connected components $\cD_{2d}$ and $\cD_{2d}'$ are
interchanged by the group $\Orth(L_{2d})$. The index $2$ subgroup that
fixes the components is $\Orth^+(L_{2d})$. Finally we define
\begin{equation}\label{F2d}
\cF_{2d} = \Tilde{\Orth}(L_{2d}) \backslash \Omega_{2d} =
\Tilde{\Orth}^+(L_{2d}) \backslash \cD_{2d}.
\end{equation}

It is important to note that the situation is much better here than in
the non-polarised case.  The reason lies in the change of signature
which is now $(2,19)$ rather than $(3,19)$. We are thus dealing with
locally symmetric hermitian domains and, as a result, the action of
the group $\Orth(L_{2d})$ on $\Omega_{2d}$ is properly
discontinuous. Hence the quotient space of $\Omega_{2d}$ by any
subgroup of $\Orth(L_{2d})$ of finite index (i.e.\ an \emph{arithmetic
  subgroup}) is a complex space with only finite quotient
singularities (also sometimes called a $V$-manifold). By a famous
result of Baily and Borel \cite{BB} these quotients are
quasi-projective and thus carry a natural structure of an algebraic
variety. We shall discuss various compactifications of these
quasi-projective varieties below in Section~\ref{sect:compactifications}.

In order to describe the moduli space of polarised $\Kthree$ surfaces
we need one more ingredient.  For $h\in L_{\Kthree}$ we define
\begin{equation*}
\Delta_h= \{ \delta\in L_{\Kthree} \mid \delta^2=-2, \,  (\delta,h)=0 \}.
\end{equation*}
For each $\delta \in \Delta_h$ we define the hyperplane
$H_\delta=\delta^{\perp}_{L_{\Kthree}}$, i.e.\ the hyperplane fixed by the
Picard-Lefschetz reflection defined by $\delta$. We set
\begin{equation*}
\Omega_{2d}^0 = \Omega_{2d} \setminus \bigcup_{\delta \in \Delta_h}
\big(H_\delta \cap \Omega_{2d}\big).
\end{equation*}
There are only finitely many $\Tilde{\Orth}(L_{2d})$-orbits in
$\Delta_h$. This follows immediately from Lemma~\ref{lem:eichler},
below: in fact there are at most two orbits by \cite[Proposition
  2.4(ii)]{GHSprop}.  Since the group acts properly discontinuously on
$\Omega$, the hyperplanes $H_\delta$ for $\delta\in \Delta_h$ form a
locally finite collection.  Clearly, the action of the group
$\Tilde{\Orth}(L_{2d})$ on $\Omega_{2d}$ restricts to an action on
$\Omega_{2d}^0$.  We define
\begin{equation*}
\cF_{2d}^0 = \Tilde{\Orth}(L_{2d}) \backslash \Omega_{2d}^0.
\end{equation*}
Note that this is again a quasi-projective variety (it arises from
$\cF_{2d}$ by removing finitely many divisors) with only finite
quotient singularities.

\begin{theorem}\label{thm:modulikthree}
The variety $\cF_{2d}^0$ is the moduli space of polarised $\Kthree$
surfaces of degree $2d$, i.e.\ its points are in $1:1$ correspondence
with polarised $\Kthree$ surfaces of degree $2d$.
\end{theorem}
\begin{proof}
Let $(S,\cL)$ be a polarised $\Kthree$ surface with $\deg(\cL)=2d$. We
consider polarised markings $\phi\colon H^2(S,\ZZ) \to L_{\Kthree}$
with $\phi(c_1(\cL))=h$. Since $(\omega_S,c_1(\cL))=0$ we find that
$\phi(\omega_S) \in \Omega_{2d}$. In fact, since $\cL$ is ample it has
positive degree on all $(-2)$-curves and hence $\phi(\omega_S)$ lies
in $\Omega_{2d}^0$. Any two polarised markings differ by an element in
$\Orth(L_{\Kthree},h)= \Tilde{\Orth}(L_{2d})$ and hence we obtain a
well-defined map which associates to an isomorphism class $(S,\cL)$ a
point in $\cF_{2d}^0$. This map is injective: assume that a point in
$\cF_{2d}^0$ arises from two polarised surfaces $(S,\cL)$ and
$(S',\cL')$. Then there exists a Hodge iso\-me\-try $H^2(S',\ZZ) \to
H^2(S,\ZZ)$ mapping the ample class $c_1(\cL')$ to $c_1(\cL)$. It then
follows from the strong Torelli Theorem (Theorem~\ref{thm:strongtorelli})
that this map is induced by an isomorphism $(S,\cL) \cong
(S',\cL')$. Thus we get an injective map from the set of isomorphism
classes of degree $2d$ polarised $\Kthree$ surfaces into
$\cF_{2d}^0$. Finally, the surjectivity of this map follows from the
surjectivity of the period map, Theorem~\ref{thm:surjectivitytorelli}.
\end{proof}

In the literature one often finds references to $\cF_{2d}$ as the
moduli space of polarised $\Kthree$ surfaces.  One can interpret the
points in the complement of $\cF_{2d}^0$ as weakly- or semi-polarised
$\Kthree$ surfaces, i.e.\ $\cL$ has positive degree and is nef, but
not ample, as it has degree $0$ on some nodal class(es).
Alternatively, one can consider ample line bundles on
$\Kthree$ surfaces with rational double points.  There is still a
version of the strong Torelli theorem for weakly polarised $\Kthree$
surfaces due to D.~Morrison \cite{Mo}, but the precise formulation is
subtle. For purposes of the birational geometry of these spaces, it
obviously does not matter whether one works with $\cF_{2d}$ or its
open part $\cF_{2d}^0$.

The notion of polarised $\Kthree$ surfaces was generalised by
Nikulin~\cite{Nik1}
to that of \emph{lattice-polarised}
$\Kthree$ surfaces: see also Dolgachev's paper \cite{Dol} for a
concise account, in particular in connection with mirror symmetry.  To
describe this, we fix a lattice $M$ of signature $(1,t)$, which we
assume can be embedded primitively into the $\Kthree$-lattice
$L_{\Kthree}$. The cone $V_M= \{ x \in M_{\RR} \mid (x,x)>0 \}$ has
two connected components: we fix one and denote it by $\cC_M$. Let
\begin{equation*}
\Delta_M=\{ d\in M \mid (d,d)=-2 \}
\end{equation*}
and choose a decomposition $\Delta_M = \Delta_M^+ \cup
(-\Delta_M^+)$. Moreover let
\begin{equation*}
\cC_M^+= \{h \in V(M)^+ \cap M \mid (h,d) > 0 \mbox{ for all } d\in
\Delta_M^+ \}.
\end{equation*}
An \emph{$M$-polarised} $\Kthree$ surface is then a pair $(S,j)$ where
$S$ is a $\Kthree$ surface and $j\colon M \emb \Pic(S)$ is a primitive
embedding.  An isomorphism between $M$-polarised $\Kthree$ surfaces is
an isomorphism between the surfaces that commutes with the primitive
embeddings. If $\omega$ is a non-zero $2$-form on We call $(S,j)$
ample (or pseudo-ample), if $j(\cC_M^+)$ contains an ample (or
pseudo-ample) class. The classical case of polarised $\Kthree$
surfaces is the case where $t=0$ and $M=\latt{2d}$.

The theory of moduli of polarised $\Kthree$ surfaces carries over
naturally to lattice polarised $\Kthree$ surfaces.  For this, one has
to consider the domain $\Omega_M= \{ \omega \in \Omega \mid (w,M)=0
\}$ and the group $\Orth(L_{\Kthree},M)$ of orthogonal transformations
of the $\Kthree$ lattice which fixes $M$.  The role of the variety
$\cF_{2d}$ is then taken by the quotient
\begin{equation}\label{latticepol}
\cF_M= \Orth(L_{\Kthree},M) \backslash \Omega_M.
\end{equation}

Lattice polarised $\Kthree$ surfaces play a role in mirror
symmetry. For this we consider \emph{admissible} lattices,
i.e.\ lattices $M$ which admit an embedding $j\colon M \emb L_{\Kthree}
$ such that
\begin{equation*}
M^{\perp}_{L_{\Kthree}} \cong U(m) \oplus \Tilde{M}.
\end{equation*}
The choice of such a splitting determines a primitive embedding
$\tilde{\jmath}\colon \Tilde{M} \emb L_{\Kthree}$.  The lattice
$\Tilde{M}$ is hyperbolic: more precisely, its signature is $(1,
18-t)$.  The variety $\cF_{\Tilde{M}}$ is then a mirror family to
$\cF_M$. There are more ingredients to the concept of mirror symmetry
which we will not describe here, such as the Yukawa coupling and the
mirror map.  For details we refer to \cite{Dol}: see also \cite{GN4}, \cite{GN5}.

Finally, we want to comment on the relationship between the
construction of moduli spaces of $\Kthree$ surfaces as quotients of
homogeneous domains and GIT constructions.  By Viehweg's results
\cite{Vi} moduli spaces of polarised varieties with trivial canonical
bundle exist and can be constructed as GIT quotients. For this we
first fix a Hilbert polynomial $P(m)$ of a line bundle on a $\Kthree$
surface.  This is of the form $P(m)=m^2d +2$ where the degree of the
line bundle is $2d$.  Let $\cM_{2d}$ be the GIT moduli space of degree
$2d$ polarised $\Kthree$ surfaces. We want to relate this to
$\cF_{2d}$.

We first note that for any ample line bundle $\cL$ on a $\Kthree$
surface $S$ its third power $\cL^{\otimes 3}$ is very ample. For this,
see \cite{S-D}, but in general one can use Matsusaka's big theorem
(\cite{Mat}, \cite{LM}) to show that there is a positive integer $m_0$
such that $\cL^{\otimes m_0}$ is very ample for all polarised
varieties $(X,\cL)$ with fixed Hilbert polynomial.  Now let $m_0 \ge
3$ be sufficiently big.  Then we have embeddings $f_{|\cL^{\otimes
    m_0}|}\colon S \to \PP^{N-1}$ where $N=h^0(S,\cL^{\otimes
  m_0})=P(m_0)$.  Such an embedding depends on the choice of a basis
of $H^0(S,\cL^{\otimes m_0})$. Let $H$ be an irreducible component of
the Hilbert scheme $\Hilb_P(\PP^{N-1})$ containing at least one point
corresponding to a smooth $\Kthree$ surface $S$.  Let $H_{\sm}$ be the
open part of $H$ parametrising smooth surfaces. Then it is easy to
prove that $H_{\sm}$ is smooth and that every point in $H_{\sm}$
parametrises a $\Kthree$ surface.  There exists a universal family
$\cS_{\sm} \to H_{\sm}$.  The group $\SL(N,\CC)$ acts on $H_{\sm}$ and
every irreducible component of the GIT moduli space of degree $2d$
polarised $\Kthree$ surfaces is of the form $\SL(N,\CC) \backslash
H_{\sm}$. Let $\cM_{2d}'$ be such a component.  Choosing local
polarised markings for the universal family one can construct a map
$H_{\sm} \to \cF^0_{2d}$, which clearly factors through the action of
$\SL(N,\CC) $, i.e.\ gives rise to a map $\pi\colon \cM'_{2d} \to
\cF^0_{2d}$. By construction this is a holomorphic map.  On the other
hand both $\cM'_{2d}$ and $\cF_{2d}$ are quasi-projective varieties.
It then follows from a theorem of Borel \cite{Bl2} that $\pi$ is a
morphism of quasi-projective varieties.

We claim that $\pi$ is an isomorphism and that $\cM_{2d}$ has only one
component.  First of all we note that one can, as in the proof of
\cite[Theorem~1.5]{GHSsymp}, take a finite \'etale cover $H'_{\sm} \to
H_{\sm}$ such that the action of $\SL(N+1,\CC)$ lifts to a free action
on $H'_{\sm}$ as well as on the pullback $\cS'_{\sm} \to H'_{\sm}$ of
the universal family.  This gives a quotient family over
$Z_{\sm}=\SL(N+1,\ZZ) \backslash H'_{\sm}$ which is smooth and maps
finite-to-one to $\cM_{2d}$.  By the local Torelli theorem the natural
map $Z_{\sm} \to \cF^0_{2d}$ has discrete fibres and hence the same
is true for $\pi$. But then $\pi$ must be dominant. Now we can use
Theorem~\ref{thm:modulikthree} to conclude that $\cM_{2d}$ is
irreducible and that $\pi$ is a bijection. Since $\cF^0_{2d}$ is a
normal variety, it also follows that $\pi$ is an isomorphism.  We can
thus summarise our discussion as follows.
\begin{theorem}\label{thm:gitperiods}
There is an isomorphism $\cM_{2d} \cong \cF^0_{2d}$, i.e.\ the GIT
moduli space $\cM_{2d}$ is isomorphic to the modular variety
$\cF^0_{2d}$.
\end{theorem}

\section{Irreducible symplectic manifolds}\label{sect:irrsymplectic}

In this section we recall the main properties of irreducible
symplectic manifolds, discuss the Torelli theorem and give basic facts
about moduli spaces of polarised symplectic manifolds.

\subsection{Basic theory of irreducible symplectic manifolds}\label{subsect:irrsymplecticbasic}

The theory of irreducible symplectic manifolds is less
developed than that of $\Kthree$ surfaces. Nevertheless, several
results have been proved over the last $30$ years.
\begin{definition}\label{def:irrsymplectic}
  A complex manifold $X$ is called an \emph{irreducible symplectic
    manifold} or \emph{hyperk\"ahler manifold} if the following
  conditions are fulfilled:
\begin{itemize}
\item[{\rm (i)}] $X$ is a compact K\"ahler manifold;
\item[{\rm (ii)}] $X$ is simply-connected;
\item[{\rm (iii)}] $H^0(X,\Omega^2_X) \cong \CC \omega$ where $\omega$
  is an everywhere nondegenerate holomorphic $2$-form.
\end{itemize}
\end{definition}

According to the Bogomolov decomposition theorem \cite{Bog2},
irreducible symplectic manifolds are one of the building blocks for
compact K\"ahler manifolds with trivial canonical bundle: see also
\cite[Th\'eor\`eme 2]{Be}. The others are abelian varieties and
Calabi-Yau varieties (here we mean Calabi-Yau in its strictest sense,
i.e.\ a compact K\"ahler manifold $X$ such that $\pi_1(X)=1$ and
$H^0(X, \Omega^i_X)=0$ for $0<i<\dim X$).

In dimension two the only irreducible symplectic manifolds are
$\Kthree$ surfaces. Although irreducible symplectic manifolds have now
been studied for nearly $30$ years, only four classes of such
manifolds have so far been discovered and it is a wide open problem
whether other types exist or not.  The known examples are:

\rm{(i)} The length $n$ Hilbert scheme $S^{[n]}=\Hilb^n(S)$ for a
$\Kthree$ surface $S$, and its deformations.  Note that the
deformation space of such a variety has dimension $21$ if $n\ge 2$ and that, since
$\Kthree$ surfaces only depend on $20$ parameters, a general
deformation will not itself be of the form $S^{[n]}$. We shall refer
to these varieties as irreducible symplectic manifolds of
\emph{deformation $\Kthree^{[n]}$ type} or \emph{deformation
  $\Kthree^{[n]}$ manifolds}.

\rm{(ii)} Let $A$ be a $2$-dimensional complex torus and consider the
length-$(n+1)$ Hilbert scheme $A^{[n+1]}=\Hilb^{n+1}(A)$ together with
the morphism $p\colon A^{[n+1]} \to A$ given by addition. Then
$X=p^{-1}(0)$ is an irreducible symplectic manifold, called a
\emph{generalised Kummer variety} (even though it is not necessarily
algebraic).  The deformation space of these manifolds has dimension
$5$ if $n\ge 2$, again one more than for $2$-dimensional complex tori.

\rm{(iii)} O'Grady's irreducible symplectic manifolds of dimension
$6$, described in~\cite{OG2}. These are deformations of
(desingularised) moduli spaces of sheaves on an abelian surface and
depend on $6$ parameters.

\rm{(iv)} O'Grady's irreducible symplectic manifolds of dimension
$10$, described in~\cite{OG1}. These are deformations of
(desingularised) moduli spaces of sheaves on a $\Kthree$ surface and
depend on $22$ parameters.

Other moduli spaces of sheaves, apart from those considered in
\cite{OG1} and \cite{OG2}, cannot be desingularised symplectically:
see~\cite{KLS} and also~\cite{Zo}.

In many ways irreducible symplectic manifolds behave like
$\Kthree$ surfaces, but there are also important differences, as we
shall see later. We first notice that it follows immediately from the
definition that $X$ must have even dimension $2n$ over $\CC$ and that
its canonical bundle $\omega_X$ is trivial (an $n$-fold exterior power
of a generator $\omega$ of $H^0(X,\Omega^2_X)$ will define a
trivialisation of the canonical bundle).  Clearly
$h^{2,0}(X)=h^{0,2}(X)=1$ and $h^{1,0}(X)=h^{0,1}(X)=0$. By a result
of Bogomolov \cite{Bog}, the deformation space of $X$ is
unobstructed. This result was generalised to Ricci-flat manifolds by
Tian~\cite{Ti} and Todorov~\cite{Tod}, and algebraic proofs were given
by Kawamata~\cite{Kaw} and Ran~\cite{Ran} (see also~\cite{Fuj}).
Since
\begin{equation*}
T_{[0]} \Def(X) \cong H^1(X,T_X) \cong H^1(X, \Omega^1_X)
\end{equation*}
the dimension of the deformation space is $b_2(X)-2$.

As in the $\Kthree$ case we have a Hodge decomposition $H^2(X,\CC) =
H^{2,0} \oplus H^{1,1} \oplus H^{0,2}$ with $H^{2,0}$ and $H^{0,2}$
both $1$-dimensional. Unlike the $\Kthree$ case the intersection form
does not immediately provide $H^2(S,\ZZ)$ with the structure of a
lattice. It was, however, shown by Beauville \cite{Be} that
$H^2(X,\ZZ)$ does carry a natural structure as a lattice.  To define
this, let $\omega \in H^{2,0}(X)$ be such that
$\int_X(\omega\Bar{\omega})^n=1$ and define
\begin{equation*}
q'_X(\alpha)=\frac{n}{2}\int_X\alpha^2(\omega\Bar{\omega})^{n-1} +
(1-n)\left(\int_X\alpha\omega^{n-1}\Bar{\omega}^{n}\right)
\left(\int_X\Bar{\alpha}\omega^{n}\Bar{\omega}^{n-1}\right).
\end{equation*}
After multiplication by a positive constant $\gamma$ the quadratic
form $q_X=\gamma q'_X$ defines an indivisible integral symmetric
bilinear form $(\quad ,\quad )_X$ on $H^2(X,\ZZ)$: this is the
Beauville form. Clearly $(\omega,\omega)_X=0$ and
$(\omega,\overline{\omega})_X >0$.

There is another way to introduce the Beauville form. For this let
$v(\alpha)=\alpha^{2n}$ be given by the cup product. Then, by a result
of Fujiki~\cite[Theorem 4.7]{Fuj}, there is a positive rational number
$c$, the \emph{Fujiki invariant}, such that
\begin{equation*}
v(\alpha)=cq_X(\alpha)^n
\end{equation*}
for all $\alpha \in H^2(X,\ZZ)$. In this sense the Beauville form can
be derived from the cup product of the cohomology.

\begin{proposition}\label{prop:Beauvillelattices}
The Beauville lattices and Fujiki invariants for the known examples of
irreducible symplectic manifolds are as follows:

{\rm(i)} The Beauville lattice of a deformation $\Kthree^{[n]}$
manifold is $L_{\Kthree,2n-2}=3U \oplus 2E_8(-1) \oplus
\latt{-2(n-1)}$. It has rank $23$, one more than the $\Kthree$
lattice, to which it is closely related.  If $X=S^{[n]}$ for a
$\Kthree$ surface $S$, then $3U \oplus 2E_8(-1)$ comes from
$H^2(S,\ZZ)$ and the summand $\latt{-2(n-1)}$ is generated (over
$\QQ$) by the exceptional divisor $E$, which is the blow-up of the
diagonal in the symmetric product $S^{(n)}$.  As an element in the
Picard group the divisor $E$ is $2$-divisible. The Beauville lattice
remains constant under deformations. The Fujiki invariant is
$c=(2n)!/(n!2^n)$.

{\rm(ii)} The Beauville lattice of a generalised Kummer variety (or
deformation thereof) is $3U \oplus \latt{-2(n+1)}$ and the Fujiki
invariant is $c=(n+1)(2n)!/(n!2^n)$.

{\rm(iii)} The Beauville lattice of the $6$-dimensional example of
O'Grady is $3U \oplus \latt{-2} \oplus \latt{-2}$ and the Fujiki
invariant is $c=60$.

{\rm(iv)} The Beauville lattice of the $10$-dimensional example of
O'Grady is $3U \oplus 2E_8(-1) \oplus A_2(-1)$ and the Fujiki
invariant is $c=945$.
\end{proposition}

For proofs, see \cite{Be}, \cite{Rap1} and \cite{Rap2}. Note that all
these lattices are even. It is, however, not known whether this is a
general fact for irreducible symplectic manifolds.

Let $L$ be an abstract lattice which is isomorphic to
$(H^2(X,\ZZ),q_X)$ for some irreducible symplectic manifold $X$.  A
\emph{marking} is an isomorphism of lattices $\phi\colon H^2(X,\ZZ)
\isoto L$.  Let $p\colon \cX \to U$ be a representative of the
Kuranishi family of deformations of $X$ with sufficiently small and
contractible base.  Note that by unobstructedness the base space of
the Kuranishi family is smooth and of dimension $b_2(X)-2$.  The
marking $\phi$ for $X$ defines a marking for $\cX$ and we obtain a
period map $\pi_U\colon U \to \Omega_L$ to the period domain defined
by \eqref{generalperioddomain}. As in the $\Kthree$ case we have a local Torelli theorem.
\begin{theorem}\emph{(Beauville)}\label{thm:symplecticlocaltorelli}
The differential of the period map defined by the Kuranishi family is
an isomorphism and thus the period map is a local isomorphism.
\end{theorem}
\begin{proof}
See \cite{Bog2}, \cite{Be}.
\end{proof}

As in the $\Kthree$ case one can define a moduli space of marked
irreducible symplectic manifolds (of a given type).  Again, this will
not be Hausdorff. Another result which carries over from the $\Kthree$
case is surjectivity of the period map.
\begin{theorem}\emph{(Huybrechts)}\label{thm:surjectivityirreduciblesymplectic}
Let $L$ be a lattice of an irreducible symplectic manifold and let
$\Omega_L$ be the associated period domain.  If $\cM'_L$ is a
non-empty component of the moduli space $\cM_L$ of irreducible
symplectic manifolds with Beauville lattice $L$, then the period map
$\pi\colon \cM'_L \to \Omega_L$ is surjective.
\end{theorem}
\begin{proof}
A proof can be found in \cite[Section 8]{Huy1}.
\end{proof}

\subsection{Hodge theoretic Torelli theorem}\label{subsect:hodgetorelli}

So far, many results from $\Kthree$ surfaces have carried over to
other irreducible symplectic manifolds.  The situation changes when it
comes to the global Torelli theorem. The first counterexample to this
is due to Debarre \cite{Deb}. He showed the following: let $S$ be a
$\Kthree$ surface containing only one curve, which is a smooth
rational curve $C$, and consider the Hilbert scheme $X=S^{[n]}$ with
$n \ge 2$.  Then $X$ contains $S^n(C)\imic\PP^n$.  One can perform an
elementary transformation on $X$ by first blowing up $S^n(C)$ and then
contracting the exceptional divisor in another direction. This gives
another compact complex manifold $X'$ which is bimeromorphic but not
isomorphic to $X$. In general $X'$ need not be K\"ahler, but Debarre
produced an example where $X'$ does have a K\"ahler structure and thus
is an irreducible symplectic manifold.

Since the natural bimeromorphic transformation $f\colon X' \ratmap X$
defines a Hodge isometry $f^*\colon H^2(X,\ZZ) \to H^2(X',\ZZ)$ this
gives a counterexample to the global Torelli theorem.  It should be
noted, though, that neither the surface $S$ nor the varieties $X$ and
$X'$ are projective.  Moreover, the existence of $(-2)$-curves on a
$\Kthree$ surface $S$ is exactly the cause for the failure of the
Hausdorff property for the base of the universal family.

Debarre's counterexample would still allow for a version of the
Torelli theorem where the existence of a Hodge isometry only implies
birational equivalence, not an isomorphism (for $\Kthree$ surfaces the
two notions coincide). However, this is also ruled out by the
following counterexample which is due to Y.~Namikawa \cite{Nam}.  For
this one starts with a generic abelian surface $A$ with a polarisation
of type $(1,3)$.  Then the dual abelian surface $\Hat{A}$ also carries
a $(1,3)$-polarisation. Let $X=\Km^{[2]}A$ and $\Hat{X}=\Km^{[2]}(\Hat
A)$ be the associated generalised Kummer varieties of dimension
$4$. Then $X$ and $\Hat{X}$ are birationally equivalent if and only
if $A$ and $\Hat{A}$ are isomorphic abelian surfaces.  The reason for
this is the following: every birational isomomorphism must send the
exceptional divisor $E$ on $X$ to the exceptional divisor $\Hat{E}$ on
$\Hat{X}$. Since the Albanese of $E$ and $\Hat{E}$ are $A$ and
$\Hat{A}$ respectively, this implies that $A$ and $\Hat{A}$ are
isomorphic. This is not the case for general $A$.  This shows that
Namikawa's example gives a counterexample even to the birational
global Torelli theorem.  Moreover, one can easily make this into a
counterexample to the polarised Torelli theorem.  This can be done by
choosing polarisations of the form $mL- \delta$ and $m\Hat{L} -
\Hat{\delta}$, where $m$ is sufficiently large, $L$ and $\Hat{L}$ are
induced from the polarisation on $A$ and $\Hat{A}$ respectively,
$\delta=2E$ and $\Hat{\delta}=2\Hat{E}$ (the exceptional divisors are
$2$-divisible in the Picard group). The Hodge isomorphism respects
these polarisations.

At first, these counterexamples seem to indicate that there is no
chance of proving a version of the global Torelli theorem for
irreducible symplectic manifolds. However, the above example is not as
surprising as it seems at a first glance. It is well known, and also
well understood in terms of period domains and arithmetic groups, that
$A$ and $\Hat{A}$ are not isomorphic as polarised abelian surfaces
(their period points in the Siegel space are inequivalent under the
paramodular group), but that the associated Kummer surfaces $\Km(A)$
and $\Km(\Hat{A})$ are isomorphic (and their period points in the
corresponding type~IV domain are equivalent under the orthogonal
group). Details can be found in \cite{GH2}.  An analysis of this
situation suggests that a version of the Torelli theorem could hold if
one considers Hodge isometries with extra conditions.

Verbitsky \cite{Ver} has announced a global Torelli theorem for
irreducible symplectic manifolds (see also Huybrecht's Bourbaki talk
\cite{Huy2}).  His results were further elucidated by Markman
\cite{Mar4}.  The crucial idea here is to use monodromy operators and
parallel transport operators.  Markman first noticed the importance of
these operators for the study of irreducible symplectic ma\-ni\-folds,
developing his ideas in a series of papers \cite{Mar1}, \cite{Mar2},
\cite{Mar3}.  To define them, let $X_1, X_2$ be irreducible symplectic
manifolds.  We say that $f\colon H^*(X_1,\ZZ) \to H^*(X_2,\ZZ)$ is a
\emph{parallel transport operator} if there exists a smooth, proper
flat family $\pi\colon \cX \to B$ of irreducible symplectic manifolds
together with points $b_1,\ b_2\in B$ such that there are isomorphisms
$\alpha_i\colon X_i \isoto \cX_{b_i}$ and a continuous path $\gamma
\colon [0,1] \to B$ with $\gamma(0)=b_1$ and $\gamma(1)=b_2$, such
that the parallel transport in the local system $R\pi_*\ZZ$ along
$\gamma$ induces the isomorphism $(\alpha_2^{-1})^* \circ f \circ
\alpha_1^* \colon H^*(\cX_{b_1},\ZZ)\to H^*(\cX_{b_2},\ZZ)$.

For a single irreducible symplectic manifold $X$, we call an
automorphism $f\colon H^*(X,\ZZ) \to H^*(X,\ZZ)$ a \emph{monodromy
  operator} if it is a parallel transport operator (with $X_1=X_2=X$).
The \emph{monodromy group} $\Mon(X)$ is defined as the subgroup of
$\GL(H^*(X,\ZZ))$ generated by monodromy operators. Restricting the
group action to the second cohomology group we obtain a subgroup
$\Mon^2(X)$ of $\GL(H^2(X,\ZZ))$.  Since monodromy operators preserve
the Beauville form we obtain a subgroup $\Mon^2(X) \subset
\Orth(H^2(X,\ZZ))$.

Based on Verbitsky's results \cite{Ver}, Markman \cite{Mar4} has
formulated the following Hodge theoretic global Torelli theorem.
\begin{theorem}
\label{thm:Hodgetorelli}\emph{(Hodge theoretic Torelli)}
Suppose that $X$ and $Y$ are irreducible symplectic manifolds.
\begin{enumerate}
\item[\rm(i)] If $f \colon H^2(Y,\ZZ) \to H^2(X,\ZZ)$ is an
  isomorphism of integral Hodge structures which is a parallel
  transport operator, then $X$ and $Y$ are bimeromorphic.
\item[\rm(ii)] If, moreover, $f$ maps a K\"ahler class of $Y$ to a
  K\"ahler class of $X$, then $X$ and $Y$ are isomorphic.
\end{enumerate}
\end{theorem}
\begin{proof}
The first part of the theorem follows easily from Verbitsky's
results. The second part uses in addition results on the K\"ahler cone
of irreducible symplectic manifolds.  For a more detailed discussion
see \cite[Theorem~1.3]{Mar4} and \cite[Section~3.2]{Mar4}.
\end{proof}

\subsection{Moduli spaces of polarised irreducible symplectic manifolds}\label{subsect:modulipolsymp}

We shall now turn to the case of polarised irreducible symplectic
manifolds. In the $\Kthree$ case we saw that the degree is the only
invariant of a polarisation, or, equivalently, that there is only one
$\Orth(L_{\Kthree})$-orbit of primitive vectors of given length.  This
is no longer true in general, as can already be seen in the case of
$S^{[2]}$. Recall that the Beauville lattice in this case is
isomorphic to $L_{\Kthree,2} =3U \oplus 2E_8(-1) \oplus \latt{-2}$.
If $h$ is a primitive vector the number $\div (h)$ (the divisor
of $h$: see Equation~\eqref{def:divv} below) is the positive generator
of the ideal $(h,L_{\Kthree,2})$, which is the biggest positive
integer by which one can divide $h$ as a vector in the dual lattice
$L_{\Kthree,2}^{\vee}$. Since $L_{\Kthree,2}$ is not unimodular, but
has determinant $2$, the divisor $\div(h)$ can be $1$ or $2$. Indeed,
both of these happen when $d \equiv -1 \mod 4$ and accordingly we have
one $\Orth(L_{\Kthree,2})$-orbit if $d \not\equiv -1 \mod 4$ and two
if $d \equiv -1 \mod 4$.  Details of this can be found in
\cite{GHSsymp}.  These two cases are referred to as the \emph{split}
case ($\div(h)=1$) and the \emph{non-split} case ($\div(h)=2$). The
reason for this terminology is the behaviour of the orthogonal
lattice: if $h^2=2d$ the possibilities for
$L_h=h^\perp_{L_{\Kthree,2}}$ are (see Example~\ref{ex:K3n} below)
\begin{equation}\label{splitlattice}
L_h=2U \oplus 2E_8(-1) \oplus \latt{-2} \oplus \latt{-2d}
\mbox{ for } \div(h)=1
\end{equation}
and
\begin{equation}\label{nonsplitlattice}
L_h= 2U \oplus 2E_8(-1) \oplus \begin{pmatrix}
-2 & 1\\
1 &-\frac{d+1}{2}
\end{pmatrix} \mbox{ for } \div(h)=2.
\end{equation}
For the higher dimensional case of $S^{[n]}$ the situation becomes
more involved as the possibilities for the divisor of $h$ increase in
number.  Whenever the primitive vectors of length $d$ form more than
one orbit, the moduli space of polarised irreducible symplectic
manifolds of degree $d$ will not be connected: see \cite{GHSsymp}.

In order to discuss moduli spaces of polarised irreducible symplectic
mani\-folds we first fix some discrete data: the dimension $2n$, the
Beauville lattice (considered as an abstract lattice $L$), and the
Fujiki invariant $c$. Together $L$ and $c$ define the \emph{numerical
  type} of the irreducible symplectic manifold, denoted by $\bdN$.
Next we choose a \emph{polarisation type}, i.e.\ an $\Orth(L)$-orbit
of a primitive vector $h \in L$. Viehweg's theory gives us the
existence of a moduli space $\cM_{n,\bdN,h}$ parametrising polarised
irreducible symplectic manifolds $(X,\cL)$ of dimension $2n$ with the
chosen Beauville lattice, Fujiki invariant and polarisation type.
This is a quasi-projective variety and can be constructed as a GIT
quotient as in the $\Kthree$ case, the only difference being that we
must here invoke Matsusaka's big theorem~\cite{Mat} and a result of
Koll\'ar and Matsusaka \cite{KM} to be guaranteed a uniform bound
$N_0$ such that $\cL^{\otimes N_0}$ is very ample for all pairs
$(X,\cL)$.

Although there is not a Torelli theorem as in the $\Kthree$ case,
these moduli spaces are still related to quotients of homogeneous
domains of type~IV.  Let $\Omega_L$ be the period domain defined by
the lattice $L$ and let $L_h=h^{\perp}_L$. This is a lattice of
signature $(2,\rk(L)-3)$ and defines a homogeneous domain
$\Omega_{L_h}$ of type IV. Let $\Orth(L,h)$ be the stabiliser of $h$
in $\Orth(L)$. This can be considered as a subgroup of
$\Orth(L_h)$. The domain $\Omega_{L_h}$ has two connected components,
of which we choose one, which we denote by $\cD_{L_h}$. Again
$\Orth^+(L,h)$, the subgroup of $\Orth(L,h)$ of real spinor norm
$1$, is the subgroup fixing the components of $\Omega_{L_h}$.

\begin{theorem}\label{thm:modulisymplectic}
For every component $\cM'_{n,\bdN,h}$ of the moduli space
$\cM_{n,\bdN,h}$ there exists a finite to one dominant morphism
\begin{equation*}
\psi\colon  \cM'_{n,\bdN,h} \to \Orth^+(L,h) \backslash \cD_{L_h}.
\end{equation*}
\end{theorem}
The proof of this is analogous to the proof of
Theorem~\ref{thm:gitperiods}. There are, however, differences compared
to the $\Kthree$ case.  In general, $\psi$ will not be injective (see
the discussion below).  There is also a difference concerning the
image of~$\psi$.  In the $\Kthree$ case we know that a big and nef
line bundle is ample if and only if it has positive degree on the
nodal curves.  Hence it is necessary and sufficient to remove the
hyperplanes orthogonal to the nodal classes. So far, no complete
analogue is known for the higher dimensional case, but Hassett and
Tschinkel have proved partial results for $n=2$ in~\cite{HT1}, and
more precise results in a special case in~\cite{HT2}.

Nevertheless, Theorem~\ref{thm:modulisymplectic} is enough to prove
results on the Kodaira dimension of moduli spaces of polarised
irreducible symplectic manifolds. This was done in \cite{GHSsymp},
\cite{GHSdim21}: see Theorem~\ref{thm:splitgt} and Theorem~\ref{thm:21gt}
below.

Very recent work of Verbitsky \cite{Ver} and Markman \cite{Mar4}
improves Theorem~\ref{thm:modulisymplectic}, using a polarised
analogue of the monodromy group $\Mon^2(X) \subset \Orth(H^2(X),\ZZ)$.
Let $H$ be an ample divisor on $X$.  We call an element in $f \in
\Mon(X)$ a \emph{polarised parallel transport operator of the pair
  $(X,H)$} if it is a parallel transport operator for a family
$\pi\colon \cX \to B$ with base point $b_0 \in B$ and isomorphism
$\alpha\colon X \to \cX_{b_0}$ which fixes $c_1(H)$, such that there
exists a flat section $h$ of $R^2\pi_* \ZZ$ with
$h(b_0)=\alpha_*(c_1(H))$ and $h(b)$ an ample class in
$H^2(\cX_b,\ZZ)$ for all $b \in B$. These operators define a subgroup
$\Mon^2(X,H) \subset \Orth(H^2(X),\ZZ)$. It was shown by Markman
\cite[Proposition~1.9]{Mar4} that $\Mon^2(X,H)$ is the stabiliser of
$c_1(H)$ in $\Mon^2(X)$. Given a marking $\phi\colon H^2(X,\ZZ) \to L$
this defines a subgroup
\begin{equation}\label{monodromyimage}
\Gamma =\phi(\Mon^2(X,H))\subset \Orth(L,h),
\end{equation}
which can be shown to be independent of the marking $\phi$: see
\cite[Section~7.1]{Mar4}.

Let $\cM'_{n,\bdN,h}$ be a component of the moduli space of polarised
irreducible symplectic manifolds with fixed numerical type and
polarisation type. Given an element $(X,H)$ in this component one thus
obtains a group $\Gamma \subset \Orth(L,h)$ as above, which is
also independent of the chosen pair $(X,H)$ by the results of
\cite[Section~7.1]{Mar4}. In fact $\Gamma \subset \Orth^+(L,h)$, as
monodromy operators are obviously orientation-preserving: see \cite[Section~1.2]{Mar4}.
Thus $\Gamma$ acts on the homogeneous domain $\cD_{L_h}$.

\begin{theorem}\label{thm:moduliimmersion}
The map $\psi$ from Theorem~\ref{thm:modulisymplectic} lifts to an
open immersion
\begin{equation*}
\Tilde\psi\colon  \cM'_{n,\bdN,h} \to \Gamma \backslash \cD_{L_h},
\end{equation*}
where $\Gamma$ is as in Equation~\eqref{monodromyimage}.
\end{theorem}
\begin{proof}
It is easy to see that the map $\psi\colon \cM'_{n,\bdN,h} \to
\Orth^+(L,h) \backslash \cD_{L_h}$ lifts to a map $\Tilde\psi\colon
\cM'_{n,\bdN,h} \to \Gamma \backslash \cD_{L_h}$ (see the
beginning of the proof of \cite[Theorem 2.3]{GHSsymp}). The hard part
is the injectivity of $\Tilde\psi$ and this is where the Torelli
theorem for irreducible symplectic manifolds is used. For details we
refer the reader to \cite{Mar4}.
\end{proof}
\begin{remark}\label{rem:improvement}
We note that in general the projective group $\Gamma/\pm 1$ is a
proper subgroup of $\Orth^+(L,h)/\pm 1$ and thus
Theorem~\ref{thm:moduliimmersion} is a substantial improvement of
Theorem~\ref{thm:modulisymplectic}.
\end{remark}

In the case of irreducible symplectic manifolds of $\Kthree^{[n]}$
type this can be made explicit.  For an even lattice $L$ we define
$\Ref(L)$ to be the subgroup generated by $-2$-reflections and the
negative of $+2$-reflections. This is a subgroup of $\Orth^+(L)$.  If
$X$ is an irreducible symplectic manifold of $\Kthree^{[n]}$ type,
then we define $\Ref(X)$ accordingly.

\begin{theorem}\label{thm:monodromygroup}\emph{(Markman)}
If $X$ is a deformation $\Kthree^{[n]}$ manifold then
\begin{equation*}
\Mon^2(X) = \Ref(X).
\end{equation*}
\end{theorem}
\begin{proof}
This is proved in \cite[Theorem 1.2]{Mar2}.
\end{proof}

For a lattice $L$ we let
\begin{equation*}
\Hat\Orth(L) = \{ g \in \Orth(L) \mid g|_{L^{\vee}/L}= \pm \id_{L^{\vee}/L} \}
\end{equation*}
and given an element $h \in L$ we set $\Hat\Orth(L,h) = \{ g \in
\Hat\Orth(L) \mid g(h)=h \}$. Recall the convention of
Equation~\eqref{defineOtilde+} and that by
Proposition~\ref{prop:Beauvillelattices}(i) the Beauville lattice is
$L_{\Kthree,2n-2}$.

It then follows from Theorem~\ref{thm:monodromygroup} in conjunction with
Kneser's result \cite[Satz 4]{Kn} that
\begin{equation}\label{reflectiongroup}
\Ref(L_{\Kthree,2n-2}) = \Hat\Orth^+(L_{\Kthree,2n-2}).
\end{equation}

Combining this with Theorem~\ref{thm:moduliimmersion} we thus obtain

\begin{theorem}\label{thm:cover}
  Let $\cM'_{h}$ be an irreducible component of the moduli space of
  polarised deformation $\Kthree^{[n]}$ manifolds. Then the map $\psi$
  of Theorem~\ref{thm:modulisymplectic} factors through the finite
  cover $\Hat\Orth^+(L_{\Kthree,2n-2},h)\backslash\cD_{L_h}\to
  \Orth^+(L_{\Kthree,2n-2},h)\backslash\cD_{L_h}$ that is, there is a
  commutative diagram
\begin{equation*}
\xymatrix{
{\cM'_{h}} \ar[r]^(.3){\Tilde\psi} \ar[dr]^{\psi}
& {\Hat\Orth^+(L_{\Kthree,2n-2},h) \backslash \cD_{L_h}}
\ar[d]\\
& {\Orth^+(L_{\Kthree,2n-2},h) \backslash \cD_{L_h}.}
}
\end{equation*}
Moreover the map $\Tilde\psi$ is an open immersion.
\end{theorem}
\begin{remark}\label{rem:correction}
  In \cite[Proposition 2.3]{GHSsymp} we stated that the map $\psi$
  lifts to the quotient ${\Tilde\Orth^+(L_{\Kthree,2n-2},h) \backslash
    \cD_{L_h}}$. This is not correct since, contrary to what was said
  in the proof, the projective groups
  $\Tilde\Orth^+(L_{\Kthree,2n-2},h)/\pm 1$ and
  $\Hat\Orth^+(L_{\Kthree,2n-2},h)/\pm 1$ are not identical if $n>2$.
  In that case, in fact, $\Tilde\Orth^+(L_{\Kthree,2n-2},h)/\pm 1$ is
  an index $2$ subgroup of $\Hat\Orth^+(L_{\Kthree,2n-2},h)/\pm 1$.
  If, however, $n=2$, then the two groups coincide since
  $\Tilde\Orth^+(L_{\Kthree,2})= \Hat\Orth^+(L_{\Kthree,2})$ and thus
  the results of \cite{GHSsymp} are not affected by this error.
\end{remark}
\begin{remark}\label{rem:answersquestion}
Theorem~\ref{thm:cover} gives an affirmative answer to
\cite[Question~2.6]{GHSsymp} (with the correct group).
\end{remark}
\begin{remark}\label{rem:connectedness}
The results discussed so far do not give an answer to the question
whether moduli spaces of polarised irreducible symplectic manifolds of
given deformation type and given type of polarisation are always
connected. Apostolov \cite{Ap} has obtained some results on this. For
example in the $\Hilb^{[n]}$ case these moduli spaces are always
connected for $n=2$ (both in the split and the non-split case), but
in general there can be more than one component.
\end{remark}

\section{Projective models}\label{sect:projectivemodels}

Besides the abstract theory of moduli spaces there is a vast
literature which deals with concrete geometric descriptions of
$\Kthree$ surfaces, and to a much lesser degree, also of irreducible
symplectic manifolds of higher dimension. The easiest example is
degree $4$ surfaces in $\PP^3$. Any smooth quartic surface is a
$\Kthree$ surface and counting parameters one obtains a family of
dimension $34 - 15= 19$, because $34$ is the number of quartics and $15$
is the dimension of $\PGL(4,\CC)$. This argument shows that the moduli
space $\cF_4$ of polarised $\Kthree$ surfaces of degree $4$ is
unirational. The same approach yields unirationality for degrees
$2d=2$, $6$ and $8$, as these correspond to double covers of the projective
plane branched along a sextic curve, complete intersections of type
$(2,3)$ in $\PP^4$, and complete intersections of type $(2,2,2)$ in
$\PP^5$ respectively.

In general it can be very hard to decide whether a moduli space of
polarised $\Kthree$ surfaces of low degree is unirational or
not. Mukai (\cite{Mu1}, \cite{Mu2}, \cite{Mu3}, \cite{Mu4}, \cite{Mu5})
has contributed most significantly to this problem.  So far there are
three approaches to proving unirationality.
\begin{itemize}
\item[\rm{(1)}] Describing the $\Kthree$ surfaces as complete
  intersections in homogeneous spaces (this can be used for $1 \leq d
  \leq 9$, $d=11,\ 12,\ 17,\ 19$).
\item[\rm{(2)}] Using non-abelian Brill-Noether theory of vector
  bundles over algebraic curves (here one obtains results for
  $d=6,\ 8,\ 10,\ 16$).
\item[\rm{(3)}] Using specific geometric constructions for certain
  degrees ($d=11$, $12$, $15$, $19$). An example is Mukai's most recent
  work (\cite{Mu5}) for $d=15$ where he describes $\Kthree$ surfaces
  of degree $30$ as complete intersections in a certain rank $10$
  vector bundle on the Ellingsrud-Piene-Str{\o}mme moduli space of
  twisted cubics.
\end{itemize}

There are only a few results about rationality for these
cases. Shepherd-Barron proves rationality for the cases $d=3$ in
\cite{S-B2} and $d=9$ in \cite{S-B1}. 

For a discussion of low degree cases we also refer the reader to
Voisin's Bourbaki expos\'e~\cite{Vo2}. One can summarise the results
as follows.
\begin{theorem}\label{thm:lowdegree}
The moduli spaces $\cF_{2d}$ of polarised $\Kthree$ surfaces of degree
$2d$ are unirational for $1 \leq d \leq 12$ and $d=15,\ 16,\ 17,\ 19$.
\end{theorem}

Some other special moduli spaces related to $\Kthree$ surfaces have
also been studied. Perhaps the most notable of these is the moduli
space of Enriques surfaces, which is a special case of the lattice
polarised $\Kthree$ moduli spaces given in~\eqref{latticepol}. This
approach was first seen in Horikawa's announcement \cite{HorEnr} of a
Torelli theorem for Enriques surfaces. It was shown by Kondo
\cite{KoEnr} that the moduli space of Enriques surfaces is rational.

It is also natural in this context to consider moduli of $\Kthree$
surfaces with automorphisms. There is an extensive literature on such
surfaces, much of it touching on moduli problems: for example \cite{GaS}
and the recent work of Ma and Yoshikawa \cite{MY}. 

Many of these moduli spaces turn out to be unirational or even
rational, typically because the symmetry exhibits the $\Kthree$
surfaces as covers of $\PP^2$ and the family of possible ramification curves can
be parametrised.

Some moduli of $\Kthree$ surfaces with extra structure can be
described as ball quotients: see for example \cite{DGK}.  Again there
is an extensive literature on this subject. For a guide, we refer the
reader to the survey \cite{DK} by Dolgachev and Kondo. In the
introduction to \cite{DK} it is conjectured that all Deligne-Mostow
arithmetic complex ball quotients are moduli spaces of $\Kthree$
surfaces.

Much less is known in the case of irreducible symplectic manifolds,
but some cases have been studied.

\begin{example}\label{ex:voisin}
A classical case is the Fano variety of lines contained in a cubic
fourfold, which was studied in detail by Voisin \cite{Vo1}. These are
varieties of $\Kthree^{[2]}$ type. In our terminology this corresponds
to the degree~$6$ non-split case and the lattice $L_h$ orthogonal to
the polarisation vector is isomorphic to $2U \oplus 2E_8(-1) \oplus
A_2(-1)$.
\end{example}
\begin{example}\label{ex:ogradyEPW}
O'Grady studied double covers of
Eisenbud-Popescu-\-Wal\-ter sextics in \cite{OG4}. This is the case of split
polarisation of minimal degree (degree~$2$) for the
$\Kthree^{[2]}$-type. The lattice $L_h$ is $2U \oplus 2E_8(-1) \oplus
\latt{-2} \oplus \latt{-2}$.
\end{example}
\begin{example}\label{ex:ilievranestad}
Iliev and Ranestad (\cite {IR1}, \cite{IR2}) have shown that the
variety of sums of powers $\VSP(F,10)$ of presentations of a general
cubic form $F$ in $6$~variables as a sum of $10$~cubes is an
irreducible symplectic $4$-fold. These are deformations of length~$2$
Hilbert schemes of $\Kthree$ surfaces with a degree~$14$
polarisation. The precise nature of the polarisation of the irreducible
symplectic manifold is unknown.
\end{example}
\begin{example}\label{ex:debarrevoisin}
Debarre and Voisin (\cite{DV}) have constructed examples in the
Grassmannian $\Gr(6,V)$ where $V$ is a $10$-dimensional complex vector
space.Starting with a sufficiently general form $\sigma\colon \wedge^3
V \to \CC$ they show that the subspace of $\Gr(6,V)$ consisting of
$6$-planes $L$ such that $\sigma|_{\wedge^3 L} =0$ is an irreducible
symplectic fourfold of $\Kthree^{[2]}$-type. This defines a
$20$-dimensional family with polarisation of non-split Beauville
degree $2d=22$: the lattice $L_h$ is $2U \oplus 2E_8(-1)
\oplus \begin{pmatrix} -2 & 1\\ 1 &-6 \end{pmatrix}$.
\end{example}

Many authors have asked about the construction of geometrically
meaningful compactifications of moduli spaces of polarised $\Kthree$
surfaces. There are few general results known about this. For small
degree, some results can be found in \cite{Sha2}, \cite{Sha3},
\cite{St} and \cite{Sc}. A partial compactification is discussed in
\cite{Fr} and there is an approach via log geometry in \cite{Ol}.

\section{Compactifications}\label{sect:compactifications}

The spaces $\cF_{2d}$ (defined by~\eqref{F2d} above) and the other
quotients of period domains described in
Section~\ref{sect:projectivemodels} are complex analytic spaces by
construction. We observed above that they are in fact quasi-projective
varieties by the results of Baily and Borel, and the GIT moduli spaces
of polarised $\Kthree$ surfaces and of irreducible symplectic
manifolds are quasi-projective by the general results of Viehweg.
Nevertheless, we require projective models and preferably smooth, or
nearly smooth, models also.

In this section we describe the most commonly used compactifications
and we give some results about the singularities that arise. We begin
by describing the class of spaces we wish to compactify.

\subsection {Modular varieties of orthogonal
  type}\label{subsect:orthogonaltype}

As usual we let $L$ be an integral lattice of signature $(2,n)$,
$n\ge 3$, and consider the symmetric space
\begin{equation}\label{symmdomain}
\cD_L=\{ x\in \PP(L\tensor \CC)\mid (x,x)=0,\ (x,\bar x)>0\}^+
\end{equation}
where the superscript ${}^+$ denotes a choice of one of the two
connected components of $\Omega_L$. We let $\Gamma$ be a subgroup of
finite index in $\Orth^+(L)$. Any such $\Gamma$ acts properly
discontinuously on $\cD_L$, but in general there are elements of
finite order in $\Gamma$ and they have fixed points in $\cD_L$.

The quotient
\begin{equation}\label{definemodularvariety}
\cF_L(\Gamma)=\Gamma\backslash \cD(L)
\end{equation}
is called a \emph{modular variety of orthogonal type} or
\emph{orthogonal modular variety}. In particular
it is a locally symmetric variety, i.e.\ a variety that is the
quotient of a symmetric space by a discrete group of automorphisms. It
is not compact and is by its construction a complex analytic space: if
$\Gamma$ is torsion-free it is a complex manifold. In fact it is a
quasi-projective variety by~\cite{BB}.

Some particularly important examples of orthogonal modular varieties are:
\begin{itemize}
\item[\rm{(a)}] the moduli spaces of polarised $\Kthree$ surfaces (the
  signature is $(2,19)$);
\item[\rm{(b)}] the moduli spaces of lattice-polarised $\Kthree$
  surfaces (signature $(2,n)$, with $n<19$);
\item[\rm{(c)}] the moduli spaces of polarised abelian or Kummer
  surfaces (signature $(2,3)$);
\item[\rm{(d)}] the moduli space of Enriques surfaces (signature
  $(2,10)$);
\item[\rm{(e)}] quotients of the period domains of polarised
  irreducible symplectic varieties (signature $(2,4)$,  $(2,5)$,
  $(2,20)$ and $(2,21)$ in the known cases).
\end{itemize}

For nearly all the orthogonal modular varieties $\cF_L(\Gamma)$ that
occur in this article, $\Gamma$ is not torsion-free. The fixed points
can lead to singularities of $\cF_L(\Gamma)$. Since the stabiliser of
any point of $\cD_L$ is finite, the singularities are finite quotient
singularities: that is, locally analytically they are isomorphic to a
quotient of $\CC^n$ by a finite subgroup $G$ of $\GL(n,\CC)$. They are
not arbitrary finite quotient singularities, though, and we give some
details about them in Section~\ref{subsect:singinterior}.

The quotient $\CC^n/G$ may in fact be smooth, however. This happens, by
a result of Chevalley~\cite{Che}, if
and only if $G$ is generated by quasi-reflections (see
Definition~\ref{def:qref}). More importantly for us, the ramification
divisors of $\cD_L\to\cF_L(\Gamma)$ are precisely the fixed loci of
elements of $\Gamma$ acting as quasi-reflections on the tangent
space. In fact the only quasi-reflections that occur for orthogonal
modular varieties are reflections. (Let us emphasise that here we are
discussing the action of an element of $G$, not of $\Gamma$: being a
reflection on the tangent space to $\cD_L$ at a fixed point is not the
same as being a reflection as an element of $\Orth(L)$.)

The existence of elements acting as reflections, and thus of
ramification divisors, is a significant feature of orthogonal modular
varieties. For Siegel modular varieties (the symplectic group rather
than the orthogonal group) there are no ramification divisors, except
in the case of Siegel modular $3$-folds. Siegel modular $3$-folds,
however, may also be regarded as orthogonal modular varieties because
of the isogeny between $\Sp_2$ and $\SO(2,3)$.

Differential forms on $\cF_L(\Gamma)$ may be interpreted as modular
forms for $\Gamma$: see Section~\ref{subsect:orthmodfms} for more
details. Therefore arithmetic information (modular forms) may be used
to obtain geometric information about $\cF_L(\Gamma)$. In particular
we can use modular forms to decide whether $\cF_L(\Gamma)$ is of
general type, or more generally to try to determine its Kodaira
dimension. If $Y$ is a connected smooth projective variety of
dimension $n$, the \emph{Kodaira dimension} $\kappa(Y)$ of $Y$ is
defined by
\begin{equation*}
\kappa(Y)=\trdeg\Big(\bigoplus_{k\ge 0}H^0(Y,kK_Y)\Big)-1,
\end{equation*}
or $-\infty$ if $H^0(Y,kK_Y)=0$ for all $k> 0$. Thus $h^0(Y,kK_Y)\sim
k^{\kappa(Y)}$ for $k$ sufficiently divisible. The possible values of
$\kappa(Y)$ are $-\infty,0,1,\ldots,n=\dim Y$, and $Y$ is said to be of
\emph{general type} if $\kappa(Y)$. The Kodaira dimension
is a bimeromorphic invariant so it makes sense to extend the
definition to arbitrary irreducible quasi-projective varieties $X$ by
putting $\kappa(X)=\kappa(\Tilde X)$ for $\Tilde X$ a
desingularisation of a compactification of $X$.

With this in mind, we now turn to describing some algebraic
compactifications of $\cF_L(\Gamma)$, and the singularities that can
occur. Much further information about compactifications of
locally symmetric varieties (not all algebraic) may be found in the
book~\cite{BJ}, especially in \cite[Part III]{BJ}, and the references
there. However, the emphasis there is on the geometry and topology of
symmetric and locally symmetric spaces as real manifolds.

We are interested in two kinds of compactification: the
Baily-Borel compactification $\cF_L(\Gamma)^*$ and the toroidal
compactifications $\Bar{\cF_L(\Gamma)}$. We shall describe the
construction and some of the properties of each. 

\begin{remark}\label{rem:arithmeticgroup}
The constructions can also be made if $\Gamma$ is an arithmetic
subgroup of $\Orth^+(L\otimes\QQ)$; that is, if $\Gamma< \Orth(L\otimes
\QQ)$ and $\Gamma\cap \Orth^+(L)$ is of finite index in both $\Gamma$
and $\Orth^+(L)$.
\end{remark}

In some important ways the generalisation to arithmetic subgroups does
not change things much. If we are willing to change the lattice, by
Proposition~\ref{prop:rationalcoeffs} we can always assume that
$\Gamma$ is contained in $\Orth(L)$, so that we do not need rational
entries in the matrices. In particular the results of
Sections~\ref{subsect:modvarsings}--\ref{subsect:singcusps} still
hold.

\begin{proposition}\label{prop:rationalcoeffs}
If $\Gamma< \Orth(L\otimes \QQ)$ is arithmetic then there exists
a lattice $M\subset L\otimes \QQ$ such that $\Gamma< \Orth(M)$.
\end{proposition}
\begin{proof}
Let us consider $g(L)$ for all $g\in \Gamma$.
The index $[\Gamma: \Gamma\cap \Orth(L)]$ is finite, therefore
the number of different copies $g(L)$ of $L$ is finite.
Therefore the $\ZZ$-module generated by the union of all $g(L)$
($g\in \Gamma$) is finitely generated. Denote this lattice by $M$.
Then $\Gamma$ is a subgroup of $\Orth(M)$ by the definition of $M$, and
the quadratic form on $M$ is induced by the quadratic form on $L$.
If  the quadratic form is not even integral, then we can make it integral
taking a renormalisation by an integral constant $c$, i.e.\ we set
$(u,v)_M= c(u,v)_L$. Doing so does not change the orthogonal group.
It follows that $\Gamma$ can be considered as a subgroup of $\Orth(M)$
for some even integral lattice $M$.
\end{proof}
Notice, though, that the renormalisation by $c$ does change the stable orthogonal group.

\subsection{The Baily-Borel compactification}\label{subsect:bailyborel}

The Baily-Borel compactification, which in this context is often also
referred to as the Satake compactification, can be defined very
quickly as $\Proj\bigoplus M_k(\Gamma, 1)$, where $M_k(\Gamma,1)$
denotes the space of weight~$k$ modular forms with trivial character:
see Definition~\ref{def:modfm}.  A priori, however, it is not clear
that the ring of modular forms is finitely generated, nor that the
modular forms separate points of $\cF_L(\Gamma)$. Nor does that
description immediately give a picture of the boundary
$\cF_L(\Gamma)^*\setminus \cF_L(\Gamma)$.  Instead the approach of
Baily and Borel is to synthesise $\cF_L(\Gamma)$ by topological and
analytic methods, adding boundary components, and to show that the
resulting space is a projective variety. Full details are given
in~\cite{BB}, and a more detailed sketch than we give here may also be
found in~\cite{BJ}.

By writing $\cD_L$ in the form~\eqref{symmdomain} we have exhibited it
as a Hermitian domain of type~IV. As a Riemannian domain, $\cD_L\imic
\SO_0(2,n)/\SO(2)\times \SO(n)$, where $\SO_0(2,n)$ is the identity
component. The Baily-Borel compactification is defined for any
Hermitian symmetric space $\cD=G/K$ (instead of $\cD_L$) and for any
quotient $X=\Gamma\backslash \cD$ of $\cD$ by an arithmetic
group~$\Gamma$. We describe the construction in general but we keep in
mind the case $\cD=\cD_L$ as above.

An irreducible symmetric space $G/K$ is Hermitian if and only if the
centre of the maximal compact subgroup $K$ has positive dimension:
this explains why we consider only lattices of signature $(2,n)$,
because $\SO(m)\times\SO(n)$ has discrete centre unless $n=2$ or
$m=2$. Any Hermitian symmetric space of noncompact type can be
embedded as a bounded symmetric domain in the holomorphic tangent
space $T_K\cD$ (the Harish-Chandra embedding). For these facts see
\cite[Prop.~I.5.9]{BJ} or~\cite{Hel}.

The Baily-Borel compactification $\cD^{BB}$ of $\cD$ is simply the closure
of $\cD$ in the Harish-Chandra embedding, which in this case is the
closure $\bar\cD_L$ of $\cD_L$ in $\PP(L\otimes \CC)$; this in turn is
contained in the compact dual, which in this case is the quadric
\begin{equation*}
\check\cD=\{x\in \PP(L\tensor\CC)\mid (x,x)=0\}.
\end{equation*}
The tangent space to $G/K$ at $K$ is identified, as a complex manifold,
with an open subset of $\check\cD$.

A subset of $\cD^{BB}$ is called a \emph{boundary component} if it is
an analytic arc component: that is, an equivalence class under the
relation $x\sim y$ if there exist finitely many holomorphic maps
$f_i\colon \Delta=\{z\in\CC\mid |z|<1\}\to \cD^{BB}$ such that $x\in
f_1(\Delta)$, $y\in f_k(\Delta)$ and $f_i(\Delta)\cap
f_{i+i}(\Delta)\neq \emptyset$ for $1\le i<k$.  The Baily-Borel
compactification $\cD^{BB}$ decomposes as a disjoint union of boundary
components $F_P$, which are themselves symmetric spaces associated
with certain parabolic subgroups $p$ of $G$:
\begin{equation}\label{BBupper}
\cD^{BB}=\cD\amalg\coprod_P F_P.
\end{equation}
Not all parabolic subgroups occur, but only those associated with
certain collections of strongly orthogonal roots. See
\cite[Section~I.5]{BJ} for precise details.

By construction, $G$ acts on $\cD^{BB}$. The normaliser of $F_P$,
\begin{equation}\label{normaliser}
\cN(F_P)=\{g \in G\mid g(F_P)=F_P\}
\end{equation}
is a maximal parabolic subgroup of $G$ (the parabolic subgroup $P$ is
in general not maximal). We shall later also need to consider the
centraliser
\begin{equation}\label{centraliser}
\cZ(F_P)=\{g \in G\mid g|_{F_P}=\id\}.
\end{equation}
To construct the Baily-Borel compactification of $X$ -- in our case,
of $\cF_L(\Gamma)=\Gamma\backslash \cD_L$ -- one must first
restrict to rational boundary components.
\begin{definition}\label{def:ratbdycompt}
A boundary component $F$ is called a \emph{rational boundary
  component} if
\begin{itemize}
\item[(i)] the normaliser $\cN(F)$ of $F$ in $G$ is a parabolic
  subgroup defined over~$\QQ$, and
\item[(ii)] the centraliser $\cZ(F)$ contains a cocompact subgroup,
  normal in $\cN(F)$, which is an algebraic subgroup defined
  over~$\QQ$.
\end{itemize}
\end{definition}
A boundary component satisfying (i) is called a weakly rational
boundary component. It is shown in \cite[Theorem~3.7]{BB} that for the
Baily-Borel compactification (ii) follows from (i), so that weakly
rational boundary components are automatically rational.

$G$ acts on $\cD^{BB}$ but of course does not preserve rational boundary
components. However $\Gamma$, being an arithmetic subgroup, does take
rational boundary components to rational boundary components, so it
acts on $\cD^*:=\cD\amalg\coprod_{F_P\text{ rational}}F_P$. The effect of
condition~(ii) is that $\Gamma_P=\Gamma\cap \cN(F_P)$ is again a
discrete group. Moreover, $F_P$ is again a Hermitian symmetric space
and $\Gamma_P$ is an arithmetic group acting on~$F_P$.

We obtain the Baily-Borel compactification $(\Gamma\backslash\cD)^*$
by taking the quotient of $\cD^*$ by the action of $\Gamma$. Each
boundary component $\Gamma_P\backslash F_P$ has the structure of a
complex analytic space and it is shown in~\cite{BB} that these
structures can be glued together to give an analytic structure on
$(\Gamma\backslash \cD)^*$ extending the analytic structure on
$\Gamma\backslash \cD$.

It is at this stage that modular forms enter the picture. Each
boundary component is an analytic space by construction but to show
that their union is also an analytic space one must exhibit local
analytic functions, and to show that the resulting space is projective
we need analytic functions that separate points. Baily and Borel do
this by using the Siegel domain realisation of $\cD$ over a boundary,
which we describe below (Section~\ref{subsect:tube}) for the cases we
need here. For a more general description, see~\cite{Sat}. In these
coordinates one may write down suitable series (Poincar\'e-Eisenstein
series) that define modular forms having the required properties.

\begin{theorem}\label{thm:BBdescription}
The Baily-Borel compactification $(\Gamma\backslash \cD)^*$ is an
irreducible normal projective variety over $\CC$. It contains
$\Gamma\backslash \cD$ (which in our case is $\cF_L(\Gamma)$) as a
Zariski-open subset, and may be decomposed as
\begin{equation}\label{BBboundarydecomposition}
(\Gamma\backslash \cD)^*=\Gamma\backslash \cD\amalg\coprod_P
  \Gamma_P\backslash X_P,
\end{equation}
where $P$ runs through representatives of $\Gamma$-equivalence classes
of parabolic subgroups determining rational boundary components.
\end{theorem}

Although we defined rational boundary components in terms of
$\cD^{BB}$, they determin, according to condition~(i) above, rational
maximal parabolic subgroups of $G$. If $G$ is simple, which will
always be the case for us, the boundary components of $\cD^{BB}$
correspond precisely to the maximal real parabolic subgroups: see
\cite [{\S}3.2, Proposition~2]{AMRT}.

At least for the classical groups, the rational maximal parabolic
subgroups can be described in combinatorial terms. In the case of
$\Orth(L)$ with $L$ of signature $(2,n)$ they are the stabilisers of
isotropic subspaces of $L\otimes \QQ$. Because of the signature, such
spaces have dimension $2$ or $1$ (or $0$, corresponding to the
``boundary'' component $\Gamma\backslash \cD_L$).Therefore we obtain
the following description of $\cF_L(\Gamma)^*$.
\begin{theorem}
$\cF_L(\Gamma)^*$ decomposes into boundary components as
\begin{equation}\label{boundarycomponents_orth}
\cF_L(\Gamma)^*=\cF_L(\Gamma)\amalg\coprod_\Pi X_\Pi \amalg
\coprod_\ell Q_\ell,
\end{equation}
where $\ell$ and $\Pi$ run through representatives of the finitely
many $\Gamma$-orbits of isotropic lines and isotropic planes in
$L\otimes \QQ$ respectively. Each $X_\Pi$ is a modular curve, each
$Q_\ell$ is a point, and $Q_\ell$ is contained in the closure of
$X_\Pi$ if and only if the representatives may be chosen so that
$\ell\subset \Pi$.
\end{theorem}

$X_\Pi$ and $Q_\ell$ are usually referred to as $1$-dimensional and
$0$-dimensional boundary components, or corank~$1$ and corank-$2$
boundary components. The boundary components of the Baily-Borel
compactification are also known as the \emph{cusps}.

\subsection{Toroidal compactifications}\label{subsect:toroidal}

Toroidal compactifications in general are described in the
book~\cite{AMRT}. They are made by adding a divisor at each cusp.
Locally in the analytic topology near a cusp, the toroidal
compactification is a quotient of an open part of a toric variety over
the cusp: this variety is determined by a choice of admissible fan in
a suitable cone, and the choices must be made so as to be compatible
with inclusions among the closures of the Baily-Borel boundary
components. A summary may be found in~\cite[Chapter~III, \S5]{AMRT}.

The case we are concerned with,
of $\Orth(2,n)$, is simpler than the general case because only the
$0$-dimensional cusps need any attention. However, we shall begin by
describing the general theory, starting with $\cD=G/K$ and an action
of an arithmetic group $\Gamma$.

Let $F$ be a boundary component: we may as well assume immediately
that it is a rational boundary component. In general one has a
description of $\cD$ as a Siegel domain, an analytic open subset inside
\begin{equation}\label{DFnonhol}
\cD(F):=F\times V(F)\times U(F)_\CC.
\end{equation}
In this decomposition, $U(F)$ is the centre of the unipotent radical
$W(F)$ of $\cN(F)$, the normaliser of $F$ in $G$, and $V(F)\imic
W(F)/U(F)$ is an abelian Lie group.  This is not a holomorphic
decomposition ($V(F)$ does not have a natural complex structure) but
$U(F)_\CC$ acts holomorphically on $\cD(F)$ and in the diagram
\begin{equation*}
\xymatrix{
\cD(F) \ar[dd]_{\pi_F} \ar[dr]^{\pi'_F}\\
&\cD(F)' \ar[dl]_{p_F}&=&\cD(F)/U(F)_{\CC}\\
F
}
\end{equation*}
all the maps are holomorphic.

Now $\cD$ is given by a tube domain condition: there is a cone
$C(F)\subset U(F)$ such that
\begin{equation}\label{Dtube}
\cD=\{x \in \cD(F)\mid \Im(\pr_U(x))\in C(F)\}.
\end{equation}
where $\pr_U\colon \cD(F)\to U(F)_\CC$ is the projection map from the
decomposition~\eqref{DFnonhol}: this is a holomorphic map, even
though~\eqref{DFnonhol} is not a holomorphic product decomposition.

In fact, there is a holomorphic product decomposition of $\cD(F)$
which is (perhaps confusingly) similar:
\begin{equation}\label{Dfhol}
\cD(F)\imic U(F)_\CC\times \CC^k\times F,
\end{equation}
where, of course, $k={\frac{1}{2}}\dim_{\RR}V(F)$ but $\CC^k$ is not
naturally identified with $V(F)$.

Denoting the map $x\in \cD(F)\mapsto \Im(\pr_U(x))\in U(F)$ by
$\phi_F$, as in~\cite{AMRT}, we have the diagram
\begin{equation*}
\xymatrix{
C(F)&\subset&U(F)\\
\cD \ar[u]^{\phi_F} \ar[drr] \ar[dd]^{\pi_F} &\subset
&\cD(F) \ar[u]_{\phi_F} \ar[d]^{\pi'_F} &\subset &\check\cD\\
&& \cD(F)' \ar[dll]_{p_F}\\
F
}
\end{equation*}
in which $\pi'_F\colon \cD(F)\to \cD(F)'$ and $p_F\colon \cD(F)'\to F$
are principal homogeneous spaces for $U(F)_\CC$ and $V(F)$
respectively.

When $\Gamma$ acts, the group that acts on $\cD(F)$ and on $\cD$ is
$\cN(F)_\ZZ=\Gamma\cap\cN(F)$, which is a discrete group because $F$ is a
rational boundary component. So, looking at the action of
$U(F)_\ZZ=\Gamma\cap U(F)$, we get a principal fibre bundle
\begin{equation}\label{bdyfibration}
\cD(F)/U(F)_\ZZ \To \cD(F)'
\end{equation}
whose fibre is $T(F)=U(F)_\CC/U(F)_\ZZ$, an algebraic torus over
$\CC$.

Toroidal compactification proceeds by replacing this torus with a
toric variety $X_{\Sigma(F)}$ and taking the closure of $\cD/U(F)_\ZZ$ in
the $X_{\Sigma(F)}$-bundle over $\cD(F)'$ that results. Doing this for each
cusp $F$ separately one can then take the quotients of each such
$X_{\Sigma(F)}$ by $\cN(F)_\ZZ$ and (under suitable conditions) glue the
resulting pieces together by identifying the copies of $\cD/\Gamma$
contained in each one.

In this process $\Sigma(F)$ is in general very far from unique. It is
a fan in $C(F)$ (more precisely, of the convex hull of the rational
points of the closure of $C(F)$), i.e.\ a decomposition of $C(F)$ into
rational polyhedral cones (the integral structure is given by the
lattice $U(F)_\ZZ\subset U(F)$), which is required to be
$\cN(F)_\ZZ$-equivariant and locally finite, but is not itself finite
except in trivial cases: thus $X_{\Sigma(F)}$ is locally Noetherian,
but not Noetherian.

In general, in order for the gluing procedure to work, the fans must
satisfy a compatibility condition between fans for different cusps
that arises when one cusp is in the closure of another, but in the
case of $\Orth(2,n)$ the condition is automatically satisfied. The
reason is that the $1$-dimensional cusps have $\dim_\RR U(F)=1$ and
$C(F)=\RR_+$, and the cone decomposition is therefore unique, and
trivial. At the $0$-dimensional cusps, in contrast, one has $\dim_\RR
U(F)$, so $\cD$ is actually a tube domain in $\cD(F)$: we describe
this situation explicitly in Section~\ref{subsect:tube}.

At the $0$-dimensional cusps, therefore, many different choices of
compactification are possible. Below we shall choose one that suits
our purpose.

In the end we need to take the quotient by $\cN(F)_\ZZ$, not just
$U(F)_\ZZ$. This has two consequences. First, this is why it is
necessary to choose $\Sigma(F)$ and hence $X_{\Sigma(F)}$ in such a
way that $N(F)_\ZZ$ still acts; secondly, even if $X_{\Sigma(F)}$ is chosen
to be smooth, the action of $N(F)_\ZZ$ may reintroduce quotient
singularities into the finished toroidal compactification. It is easy
to choose $X_{\Sigma(F)}$ to be smooth, by the usual method of
subdivision to resolve toric singularities.

\begin{theorem}
A suitable choice of fans $\{\Sigma(F)\}$ for rational boundary
components $F$ determines a toroidal compactification
$\Bar{\cD/\Gamma}$ of $\cD/\Gamma$. This compactification may be
chosen to be projective, and to have at worst finite quotient
singularities.
\end{theorem}
\begin{proof}
The only part not described above is the assertion that the
compactification may be chosen to be projective. This is shown
in~\cite[Ch.IV, \S2]{AMRT} with the extra assumption that $\Gamma$ is
neat, which is harmless because one may work with a neat normal
subgroup $\Gamma'<\Gamma$ of finite index and then use the
$\Gamma/\Gamma'$-action. See also~\cite[V.5]{FC} for more details in
the Siegel (symplectic group) case, in a more arithmetic framework.
\end{proof}

\subsection{Canonical singularities}\label{subsect:canonicalsings}

In this part, we give sufficient conditions for the moduli space or a
suitable toroidal compactification of it to have canonical
singularities. We outline the proof, from \cite[Section~2]{GHSK3},
that orthogonal modular varieties of dimension $n\ge 9$ satisfy these
conditions.

\begin{definition}\label{def:cansings}
A normal complex variety $X$ is said to have \emph{canonical singularities}
if it is $\QQ$-Gorenstein and for some (hence any) resolution of
singularities $f\colon \Tilde X\to X$ the discrepancy
$\Delta=K_{\Tilde X}-f^* K_X$
is an effective Weil $\QQ$-divisor.
\end{definition}
Recall that $X$ being $\QQ$-Gorenstein means that for some $r\in\NN$,
if $K_X$ is a canonical (Weil) divisor on $X$ then $rK_X$ is
Cartier. Therefore $f^*K_X$ makes sense: by definition it is the
$\QQ$-Cartier divisor $\frac{1}{r}f^*(rK_X)$.

$\Delta=\sum \alpha_i E_i$ is supported on the irreducible exceptional
divisors $E_i$ for $f$, so $X$ has canonical singularities if and only
if the rational numbers $\alpha_i$ are all non-negative. Equivalently,
$X$ has canonical singularities if and only if on any open set
$U\subset X$, any pluricanonical form (i.e.\ section of $rK_x$ for
some $r$) on the smooth part of $U$ extends holomorphically to the
whole of $\Tilde U$. For more detail on canonical singularities,
see~\cite{YPG}.

A point $P\in X$ is said to be a canonical singularity if some
neighbourhood $X_0$ of $P$ has canonical singularities.

As we saw in Section~\ref{subsect:orthogonaltype}, the singularities
of $\cF_L(\Gamma)$ are finite quotient singularities, arising at the images of
points of $\cD$ whose stabiliser in $\Gamma$ is a non-trivial finite
group. Any such action can be linearised locally~\cite{Ca}: we
therefore consider the action of a finite subgroup $G<\GL(n,\CC)$ on
$\CC^n$ and the singularities of the quotient $X=\CC^n/G$. Any element
$g\in G$ can be diagonalised since it is of finite order, and the
eigenvalues of $g$ are roots of unity.

\begin{definition}\label{def:qref}
An element $g$ is a \emph{quasi-reflection} if exactly one of the
eigenvalues is different from~$1$. It is a \emph{reflection} if that
eigenvalue is~$-1$.
\end{definition}

For a cyclic subgroup $\latt{g}\subset \GL(n,\CC)$ of finite order
$m>1$, we choose a primitive $m$th root of unity $\zeta$ (without loss
of generality, $\zeta=e^{2\pi i/m}$) and we define the Reid-Tai sum
\begin{equation}\label{RTsum}
\Sigma(g)=\sum\left\{\frac{a_i}{m}\right\}
\end{equation}
where the eigenvalues of $g$ are $\zeta^{a_i}$ and $\{\,\}$ denotes
the fractional part, $0\le \{q\}<1$. For convenience we set
$\Sigma(1)=1$. The usual form of the
Reid-Tai criterion is the following.

\begin{proposition}\label{RSToriginal}
Suppose that $G$ is a finite subgroup of $\GL(n,\CC)$ containing no
quasi-reflections. Then $\CC^n/G$ has canonical singularities if and
only if $\Sigma(g)\ge 1$ for all $g\in G$.
\end{proposition}
This is sufficient if one wants to classify singularities, since any
quotient singularity is isomorphic to a quotient singularity where
there are no quasi-reflections. In our situation, the isotropy groups
sometimes do contain quasi-reflections, so we want a version of the
criterion that can be applied directly in that case. First, we state a
lemma that will allow us to consider the elements of $G$ one at a
time.

\begin{lemma}\label{lem:cyclic}
Suppose $G\subset \GL(n,\CC)$ is a finite group. If $\CC^n/\latt{g}$
has canonical singularities for every $g\in G$, then $\CC^n/G$ has
canonical singularities.
\end{lemma}

\begin{proof}
Let $\eta$ be a form on $(\CC^n/G)_{\reg}$ and let $\pi\colon \CC^n
\to \CC^n/G$ be the quotient map. Then $\pi^*(\eta)$ is a
$G$-invariant regular form on $\CC^n \setminus
\pi^{-1}(\CC^n/G)_{\sing}$.  Since $\pi^{-1}(\CC^n/G)_{\sing}$ has
codimension at least~$2$, the form $\pi^*(\eta)$ extends by Hartog's
theorem to a $G$-invariant regular form on all of~$\CC^n$. Now the
claim follows from \cite[Proposition~3.1]{Ta}, which says that a
$G$-invariant form on $\CC^n$ extends to a desingularisation of
$\CC^n/G$ if and only if it extends to a desingularisation of
$\CC^n/\latt{g}$ for every $g \in G$.
\end{proof}

The converse of Lemma~\ref{lem:cyclic} is false.

Suppose that $g\in GL(n,\CC)$ is of order $m=sk$, where $k$ is the
smallest positive integer such that $g^k$ is either a quasi-reflection
or the identity. Order the eigenvectors so that the first $n-1$
eigenvalues of $g^k$ are equal to $1$, so that the last eigenvalue is
a primitive $s$th root of unity. We define a modified Reid-Tai sum
\begin{equation}\label{modRTsum}
\Sigma'(g)=\left\{\frac{sa_n}{m}\right\}+\sum\left\{\frac{a_i}{m}\right\}
\end{equation}
where again the eigenvalues of $g$ are $\zeta^{a_i}$, and put
$\Sigma'(1)=1$. The idea (originating in an observation of Katharina
Ludwig) is that this enables us to handle the quasi-reflections correctly.

\begin{proposition}\label{RSTmodified}
Suppose that $G$ is a finite subgroup of $\GL(n,\CC)$. Then $\CC^n/G$
has canonical singularities if $\Sigma'(g)\ge 1$ for all $g\in G$.
\end{proposition}
\begin{proof}
If no power of $g$ is a quasi-reflection then $s=1$ and the usual
Reid-Tai criterion (Proposition~\ref{RSToriginal}) shows that
$\CC^n/\latt g$ has canonical singularities.

Otherwise, consider $g$ with $g^k=h$ a quasi-reflection as above. The
eigenvalues of $g$ are $\zeta^{a_1},\ldots,\zeta^{a_n}$, where $\zeta$
is a primitive $m$th root of unity, $\hcf(s,a_n)=1$ and $s|a_i$ for
$i<n$. The group $\latt{h}$ is generated by quasi-reflections so
$\CC^n/\latt{h}\imic \CC^n$, and we need to look at the action of the
group $\latt{g} /\latt{h}$ on $\CC^n/\latt{h}$. The eigenvalues of the
differential of $g^l\latt{h}$ on $\CC^n/\latt{h}$ are
$\zeta^{la_1},\ldots,\zeta^{la_{n-1}},\zeta^{sla_n}$, so
\begin{equation}\label{reduce_to_RTsum}
  \Sigma(g^l\latt{h})=\Sigma'(g^l)\ge 1.
\end{equation}
Thus $(\CC^n/\latt{h}/\latt{g\latt{h}}\imic \CC^n/\latt{g}$ has
canonical singularities and the result follows by Lemma~\ref{lem:cyclic}.
\end{proof}

\subsection{Singularities of modular varieties}\label{subsect:modvarsings}

We are interested primarily in the singularities of $\cF_{2d}$ and
the other spaces mentioned in Section~\ref{sect:projectivemodels}, but we may
more generally consider the singularities of compactified locally
symmetric varieties associated with the orthogonal group of a lattice
of signature~$(2,n)$. Unless $n$ is small, it turns out that
the compactification may be chosen to have canonical singularities.

\begin{theorem}\label{thm:csmain}
  Let $L$ be a lattice of signature $(2,n)$ with $n\ge 9$, and let
  $\Gamma<\Orth^+(L)$ be a subgroup of finite index. Then there exists
  a projective toroidal compactification $\Bar\cF_L(\Gamma)$ of
  $\cF_L(\Gamma)=\Gamma\backslash\cD_L$ such that $\Bar\cF_L(\Gamma)$
  has canonical singularities and there are no branch divisors in the
  boundary. The branch divisors in $\cF_L(\Gamma)$ arise from the
  fixed divisors of $\pm$reflections.
\end{theorem}

There are three parts to the proof of this theorem. One must first
consider the singularities of the open part, which means working out
some details of the action of $\Gamma$ on $\cD_L$. Then there are the
possible singularities over the $0$-dimensional cusps: these lead to
toric questions, and here one must choose the toroidal
compactification appropriately. Finally, in order to deal with the
$1$-dimensional cusps we need a full description of the geometry there.

\subsection{Singularities in the interior}\label{subsect:singinterior}

Here we are interested in the singularities that arise at fixed points
of the action of $\Gamma$ on $\cD_L$. Let $\bdw\in L_\CC$ and let
$G\subset \Gamma$ be the stabiliser of $[\bdw]\in \cD_L$. For
$[\bdw]\in \cD_L$ we define $\WW=\CC\bdw$. Then $G$ acts on $\WW$, so
$g(\bdw)=\alpha(g)\bdw$ for some character $\alpha\colon G \to \CC^*$,
and we put $G_0=\ker \alpha$.  We also put
$S=(\WW\oplus\Bar\WW)^\perp\cap L$ (possibly $S=\{0\}$) and
$T=S^\perp\subset L$. In the case of polarised $\Kthree$ surfaces, $S$
is the primitive part of the Picard lattice and $T$ is the
transcendental lattice of the surface corresponding to the period
point~$\bdw$.

It is easy to check that $S_\CC\cap T_\CC=\{0\}$ and that $G$ acts on
$S$ and on $T$: moreover $G_0$ acts trivially on $T_\QQ$.

Since $G/G_0\subset\Aut\WW\imic\CC^*$ it is a cyclic group: we denote
its order by $r_\bdw$. So by the above, $\mu_{r_\bdw}\imic G/G_0$ acts
on $T_\QQ$.  (By $\mu_r$ we mean the group of $r$th roots of unity in
$\CC$.)

For any $r\in \NN$ there is a unique faithful irreducible
representation of $\mu_r$ over $\QQ$, which we call $\cV_r$. The
dimension of $\cV_r$ is $\varphi(r)$, where $\varphi$ is the Euler
$\varphi$ function and, by convention, $\varphi(1)=\varphi(2)=1$. The
eigenvalues of a generator of $\mu_r$ in this representation are
precisely the primitive $r$th roots of unity: $\cV_1$ is the
{$1$-dimensional} trivial representation. Note that $-\cV_d=\cV_d$ if
$d$ is even and $-\cV_d=\cV_{2d}$ if $d$ is odd.

Using the fact that $S_\CC\cap T_\CC=0$, we may check that $T_\QQ$
splits as a direct sum of irreducible representations $\cV_{r_\bdw}$
(in particular, $\varphi(r_\bdw)|\dim T_\QQ$) and that if $g\in G$ and
$\alpha(g)$ is of order $r$ (so $r|r_\bdw$), then $T_\QQ$ splits as a
$g$-module into a direct sum of irreducible representations $\cV_r$ of
dimension~$\varphi(r)$.

We are interested in the action of $G$ on the tangent space to
$\cD_L$. We have a natural isomorphism
\begin{equation*}
T_{[\bdw]}\cD_L\imic \Hom(\WW,\WW^\perp/\WW)=:V.
\end{equation*}
Suppose $g\in G$ is of order $m$ and $\alpha(g)$ is of order $r$: as
usual we take $\zeta=e^{2\pi i/m}$, and henceforth we think of $g$ as
an element of $\GL(V)$, with eigenvalues $\zeta^{a_1},\ldots,\zeta^{a_n}$.

If $\varphi(r)$ is not very small, the copy of $\cV_r$ containing
$\bdw$ already contributes at least $1$ to $\Sigma(g)$. The cases
$r=1$ and $r=2$ are also simple.

\begin{proposition}\label{prop:RSTsimplest}
Assume that $g\in G$ does not act as a quasi-reflection on $V$ and
that $\varphi(r)>4$. Then $\Sigma(g)\ge 1$.
\end{proposition}
\begin{proof} As $\xi$ runs through the $m$th roots of unity, $\xi^{m/r}$ runs
  through the $r$th roots of unity.  We denote by
  $k_1,\ldots,k_{\varphi(r)}$ the integers such that $0<k_i<r$ and
  $(k_i,r)=1$, in no preferred order. Without loss of generality, we
  assume $\alpha(g)=\zeta^{mk_2/r}$ and
  $\Bar{\alpha(g)}=\alpha(g)^{-1}=\zeta^{mk_1/r}$, with $k_1\equiv
  -k_2 \bmod r$.

  One of the $\QQ$-irreducible subrepresentations of $g$ on $L_\CC$
  contains the eigenvector $\bdw$: we call this $\VV_r^\bdw$ (it is
  the smallest $g$-invariant complex subspace of $L_\CC$ that is
  defined over $\QQ$ and contains $\bdw$). It is a copy of
  $\cV_r\tensor\CC$: to distinguish it from other irreducible
  subrepresentations of the same type we write
  $\VV_r^\bdw=\cV_r^\bdw\tensor\CC$.

  If $\bdv$ is an eigenvector for $g$ with eigenvalue
  $\zeta^{mk_i/r}$, $i\neq 1$ (in particular $\bdv\not\in\Bar\WW$),
  then $\bdv\in \WW^\perp$ since
  $(\bdv,\bdw)=(g(\bdv),g(\bdw))=\zeta^{mk_i/r}\alpha(g)(\bdv, \bdw)$.
  Therefore the eigenvalues of $g$ on $\VV_r^\bdw\cap \WW^\perp /\WW$
  include $\zeta^{mk_i/r}$ for $i\ge 3$, so the eigenvalues on
  $\Hom(\WW, \VV_r^\bdw\cap\WW^\perp/\WW)\subset V$ include
  $\zeta^{mk_1/r}\zeta^{mk_i/r}$ for $i\ge 3$. So
\begin{eqnarray}\label{W-contribution}
  \Sigma(g)&\ge&
  \sum_{i=3}^{\varphi(r)}\left\{\frac{k_1}{r}+\frac{k_i}{r}\right\}\ge
  1,
\end{eqnarray}
where the last inequality is an elementary verification.
\end{proof}

The cases $r=1$ and $r=2$ are simple to deal with because one always
finds two conjugate eigenvalues, which between them contribute $1$ to
$\Sigma(g)$.

So far we have needed no hypothesis on the dimension, but the
remaining cases ($r=3,\ 4,\ 5,\ 6,\ 8,\ 10$ or $12$) do require such a
condition because the contributions to $\Sigma(g)$ from each $\cV_r$
are small and we need to have enough of them. We refer to~\cite{GHSK3}
for details. In the end we find
\begin{theorem}\label{thm:RSTnonqref}
  Assume that $g\in G$ does not act as a quasi-reflection on $V$ and
  that $n\ge 6$. Then $\Sigma(g)\ge 1$.
\end{theorem}
One must then carry out a similar analysis for quasi-reflections. One
more dimension is needed to guarantee $\Sigma'(g)\ge 1$, because the
$a_n$ term does not help us. The analysis also has the corollary that
the quasi-reflections in the tangent space $V$ that arise are in this
case always reflections, and moreover that the elements of $\Gamma$
that they come from are themselves (up to sign) reflections, considered
as elements of $\Orth(L)$.
\begin{corollary}\label{cor:csinterior}
  If $n\ge 7$ then $\cF_L(\Gamma)$ has canonical singularities.
\end{corollary}

\subsection{Singularities at the cusps}\label{subsect:singcusps}

We now consider the boundary
$\Bar\cF_L(\Gamma)\setminus\cF_L(\Gamma)$. Cusps, or boundary
components in the Baily-Borel compactification, correspond to orbits
of totally isotropic subspaces $E\subset L_\QQ$. Since $L$ has
signature $(2,n)$, the dimension of $E$ is $1$ or $2$, corresponding
to dimension~$0$ and dimension~$1$ boundary components respectively.

A toroidal compactification over a cusp $F$ coming
from an isotropic subspace~$E$ corresponds to an
admissible fan $\Sigma$ in some cone $C(F)\subset U(F)$. We have, as
in~\cite{AMRT}
\begin{equation*}
\cD_L(F):=U(F)_\CC\cD_L\subset \check\cD_L
\end{equation*}
where $\check\cD_L$ is the compact dual of $\cD_L$ (see \cite[Chapter
II, \S2]{AMRT}).

The case $\dim E=1$, that is, isotropic vectors in~$L$, is the case
of $0$-dimensional cusps in the Baily-Borel compactification and leads
to a purely toric problem. In this case we have
\begin{equation*}
\cD_L(F)\cong F\times U(F)_\CC=U(F)_\CC.
\end{equation*}
Put $M(F)=U(F)_\ZZ$ and define the torus $\bdT(F)=U(F)_\CC/M(F)$. In
general $(\cD_L/M(F))_\Sigma$ is by definition the interior of the
closure of $\cD_L/M(F)$ in $\cD_L(F)/M(F)\cross_{\bdT(F)}
X_\Sigma(F)$, i.e.\ in $X_\Sigma(F)$ in this case, where $X_\Sigma(F)$
is the torus embedding corresponding to the torus~$\bdT(F)$ and the
fan~$\Sigma$. We may choose $\Sigma$ so that $X_\Sigma(F)$ is smooth
and $G(F):=N(F)_\ZZ/U(F)_\ZZ$ acts on $(\cD_L/M(F))_\Sigma$: this is
also implicit in \cite{AMRT} and explained in~\cite[p.173]{FC}. The
toroidal compactification is locally isomorphic to $X_\Sigma(F)/G(F)$.
Thus the problem of determining the singularities is reduced to a
question about toric varieties, which is answered by
Theorem~\ref{thm:cstoric}, below.

We take a lattice $M$ of dimension $n$ and denote its dual lattice by
$N$. A fan $\Sigma$ in $N\tensor\RR$ determines a toric variety
$X_\Sigma$ with torus $\bdT=\Hom(M,\CC^*)=N\tensor\CC^*$.

\begin{theorem}\label{thm:cstoric}
  Let $X_\Sigma$ be a smooth toric variety and suppose that a finite
  group $G<\Aut(\bdT)=\GL(M)$ of torus automorphisms acts
  on~$X_\Sigma$.  Then $X_\Sigma/G$ has canonical singularities.
\end{theorem}
\begin{proof}
This is~\cite[Theorem~2.17]{GHSK3}.  The proof also shows (with a
little modification) that there are no branch divisors contained in
the boundary over $0$-dimensional cusps either.
\end{proof}

It remains to consider the dimension~$1$ cusps. We consider a rank~$2$
totally isotropic subspace $E_\QQ\subset L_\QQ$, corresponding to a
dimension~$1$ boundary component $F$ of $\cD_L$.  The idea is to
choose standard bases for $L_\QQ$ so as to be able to identify $U(F)$,
$U(F)_\ZZ$ and $N(F)_\ZZ$ explicitly, as is done in~\cite{Sc} for
maximal $\Kthree$ lattices, where $n=19$. Then, following
Kondo~\cite{Ko1} one can analyse the group action in coordinates,
using the Siegel domain realisation of $\cD$ associated with the given
cusp. Both in \cite{Ko1} and in \cite{Sc} there are special features
that allow one to work over $\ZZ$, but in general one must work over
$\QQ$. For details we refer to~\cite{GHSK3}.

\section{Modular forms and Kodaira dimension}\label{sect:modforms}

One of the main tools in the study of the geometry of the orthogonal
modular varieties $\cF_L(\Gamma)$ is the theory of modular forms with
respect to an orthogonal group of type $\Orth(2,n)$. One application
is to prove that $\cF_L(\Gamma)$ is often of general type.  The
methods described here were used in \cite{GHSK3} to prove the
following result.
\begin{theorem}\label{thm:K3gt}
  The moduli space $\cF_{2d}$ of $\Kthree$ surfaces with a
  polarisation of degree $2d$ is of general type for any $d>61$ and
  for $d=46$, $50$, $52$, $54$, $57$, $58$ and $60$.

  If $d\ge 40$ and $d\ne 41$, $44$, $45$ or $47$ then the Kodaira
  dimension of $\cF_{2d}$ is non-negative.
\end{theorem}
Similar methods apply to irreducible symplectic manifolds and their
polarisations, discussed in Section~\ref{subsect:modulipolsymp}. For
deformations of length~$2$ Hilbert schemes of $\Kthree$ surfaces with
polarisation of split type (see Equation~\eqref{splitlattice}) there is
the following result, from~\cite{GHSsymp}.
\begin{theorem}\label{thm:splitgt}
  The variety $\cM^{[2],\text{split}}_{2d}$ is of general type if $d\ge
  12$.  Moreover its Kodaira dimension is non-negative if $d=9$ and
  $d=11$.
\end{theorem}
For the ten-dimensional O'Grady case~\cite{OG1}, there are again split
and non-split polarisations, and a fairly complete general type result
in the split case was proved in~\cite{GHSdim21}.
\begin{theorem}\label{thm:21gt}
Let $d$ be a positive integer not equal to $2^n$ with $n\ge 0$.  Then
every component of the moduli space of ten-dimensional polarised
O'Grady varieties with split polarisation $h$ of Beauville degree
$h^2=2d\ne 2^{n+1}$ is of general type.
\end{theorem}
We do not attempt to prove Theorem~\ref{thm:21gt} here but the theory
we develop in the rest of this article will give proofs (though not
with full details) of Theorem~\ref{thm:K3gt} and Theorem~\ref{thm:splitgt}.

\subsection{Modular forms of orthogonal type}\label{subsect:orthmodfms}

In Definition~\ref{def:modfm} below we follow \cite{B1}. An ``affine''
definition similar to the one usually given for of $\SL(2)$ can be
found in \cite{G2}. The \emph{affine cone} over $\cD_L$ is
$\cD_L^\bullet=\{ y\in L\tensor \CC\mid x=\CC^*y\in \cD_L\}$.
\begin{definition}\label{def:modfm}
  Suppose that $L$ has signature $(2,n)$, with $n\ge 3$. Let $k\in
  \ZZ$ and let $\chi \colon \Gamma\to \CC^*$ be a character of a
  subgroup $\Gamma< \Orth^+(L)$ of finite index.  A holomorphic
  function $F\colon\cD_L^\bullet\to \CC$ is called a \emph{modular form} of
  \emph{weight} $k$ and character $\chi$ for the group $\Gamma$ if
\begin{equation*}
F(tZ)=t^{-k}F(Z)\quad \forall\,t\in \CC^*,
\end{equation*}
\begin{equation*}
  F(gZ)=\chi(g)F(Z)\quad  \forall\,g\in \Gamma.
\end{equation*}
A modular form is called a \emph{cusp form} if it vanishes at every
cusp.
\end{definition}
The weight as defined here is what is sometimes called
\emph{arithmetic weight}. Some authors prefer to use the geometric
weight, which is $k/n$, normally only in contexts where $n|k$. We
shall always use the arithmetic weight. One may choose a complex
volume form $dZ$ on $\cD_L$ such that if $F$ is a modular form of
weight $mn$ and character $\det$ for $\Gamma$ then
$F\,(dZ)^m$ is a $\Gamma$-invariant section of $mK_{\cD_L}$:
see~\cite{Bau} for a precise account.

If $n<3$ then one has to add to Definition~\ref{def:modfm} the
condition that $F$ is holomorphic at the boundary.  According to
Koecher's principle (see \cite{Ba}, \cite{F1}, \cite{P-S}) this
condition is automatically fulfilled if the dimension of a maximal
isotropic subspace of $L\otimes \QQ$ is smaller than $n$. In
particular, this is always true if $n\ge 3$.

We denote the linear spaces of modular and cusp forms of weight $k$
and character $\chi$ by $M_k(\Gamma,\chi)$ and $S_k(\Gamma,\chi)$
respectively.  If $M_k(\Gamma,\chi)$ is nonzero then one knows that
$k\ge (n-2)/2$ (see \cite{G2}).  The minimal weight $k=(n-2)/2$ is
called \emph{singular}. Modular forms of singular weight are very
special. The first example of such forms for orthogonal groups was
constructed in \cite{G1}.  Cusp forms are possible only if
$k>(n-2)/2$.  The weight $k=\dim(\cF_L(\Gamma))$ is called
\emph{canonical} because by a lemma of Freitag
\begin{equation*}
S_n(\Gamma, \det)\cong H^0\bigl(\Tilde{\cF}_L(\Gamma),
K_{\Tilde{\cF}_L(\Gamma)}\bigr),
\end{equation*}
where $\Tilde{\cF}_L(\Gamma)$ is a smooth compact model of the modular
variety $\cF_L(\Gamma)$ and $K_{\Tilde{\cF}_L(\Gamma)}$ is the sheaf
of canonical differential forms (see \cite[Hilfssatz 2.1,
  Kap. 3]{F1}). Therefore we have the following important formula for
the geometric genus of the modular variety:
\begin{equation}\label{geom-genus}
p_g(\Tilde{\cF}_L(\Gamma))=\dim S_n(\Gamma, \det).
\end{equation}

\begin{remark}\label{rem:holocondition}
Below (see Section~\ref{subsect:fourier}) we describe the property of
being holomorphic at the boundary (needed only if $n\le 2$) in terms
of the Fourier expansions.
\end{remark}

\begin{remark}\label{rem:propertyT}
In this article we usually assume that $n\ge 3$.  In this case the
order of any character $\chi$ in Definition~\ref{def:modfm} is
finite according to Kazhdan's property~(T) (see \cite{Ka}).
\end{remark}
\begin{remark}\label{rem:onlydet}
If the lattice $L$ contains two orthogonal copies of the hyperbolic
plane $U\cong \left(\begin{smallmatrix} 0&1\\1&0
\end{smallmatrix}\right)$
and if its reduction modulo $2$ (respectively $3$) is of rank at least
$6$ (respectively $5$) then $\Tilde{\Orth}^+(L)$ has only one
non-trivial character, namely $\det$ (see \cite{GHScomm}).  In
particular the modular group $\Tilde\Orth^+(L_{2d})$ related to the
polarised $\Kthree$ surfaces has only one non-trivial character.
\end{remark}
\begin{remark}\label{rem:paramodular}
If $L_{2t}^{(5)}=2U\oplus \latt{-2t}$, of signature $(2,3)$, then the
modular forms with respect to $\Tilde{\SO}^{+}(L_{2t}^{(5)})$ coincide
with Siegel modular forms with respect to the paramodular group
$\Gamma_t$ (see \cite{G3}, \cite{GH2}, \cite{GN1}, \cite{GN3}).  In
particular, if $t=1$ we obtain the Siegel modular forms with respect
to $\Sp_2(\ZZ)$.  In contrast to Remark~\ref{rem:onlydet} the group
$\Tilde{\SO}^{+}(L_{2t}^{(5)})$ has non-trivial characters. They were
described in \cite[Section 2]{GH3}.  One can construct important cusp
forms of the minimal possible weight $1$ with non-trivial character
for the full modular group $\Tilde{\SO}^{+}(L_{2t}^{(5)})\cong
\Gamma_t$ for some $t$ (see \cite{GN3}).
\end{remark}

\subsection{Rational quadratic divisors}\label{subsect:ratquaddiv}

For any $v\in L\otimes \QQ$ such that $v^2=(v,v)<0$ we define the
\emph{rational quadratic divisor}
\begin{equation}\label{div-Dv}
\cD_v=\cD_v(L)=\{[Z] \in \cD_L\mid (Z,v)=0\}\cong \cD_{v^\perp_L}
\end{equation}
where $v^\perp_L$ is an even integral lattice of signature $(2,n-1)$.
Therefore $\cD_v$ is also a homogeneous domain of type~IV.  We note
that $\cD_v(L)=\cD_{tv}(L)$ for any $t\ne 0$.  The theory of
automorphic Borcherds products (see \cite{B3}) gives a method of
constructing automorphic forms with rational quadratic divisors.
Special divisors of this type (the reflective divisors defined below)
play an important role in the theory of moduli spaces.

The reflection with respect to the hyperplane defined by a
non-isotropic vector $r$ is given by
\begin{equation}\label{sigma_r}
\sigma_r\colon l\Mapsto l-\frac{2(l,r)}{(r,r)}r.
\end{equation}
If $r$ is primitive in $L$ and the reflection $\sigma_r$ fixes $L$,
i.e.\ $\sigma_r\in \Orth(L)$, then we say that $r$ is a
\emph{reflective vector}, also known as a \emph{root}. If $(r,r)=d$ we
say that $r$ is a $d$-vector or (if it is a root) a $d$-root. A
$2$-vector or a $-2$-vector is always a root.

If $v\in L^\vee$ and $(v,v)<0$, the divisor $\cD_v(L)$ is called a
\emph{reflective divisor} if $\sigma_v\in \Orth(L)$.  It was proved in
\cite[Corollary 2.13]{GHSK3} that for $n\ge 6$ the branch divisor of
the modular projection
\begin{equation*}
\pi_\Gamma\colon  \cD_L\to \Gamma\setminus \cD_L
\end{equation*}
is the union of the reflective divisors with respect to $\Gamma$:
\begin{equation}\label{br-div}
\Bdiv(\pi_\Gamma)=\bigcup_{\ZZ{r}\subset L,\ \sigma_r\in \Gamma\cup -\Gamma} \cD_r(L).
\end{equation}
Note that here we have to allow $r$ such that $-\sigma_r\in \Gamma$ as
well as those with $\sigma_r\in\Gamma$: compare
Remark~\ref{rem:defreflective} below, concerning modular forms.

\subsection{Low weight cusp form trick}\label{subsect:lowwttrick}

The next theorem, proved in \cite[Theorem 1.1]{GHSK3}, is called the
\emph{low weight cusp form trick}.  It plays a crucial role in the
application of modular forms to moduli problems. If $F$ is a modular
form (of any weight or character) then the \emph{divisor} $\div F$ in
$\cD_L$ is given by the equation $F(Z)=0$: this is
well-defined in view of Definition~\ref{def:modfm}.
\begin{theorem}\label{thm:gt}
Let $L$ be an integral lattice of signature $(2,n)$, $n\ge 9$. The
modular variety $\cF_L(\Gamma)$ is of general type if there exists a
non-zero cusp form $F_a\in S_a(\Gamma,\chi)$ of small weight $a<n$
vanishing with order at least $1$ at infinity such that $\div F_a\ge
\Bdiv(\pi_\Gamma)$.
\end{theorem}
\begin{proof}
  We let $\Bar\cF_L(\Gamma)$ be a projective toroidal compactification
  of $\cF_L(\Gamma)$ with canonical singularities and no ramification
  divisors at infinity, which exists by Theorem~\ref{thm:csmain}. We
  take a smooth projective model $Y$ of $\cF_L(\Gamma)$ by taking
  a resolution of singularities of $\Bar\cF_L(\Gamma)$.  We want to
  show the existence of many pluricanonical forms on $Y$.

Suppose that $F_{nk}\in M_{nk}(\Gamma, \det^k)$. By choosing a
$0$-dimensional cusp we may realise $\cD_L$ as a tube domain (see
Section~\ref{subsect:tube} for details) and use this to select a
holomorphic volume element $dZ$.  Then the differential form
$\Omega(F_{nk})=F_{nk}\,(dZ)^k$ is $\Gamma$-invariant and therefore
determines a section of the pluricanonical bundle $kK=kK_{Y}$ away
from the branch locus of $\pi\colon\cD_L\to \cF_L(\Gamma)$ and the
cusps: see \cite[p. 292]{AMRT} (but note that weight $1$ in the sense
of \cite{AMRT} corresponds to weight $n$ in our definition).

In general $\Omega(F_{nk})$ will not extend to a global section of
$kK$.  We distinguish three kinds of obstruction to its doing so.
There are \emph{elliptic obstructions}, arising because of
singularities given by elliptic fixed points of the action of
$\Gamma$; \emph{reflective obstructions}, arising from the
ramification divisors in $\cD_L$; and \emph{cusp obstructions},
arising from divisors at infinity.

In order to deal with these obstructions we consider a neat normal
subgroup $\Gamma'$ of $\Gamma$ of finite index and set $G:= \Gamma /
\Gamma'$.  Let $X:=\cF_L(\Gamma')$ and let $\Bar X :=
\Bar\cF_L(\Gamma')$ be the toroidal compactification of
$\cF_L(\Gamma')$ given by the same choice of fan as for
$\Bar\cF_L(\Gamma)$.  Then $\Bar X$ is a smooth projective manifold
with $\Bar\cF_L(\Gamma)=\Bar X/G$.  Let $D:= \Bar X \setminus X$ be
the boundary divisor of $\Bar X$. For any element $g \in G$ we define
its fixed locus ${\Bar X}^g:= \{x \in \Bar X \mid g(x)=x\}$ and denote
its divisorial part by ${\Bar X}^g_{(1)}$.  Then $R:= \bigcup_{g \neq
  1}{\Bar X}^g_{(1)}$ is the ramification divisor of the map
$\pi\colon \Bar X \to {\Bar X} /G$.

The results of Section~\ref{subsect:modvarsings} (see
Theorem~\ref{thm:csmain} and Theorem~\ref{thm:cstoric} can be
summarised as follows:
\begin{itemize}
\item[(\rm{i})] $R$ does not contain a component of $D$;
\item[(\rm{ii})] the ramification index of $\pi\colon \Bar X \to \Bar
  X/G$ along $R$ is $2$;
\item[(\rm{iii})] $\Bar X/G$ has canonical singularities.
\end{itemize}
We will now apply the low-weight cusp form trick, used for example in
\cite{F1} (for Siegel modular forms), \cite{G2}, \cite{GH1} and
\cite{GS}.  The main point is to use special cusp forms.  For this let
the order of $\chi$ be $N$ and assume that $k$ is a multiple of
$2N$. Then we consider forms $F^0_{nk}\in
S_{nk}(\Gamma,\det^k)=S_{nk}(\Gamma,1)$ of the form
\begin{equation*}
F^0_{nk}=F_a^k F_{(n-a)k}
\end{equation*}
where $F_{(n-a)k}\in M_{(n-a)k}(\Gamma,1)$ is a modular form of weight
$(n-a)k\ge k$. We claim that the corresponding forms
$\Omega(F^0_{nk})$ give rise to pluricanonical forms on $Y$. To see
this, we deal with the three kinds of obstruction in turn.

\noindent{\bf Cusp obstructions.} By definition, $\Omega(F^0_{nk})$ is
a $G$-invariant holomorphic section of $kK_X$.  Since $F_a$ is a cusp
form of weight $a < n$, the form $F^0_{nk}$ has zeroes of order $k$
along the boundary $D$ and hence extends to a $G$-invariant
holomorphic section of $kK_{\Bar X}$ by \cite[Chap. IV, Th. 1]{AMRT}.

\noindent{\bf Reflective obstructions.} Since $R \subset \div(F_a)$ by
assumption, $\Omega(F^0_{nk})$ has zeroes of order $k$ on $R \setminus
D$. By $(\rm{i})$ above, $\Omega(F^0_{nk})$ actually has zeroes of
order $k$ along all of $R$. By $(\rm{ii})$ the form $\Omega(F^0_{nk})$
descends to a holomorphic section of $kK_{{(\Bar X /G)}_{\reg}}$ where
${(\Bar X/G)}_{\reg}$ is the regular part of ${\Bar X/G}$.

\noindent{\bf Elliptic obstructions.} By $(\rm{iii})$ the form
$\Omega(F^0_{nk})$ extends to a holomorphic section of $kK_Y$.

Therefore $F_a^k M_{(n-a)k}(\Gamma,1)$ is a subspace of
$H^0(Y,kK_Y)$. The theorem now follows because according to
Hirzebruch-Mumford proportionality (see \cite{Mum}), $\dim
M_{(n-a)k}(\Gamma,1)$ grows like $k^n$.
\end{proof}

\begin{remark}\label{rem:hm}
There is another way to deal with the reflective obstructions, which
works even if a cusp form with the right properties cannot be
found. Among forms of very high weight there must be some that vanish
along the reflective divisors, because $\dim M_k(\Gamma,1)$ grows
faster with $k$ than the space of obstructions, which are sections in
some bundles on the reflective divisors. In \cite{GHSprop} we estimate
these dimensions using Hirzebruch-Mumford proportionality. This method
can be used to produce general type results even in cases where
special forms constructed by quasi pull-back are not available, but
if the quasi pull-back method is available it normally produces much
stronger results.
\end{remark}

\subsection{Reflective modular forms}\label{subsect:reflectivemodfm}

For Theorem~\ref{thm:gt} we used cusp forms of low weight ($k<n$)
with large divisor ($\div F  \ge  \Bdiv(\pi_\Gamma)$).
We construct such modular forms for the moduli spaces of polarised
$\Kthree$ surfaces and other holomorphic symplectic varieties in
Section~\ref{sect:quasipb}.  Modular forms of high weight ($k\ge n$) with
small divisor ($\div F \le \Bdiv(\pi_\Gamma)$) also have applications
to the theory of moduli spaces such as Theorem~\ref{thm:kdim} below.

\begin{definition}\label{def:reflectivefm}
A modular form $F\in M_k(\Gamma,\chi)$ is called \emph{reflective} if
\begin{equation*}
\supp (\div F) \subset \bigcup_{\ZZ{r}\subset L,\ \sigma_r\in \Gamma\cup-\Gamma} \cD_r(L)
=\Bdiv(\pi_\Gamma).
\end{equation*}
We call $F$ \emph{strongly reflective} if the multiplicity of any
irreducible component of $\div F$ is equal to one.
\end{definition}

\begin{remark}\label{rem:defreflective}
In the definition of reflective modular forms given in \cite{GN2} only
the condition $\sigma_r\in \Gamma$ was considered.  The present
definition, allowing $-\sigma_r\in\Gamma$, is explained by
equation~\eqref{br-div}.
\end{remark}

\begin{example}\label{ex:phi12}
The most famous example of a strongly reflective modular form is the
\emph{Borcherds modular form} $\Phi_{12}\in M_{12}(\Orth^+(II_{2,26}),
\det)$ (see \cite{B1}).  This is the unique modular form of singular
weight $12$ with character $\det$ with respect to the orthogonal group
$\Orth^+(II_{2,26})$ of the even unimodular lattice $II_{2,26}\cong
2U\oplus 3E_8(-1)$ of signature $(2,26)$. The form $\Phi_{12}$ is the
Kac-Weyl-Borcherds denominator function of the Fake Monster Lie
algebra.  For any $(-2)$-vector $r\in II_{2,26}$ we have
\begin{equation*}
\Phi_{12}(\sigma_r(Z))=\det (\sigma_r) \Phi_{12}(Z)=-\Phi_{12}(Z).
\end{equation*}
Therefore $\Phi_{12}$ vanishes along $\cD_r(II_{2,26})$.  According to
\cite{B1} the order of vanishing is~$1$ and the full divisor of this
modular form is the union of the mirrors of such reflections:
\begin{equation*}
\div_{\cD(II_{2,26})} \Phi_{12} = \sum_{\substack{\pm r\in II_{2,26} \vspace{1\jot}\\
r^2=-2}} \cD_r(II_{2,26}).
\end{equation*}
\end{example}

According to Eichler's criterion (see Lemma~\ref{lem:eichler}) all
$(-2)$-vectors
of $II_{2,26}$ constitute one
$\Tilde{\Orth}^+(II_{2,26})$-orbit. In other words, the ramification
divisor of the $26$-dimensional modular variety
$\cF_{II_{2,26}}(\Tilde{\Orth}^+(II_{2,26}))$ is irreducible.

\begin{remark}\label{rem:canonicalweight}
 Modular forms of weight $n$ (known as canonical weight) have special properties.
Suppose $L$ has signature $(2,n)$ and $F\in M_n(\Gamma,\det)$.  If
$\sigma_r\in \Gamma$, then $F(\sigma_r(Z))=-F(Z)$. Hence $F$ vanishes
along $\cD_r(L)$.  If $-\sigma_r\in \Gamma$, then
\begin{equation*}
(-1)^nF(\sigma_r(Z))=F((-\sigma_r)(Z))=\det (-\sigma_r)F(Z)=(-1)^{n+1}F(Z)
\end{equation*}
and $F$ also vanishes along $\cD_r(L)$.  Therefore any
$\Gamma$-modular form of canonical weight with character $\det$
vanishes along $\Bdiv(\pi_\Gamma)$.

If $S_n(\Gamma,\det)\ne 0$ then the Kodaira dimension of
$\Bar{\cF}_L(\Gamma)$ is non-negative (with no restriction needed on
the dimension~$n$), because if $F_n\in S_n(\Gamma, \det)$ then by
Freitag's lemma $F_n(Z) dZ$ defines an element of $H^0(\Bar
\cF_L(\Gamma),K_{\Bar \cF_L(\Gamma)})$.  Therefore $p_g(\Bar
\cF_L(\Gamma))\ge 1$ and the plurigenera do not all vanish.
\end{remark}

The next theorem was proved in \cite{G4} and contrasts with
Theorem~\ref{thm:gt}.
\begin{theorem}\label{thm:kdim}
Suppose that $L$ has signature $(2,n)$, with $n\ge 3$. Let $F_k\in
M_k(\Gamma, \chi)$ be a strongly reflective modular form of weight $k$
and character $\chi$ for a subgroup $\Gamma<\Orth^+(L)$ of finite
index. Then
\begin{equation}\label{kappanegative}
\kappa(\cF_L(\Gamma))=-\infty
\end{equation}
if $k>n$, or if $k$ and $F_k$ is not a cusp form.  If $k$ and
$F_n$ is a cusp form whose order of zero at infinity is at least $1$
then
\begin{equation}\label{kappazero}
\kappa(\Gamma_\chi\backslash \cD_L)=0,
\end{equation}
where $\Gamma_\chi=\ker(\chi\cdot \det)$ is a subgroup of $\Gamma$.
\end{theorem}
\begin{proof}
The requirement that the cusp form should have order of vanishing at
least~$1$ is almost always satisfied: see~\cite{GHScomm}.

  To prove \eqref{kappanegative} we have to show that there are no
  pluricanonical differential forms on $\Bar{\cF}_L(\Gamma)$.  Any
  such differential form can be obtained using a modular form (see the
  proof of Theorem~\ref{thm:gt}).  The differential form
  $F_{nm}\,(dZ)^m$ is $\Gamma$-invariant and it determines a section
  of the pluricanonical bundle $mK$ over a smooth open part of the
  modular variety away from the branch locus of $\pi\colon\cD_L\to
  \cF_L(\Gamma)$ and the cusps.  In the proof of Theorem~\ref{thm:gt}
  we indicated three kinds of obstruction to extending
  $F_{nm}\,(dZ)^m$ to a global section of $mK$. In the proof of this
  theorem we use the reflective obstruction, arising from the
  ramification divisor in $\cD_L$ by $\pm$ reflections in $\Gamma$
  (see Equation~\eqref{br-div}).  Therefore if $F_{nm}$ determines a
  global section then $F_{nm}$ has zeroes of order at least $m$ on
  $\Bdiv(\pi_\Gamma)$.  The modular form $F_k\in M_k(\Gamma, \chi)$ is
  strongly reflective of weight $k\ge n$. Hence $F_{nm}/F_k^m$ is a
  holomorphic modular form of weight $m(n-k)\le 0$.  According to
  Koecher's principle ($n\ge 3$) this function is constant.  Therefore
  $F_{nm}\equiv 0$ if $k>n$ or $F_{nm}=C\cdot F_n^m$ if $k$.  If the
  strongly reflective form $F_n$ is non-cuspidal of weight $n$, then
  $F_n^m\,(dZ)^{\otimes m}$ cannot be extended to the compact model
  because of cusp obstructions ($F_n^m$ should have zeroes of order at
  least $m$ along the boundary).  If $F_n$ is a cusp form of weight
  $k$ then we can consider $F_n$ as a cusp form with respect to the
  subgroup $\Gamma_\chi$.

  Then $F_{n}(Z)\,dZ$ is $\Gamma_\chi$-invariant and, according to
  Freitag's lemma, it can be extended to a global section of the
  canonical bundle $\Omega_{\Bar{\cF}_L(\Gamma_\chi)}$ for any smooth
  compact model $\Bar{\cF}_L(\Gamma_\chi)$ of ${\cF}_L(\Gamma_\chi)$.
  Moreover Koecher's principle shows that any $m$-pluricanonical form
  is equal, up to a constant, to $F_n^m(dZ)^{\otimes m}$,
  proving~\eqref{kappazero}.  The strongly reflective cusp form of
  canonical weight determines essentially the unique
  $m$-pluricanonical differential form on $\cF_L(\Gamma_\chi)$.
\end{proof}
We can apply Theorem~\ref{thm:kdim} to find examples of moduli spaces
of lattice-polarised $\Kthree$ surfaces having $\kappa=-\infty$ and
$\kappa=0$: see~\cite{G4}.

\section{Orthogonal groups and reflections}\label{sect:groups}

The material in this and subsequent sections is not so easily found in
the literature, so from here on we shall give slightly more detail.

For applications, the most important subgroups of $\Orth(L)$ are the
stable orthogonal groups $\Tilde\Orth(L)$, $\Tilde\Orth^+(L)$ and
$\Tilde{\SO}^+(L)$, as defined in Equations~\eqref{definestablegroup}
and~\eqref{defineOtilde+}.
The reason for using the word ``stable'' to describe $\Tilde\Orth(L)$
is the following property.
\begin{lemma}\label{lem:stable}
For any sublattice $S$ of a lattice $L$ the group  $\Tilde\Orth(S)$
can be considered as a subgroup of $\Tilde\Orth(L)$.
\end{lemma}
\begin{proof}
Let $S^\perp$ be the orthogonal complement of $S$ in $L$.
We have
\begin{equation}\label{diag-orthcoml}
S\oplus S^\perp \subset L\subset L^{\vee}\subset S^{\vee}\oplus (S^\perp)^{\vee}
\end{equation}
where $S\oplus S^\perp$ is a sublattice of finite index in $L$.
We can extend $g\in \Tilde\Orth(S)$ on $S\oplus S^\perp$
putting $g|_{S^\perp}\equiv \hbox{id}$. It is clear that
$g\in \Tilde\Orth(S\oplus S^\perp)$. We consider
$g\in \Tilde\Orth(S\oplus S^\perp)$ as an element of
$\Orth(S^{\vee}\oplus (S^\perp)^\vee)$.
For any $l^\vee\in L^{\vee}$ we have $g(l^\vee)\in l^\vee+(S\oplus S^\perp)$.
In particular,
$g(l)\in L$ for any $l\in L$ and $g\in \Tilde\Orth(L)$.
\end{proof}

Let $S$ be a primitive sublattice of $L$.
We define the groups
\[\Orth(L,S)=\{ g\in \Orth(L)\mid g|_S\in \Tilde\Orth(S)\}
\quad\text{and}\quad \Tilde\Orth(L,S)=\Orth(L,S)\cap \Tilde\Orth(L).
\]
Note that $\Orth(L,\ZZ h)=\Orth(L, h)$ if $h^2\ne \pm 2$.  The
technique of discriminant forms developed by Nikulin in \cite{Nik2}
is very useful here, and we describe the main ideas behind it below.
For simplicity we assume that all the lattices we consider are even.

Let $S^{\perp}$ be the orthogonal complement of a primitive
nondegenerate sublattice $S$ in $L$. As in the proof of
Lemma~\ref{lem:stable} we have the inclusions~\eqref{diag-orthcoml}.
The overlattice $L$ is defined by the finite subgroup
\begin{equation}\label{diag-H}
H=L/(S^{\perp}\oplus S)<(S^\perp)^\vee/S^\perp \oplus S^\vee/S=
D(S^\perp)\oplus D(S)
\end{equation}
which is an isotropic subgroup of $D(S^\perp)\oplus D(S)$.
Moreover   $L/(S\oplus S^\perp)\cong L^\vee/(S^\vee\oplus (S^\perp)^\vee)$.
We define  $\phi\colon L\to S^\vee$ by $\phi(l)(s)=(l,s)$.
Then $\ker(\phi)=S^\perp$.
Since $L/(S\oplus S^\perp)\cong \phi(L)/S$ we obtain
\begin{equation*}
|L/(S\oplus S^\perp)|= |\phi(L)/S|=|\det S|/[S^\vee:\phi(L)],
\end{equation*}
as $|\det S|=[S^\vee:S]$.  From the inclusions above
\begin{equation*}
|\det S|\cdot|\det S^\perp|=(|\det L|)[\phi(L):S]^2
=|\det L|\cdot|\det S|^2/[S^\vee:\phi(L)]^2.
\end{equation*}

In the particular case $S=\ZZ h$ and $L_h=h^\perp_L$
we have $[S^\vee:\phi(L)]=\divv(h)$, where
\begin{equation}\label{def:divv}
\divv(h)\ZZ=(h,L).
\end{equation}
The positive number $\divv(h)$ is called the \emph{divisor} of $h$ in
$L$. We have now proved the following lemma.

\begin{lemma}\label{lem:hperp}
Let $L$ be any nondegenerate even integral lattice and let $h\in L$
be a primitive vector with $h^2=2d\ne 0$. If $L_h$ is the orthogonal
complement of $h$ in $L$ then
\begin{equation*}
|\det L_h|=\frac {|(2d)\cdot \det L|}{\div (h)^2}.
\end{equation*}
\end{lemma}

We come back to the inclusion~\eqref{diag-H}.
Following \cite{Nik2} we consider the projections
\begin{alignat*}{2}
p_S\colon H \to D(S),\qquad & p_{S^\perp}\colon H\to D(S^\perp).
\end{alignat*}
Using the definitions and the fact that the lattices $S$ and $S^\perp$
are primitive in $L$ one can show (see \cite[Prop. 1.5.1]{Nik2}) that
these projections are injective and moreover that if $d_S\in p_S(H)$
then there is a unique $d_{S^\perp}\in p_{S^\perp}(H)$ such that
$d_S+d_{S^\perp}\in H$.  Using these arguments one proves the next
lemma (see \cite[Lemma 3.2]{GHSsymp})

\begin{lemma}\label{lem:O(L,S)}
Let $S$ be a primitive sublattice of an even lattice $L$ and denote by
$\bar g$ the image in $\Orth(D(L))$ of $g\in \Orth(L)$.
\begin{itemize}
\item[{(i)}] $g\in \Orth(L,S)$ if and only if $g(S)=S$, $\bar
  g|_{D(S)}=\id$ and $\bar g|_{p_{S^\perp}(H)}=\id$.
\item[{(ii)}] $\alpha\in \Orth(S^\perp)$ can be extended to
  $\Orth(L,S)$ if and only if $\bar\alpha|_{p_{S^\perp}(H)}=\id$.
\item[{(iii)}] If $p_{S^\perp}(H)=D(S^\perp)$ then
  $\Orth(L,S)|_{S^\perp}\cong \Tilde\Orth(S^\perp)$.
\item[{(iv)}] Assume that the projection $\Orth(S^\perp)\to
  \Orth(D(S^\perp))$ is surjective. Then
\begin{equation*}
\Orth(L,S)|_{S^\perp}/\Tilde\Orth(S^\perp)\cong
\{\bar\gamma \in \Orth(D(S^{\perp}))\mid \bar\gamma|_{p_{S^\perp}(H)}=\id\}.
\end{equation*}
\end{itemize}
\end{lemma}

\begin{corollary}\label{cor:O(L,S)}
If $|H|=|\det S^\perp|$ then $\Orth(L,S)|_{S^\perp}\cong
  \Tilde\Orth(S^\perp)$.
\end{corollary}

For example, the condition of Corollary~\ref{cor:O(L,S)} is true if
$L$ is an even unimodular lattice and $S$ is any primitive sublattice
of $L$ (see Example~\ref{ex:K3} below).  We recall the following
result which we call \emph{Eichler's criterion} (see~\cite[Section
  10]{E}, \cite[Section 3]{G2} and \cite{GHScomm}).

\begin{lemma}\label{lem:eichler}
Let $L$ be a lattice containing two orthogonal isotropic planes. Then
the $\Tilde\SO(L)$-orbit of a primitive vector $l\in L$ is
determined by two invariants: by its length $l^2=(l,l)$ and its image
$l^*+L$, where $l^*=l/\divv(l)$, in the discriminant group $D(L)$.
\end{lemma}

We note that $l^*$ is a primitive element of the dual lattice
$L^\vee$. Therefore $\divv(l)$ is a divisor of the exponent of the
discriminant group $D(L)$.  In particular, $\divv(l)$ divides
$\det(L)$.  Lemma~\ref{lem:eichler} can be used to classify all
possible vectors of fixed length in different lattices.

\begin{example}\label{ex:K3}
\emph{(The $\Kthree$ lattice.)}  $L(\Kthree)=3U\oplus 2E_8(-1)$ is the
even unimodular lattice of signature $(3,19)$: its discriminant group
is trivial and all the primitive vectors $h_{2d}\in L_{\Kthree}$ of
length $2d$ form a single $\SO(L_{\Kthree})$-orbit.  Therefore we can
take $h_{2d}$ in the first hyperbolic plane $U$, so
\begin{equation}\label{L2d}
(h_{2d})^\perp_{L_{\Kthree}}\cong L_{2d}=2U\oplus 2E_8(-1)\oplus \latt{-2d}.
\end{equation}
Then according to Corollary~\ref{cor:O(L,S)}
\begin{equation}\label{K3mod-gr}
\Orth^+(L_{\Kthree}, h_{2d}) \cong \Tilde\Orth^+(h_{2d}^{\perp})=
\Tilde\Orth^+(2U\oplus 2E_8(-1)\oplus \latt{-2d}).
\end{equation}
The case of $\Kthree$ surfaces is an exception, since the
$\Kthree$ lattice is unimodular.
\end{example}

\begin{example}\label{ex:K3n}
\emph{$\Kthree^{[2]}$-lattice, split and non-split polarisations.}  We
consider the Beauville lattice of a deformation $\Kthree^{[n]}$
manifold, which is isomorphic to $L_{\Kthree,2n-2}$ by
Proposition~\ref{prop:Beauvillelattices}(i).  In general there will be
several, but finitely many, orbits of primitive polarisation vectors
$h_{2d}$.  Let $h_{2d}\in L_{\Kthree,2}$ be a primitive vector of
length $2d>0$.  Then $\div(h_{2d})$ divides $|\det L_{\Kthree,2}|=2$.

All vectors with $\div(h_{2d})=1$ constitute a single
$\Tilde{\SO}^+(L_{\Kthree,2})$-orbit, by Lemma~\ref{lem:eichler}.
Therefore, as in Example~\ref{ex:K3}, we obtain
\begin{equation}\label{L2d-split}
(h_{2d})^\perp_{L_{\Kthree,2}}\cong L_{2,2d}=
2U\oplus 2E_8(-1)\oplus  \latt{-2}\oplus  \latt{-2d}.
\end{equation}
We call a polarisation determined by a primitive vector $h_{2d}$
with $\div(h_{2d})=1$ a \emph{split polarisation}.

If $h_{2d}\in L_{\Kthree,2}$ and $\div(h_{2d})=2$ then we can write
$h_{2d}$ as $h_{2d}=2v+cl_2$, where $v\in 3U\oplus 2E_8(-1)$ and $l_2$
is a generator of the orthogonal component $\latt{-2}$ in
$L_{\Kthree,2}$.  The coefficient $c$ is odd because $h_{2d}$ is
primitive.  Note that $2d=h_{2d}^2=4(v,v)-2c^2$, so $d\equiv -1 \mod
4$.  According to Eichler's criterion the
$\Tilde\SO(L_{\Kthree,2})$-orbit of $h_{2d}$ is uniquely determined by
the class $h_d^*\equiv l_2/2\mod L_{\Kthree,2}$.  Therefore as in the
case of $\div(h_{2d})=1$ all vectors with $\div(h_{2d})=2$ form only
one orbit.  We can take a representative in the form $h_{2d}=2v+cl_2
\in U\oplus \latt{-2}$.  The orthogonal complement of $h_{2d}$ in
$U\oplus \latt{-2}$ can be found by direct calculation. This is an
even rank~$2$ lattice $Q(d)$ of determinant $d$.  It follows that if
$d\equiv -1 \mod 4$ then there is only one orbit of vectors
$h_{2d}$ with $\div (h_{2d})=2$, and the orthogonal complement of
$h_{2d}$ is uniquely determined:
\begin{equation}\label{L2d-nonsplit}
(h_{2d})^\perp_{L_{\Kthree,2}}\cong L_{Q(d)}=
2U\oplus 2E_8(-1)\oplus  \left(
\begin{matrix}
-2&1\\1&-\frac{d+1}2
\end{matrix}\right).
\end{equation}
We call a polarisation of this kind a \emph{non-split
  polarisation}. We note that $|\det L_{Q(d)}|=d$ and the discriminant
group of $L_{Q(d)}$ is cyclic (see \cite[Remark 3.15]{GHSsymp}).
\end{example}

\begin{remark}\label{rem:split}
Taking the orthogonal complement of the $(-2)$-vector in $Q(d)$ we have a
split sublattice of index $2$
\begin{equation*}
\latt{-2}\oplus \latt{-2d} <  Q(d).
\end{equation*}
Therefore $L_{2,2d}<L_{Q(d)}$ is  also a sublattice of index $2$  and
according to Lemma~\ref{lem:stable}
\begin{equation}\label{2d-Q(d)}
\Tilde\Orth^+(L_{2,2d})< \Tilde\Orth^+(L_{Q(d)})
\end{equation}
is a subgroup of finite index.  It follows from this that the modular variety
$\cF_L(\Tilde\Orth^+(L_{2,2d}))$ is a finite covering of
$\cF_L(\Tilde\Orth^+(L_{Q(d)}))$ where $d\equiv -1$ mod~$4$.  One can
calculate this index using the explicit formula for the
Hirzebruch--Mumford volume (see \cite{GHSvol}).
\end{remark}

We gave a classification of all possible type of polarisations for the
symplectic varieties of $\Kthree^{[n]}$ type in \cite[Proposition
3.6]{GHSsymp}.  Results about the polarisation types of
$10$-dimensional O'Grady varieties can be found in \cite[Theorem
3.1]{GHSdim21}.

The branch divisor (Equation~\eqref{br-div}) of the modular varieties
is defined by $\pm$reflections in the modular groups.  Below we give
a description of the branch divisors in the cases of the polarised
$\Kthree$ surfaces and polarised holomorphic symplectic varieties of
type $\Kthree^{[2]}$.

Let $L$ be a nondegenerate integral lattice and $r\in L$ be a
primitive vector.  If the reflection is integral, i.e. $\sigma_r\in
\Orth(L)$ (see Equations~\eqref{sigma_r}) and \eqref{def:divv}), then
\begin{equation}\label{refldiv}
\divv (r)\mid r^2 \mid 2\divv (r).
\end{equation}
The following general result was proved in \cite{GHSK3}.

\begin{proposition}\label{prop:reflvectors}
(i) Let $L$ be a nondegenerate even integral lattice.  Let $r\in L$ be
primitive. Then $\sigma_r\in \Tilde \Orth(L)$ if and only if $r^2=\pm 2$.

If $-\sigma_r\in \Tilde \Orth(L)$, {i.e.} $\sigma_r|_{A_L}=-\id$, then
we also have
\begin{itemize}
\item[(ii)] $r^2=\pm 2a$ and $\divv(r)=a\equiv 1\bmod 2$, or $r^2=\pm a$
  and $\divv(r)=a$ or $a/2$; and
\item[(iii)] $A_L\cong (\ZZ/2\ZZ)^m\times (\ZZ/a\ZZ)$, for some $m\ge 0$.
\end{itemize}
If (iii) holds then
\begin{itemize}
\item[(iv)] If $r^2=\pm a$ and either $\divv(r)=a$ or
  $\divv(r)=a/2\equiv 1\bmod 2$, then $-\sigma_r\in \Tilde
  \Orth(L)$;
\item[(v)] If $r^2=\pm 2a$ and $\divv(r)=a\equiv 1\bmod 2$, then
  $-\sigma_r\in \Tilde \Orth(L)$.
\end{itemize}
\end{proposition}

With polarised $\Kthree$ surfaces in mind, we consider in more detail
the lattice $L_{2d}$ (see Equation~\eqref{L2d}). In this case the
ramification divisor has three irreducible components.
\begin{corollary}\label{cor:reflK3}
Let $\sigma_r$ be a reflection in $\Orth^+(L_{2d})$ defined by a
primitive vector $r\in L_{2d}$.  The reflection $\sigma_r$ induces
$\pm\id$ on the discriminant group $L_{2d}^\vee/L_{2d}$ if and only if
$r^2=-2$ or $r^2=- 2d$ and $\divv(r)=d$ or $2d$.  If $r^2=-2$ then
\begin{equation*}
r^\perp_{L_{2d}}\cong 2U\oplus E_8(-1)\oplus E_7(-1)\oplus \latt{-2d}.
\end{equation*}
If $\divv(r)=d$ then either
\begin{equation*}
r^\perp_{L_{2d}}\cong U\oplus 2E_8(-1)\oplus \latt{2} \oplus \latt{-2}
\end{equation*}
or
\begin{equation*}
r^\perp_{L_{2d}}\cong U\oplus 2E_8(-1)\oplus U(2).
\end{equation*}
\end{corollary}

See the proof in \cite[Corollary 3.4 and Proposition 3.6]{GHSK3}.
Geometrically the three cases in the last proposition correspond to
the N\'eron-Severi group being (generically) $U$, $U(2)$ or
$\latt{2}\oplus\latt{-2}$ respectively. The $\Kthree$ surfaces
(without polarisation) themselves are, respectively, a double cover of
the Hirzebruch surface $F_4$, a double cover of a quadric, and the
desingularisation of a double cover of $\PP^2$ branched along a nodal
sextic.

We note that in the case of polarised deformation $\Kthree^{[2]}$
manifolds the branch divisor has one main (i.e.\ $\divv(r)=1$)
component, with $r^2=-2$, and $6$ (respectively, $1$) additional
components ($\divv(r)>1$) for split (respectively, non-split) type
(see the proof of Proposition~\ref{prop:qpbvanishing} below).

\section{The quasi pull-back of modular forms}\label{sect:quasipb}

The main aim of this section is to show how we can construct cusp
forms of small weight, for example on the moduli spaces of polarised
$\Kthree$ surfaces.  We use the method of quasi pull-back of the
Borcherds form $\Phi_{12}$ which was proposed in \cite[pp.
200-201]{B1}.  This method was successfully applied to the theory of
moduli spaces in \cite{BKPS}, \cite{Ko2}, \cite{GHSK3}, \cite{GHSsymp}
and \cite{GHSdim21}.  In this section we review this method and prove
a new result (Theorem~\ref{thm:qpbcusp} below) showing that non-trivial quasi
pull-backs are cusp forms.

\subsection{Quasi pull-back}\label{subsect:qpb}

First we give a general property of rational quadratic divisors.  Let
$M$ be a lattice of signature $(2,m)$ and $L$ be a primitive
nondegenerate sublattice of signature $(2,n)$ where $n<m$.  Then
$L^\perp_M$ is negative definite and we have as usual $ L\oplus
L^\perp_M < M<M^{\vee} <L^\vee\oplus (L^\perp)^\vee $.  For $v\in M$
we write
\begin{equation}\label{proj}
v=\alpha+\beta,\qquad \alpha=\pr_{L^\vee}(v)\in L^\vee, \ \beta\in (L^\perp)^\vee.
\end{equation}
\begin{lemma}\label{lem:LdivDL} Let $L$ and $M$  be as above.
Then for any  $v\in M$ with  $v^2<0$ we have
\begin{equation*}
\cD_L\cap\cD_v(M)=
\begin{cases}
\cD_\alpha(L),&\text{ if }\  \alpha^2<0, \\
\emptyset,&\text{ if }\  \alpha^2\ge 0,\ \alpha\ne 0,\\
\cD_L,&\text{ if }\  \alpha=0, \text{ i.e.\ } v\in L^\perp.
\end{cases}
\end{equation*}
\end{lemma}
\begin{proof}
We have $\cD_L \subset \cD_M$ because $L$ is a sublattice of $M$.
For any $Z=X+iY\in \cD_L^\bullet$ with $X,Y\in L\otimes\RR$ we have
$(X,Y)=0$ and $(X,X)=(Y,Y)>0$. Therefore the quadratic space
$\latt{X,Y}_\RR$ is of signature $(2,0)$.  Note that $(Z,v)=0$ is
equivalent to $(Z,\alpha)=0$.  The signature of $L$ is equal to
$(2,n)$. Analysing the signature of $\latt{X,Y}_\RR\oplus
\latt{\alpha}_\RR$ we get that $\cD_L\cap\cD_v(M)$ is non-empty if
and only if $\alpha=\pr_{L^\vee}(v)$ belongs to the negative definite
quadratic space $\latt{X,Y}^\perp_{M\otimes \RR}$. This proves the
first two cases of the lemma.

The finite group $H=M/(L\oplus L^\perp)$ is a subgroup of the
orthogonal sum of the discriminant groups $D(L)\oplus D(L^\perp)$
where $D(L)=L^\vee/L$.  The decomposition~\eqref{proj} defines a
projection $\pr\colon H\to D(L)\oplus D(L^\perp)$. For a primitive
sublattice $L$ the class $(\alpha+L)\in D(L)$ is uniquely determined
by the class $\beta+L^\perp$ in $D(L^\perp)$ (see \cite[Proposition
  1.5.1]{Nik2}).  Therefore $\alpha\in L$ if and only if $\beta\in
L^\perp$. In particular $v$ is orthogonal to $L$ if and only if
$\alpha=0$.  This proves the last assertion of the lemma.
\end{proof}

In the next theorem we explain the main idea of the method of quasi
pull-back applied to the strongly reflective modular form $\Phi_{12}$
(see \cite[pp.~200--201]{B1} and \cite{BKPS}).

\begin{theorem}\label{thm:qpb}
Let $L\emb II_{2,26}$ be a primitive nondegenerate sublattice of
signature $(2,n)$, $n\ge 3$, and let $\cD_L\emb\cD_{II_{2,26}}$ be the
corresponding embedding of the homogeneous domains.  The set of
$(-2)$-roots
\begin{equation*}
R_{-2}(L^\perp)=\{r\in II_{2,26}\mid r^2=-2,\ (r, L)=0\}
\end{equation*}
in the orthogonal complement is finite.
We put $N(L^\perp)=\# R_{-2}(L^\perp)/2$. Then the function
\begin{equation}\label{qpb}
  \left. \Phi|_L=
    \frac{\Phi_{12}(Z)}{
      \prod_{r\in R_{-2}(L^\perp)/{\pm 1}} (Z, r)}
    \ \right\vert_{\cD_L}
  \in M_{12+N(L^\perp)}(\Tilde{\Orth}(L),\, \det),
\end{equation}
where in the product over $r$ we fix a finite system of
representatives in $R_{-2}(L^\perp)/{\pm 1}$.  The modular form
$\Phi|_L$ vanishes only on rational quadratic divisors of type
$\cD_v(L)$ where $v\in L^\vee$ is the orthogonal projection of a
$(-2)$-root $r\in II_{2,26}$ on $L^\vee$.
\end{theorem}
We say that the modular form $\Phi|_L$ is a \emph{quasi pull-back} of
$\Phi_{12}$ if the set of roots $R_{-2}(L^\perp)$ is non-empty.

\begin{proof}
We introduce coordinates $Z=(Z_1,Z_2)\in \cD_{II_{2,26}}$ related to
the embedding $L\emb II_{2,26}$ and the splitting~\eqref{proj}, namely
$Z_1\in L\otimes \CC$ and $Z_2\in (L\otimes \CC)^\perp$.  We have
$\Tilde{\Orth}^+(L)< {\Orth}^+(II_{2,26})$ (see
Lemma~\ref{lem:stable}) and we denote by $\tilde g\in{\Orth}^+(II_{2,26})$ the extension of
$g\in \Tilde{\Orth}^+(L)$ by $\tilde g|_{L^\perp}=\id$.

If the root system $R_{-2}(L^\perp)$ is empty, then $\Phi|_L$
is the usual pull-back of $\Phi_{12}$ on $\cD_L$.
Then $\Phi|_L\not\equiv 0$ and  we have
\begin{equation*}
\det (g)\Phi_{12}(Z_1,Z_2)=\Phi_{12}(\tilde g( Z_1,
Z_2))=\Phi_{12}(g\cdot Z_1, Z_2).
\end{equation*}
Therefore the pull-back of $\Phi_{12}$ on $\cD_L$ is a modular form
of weight $12$.  We note that even in this simple case one obtains
interesting reflective modular forms (see \cite[Section 4.2]{GN3}).

If there are roots in $R_{-2}(L^\perp)$ then the pull-back of
$\Phi_{12}$ on $\cD_L$ vanishes identically, and one has to divide by
the equations of the rational quadratic divisors, as in
Equation~\eqref{qpb}.

According to Lemma~\ref{lem:LdivDL} the order of zero of $\Phi_{12}$
along $\cD_L$ is equal to $N(L^\perp)=\#R_{-2}(L^\perp)/2$. Therefore
the non-zero function $\Phi|_L$ is holomorphic on $\cD_L^\bullet$.
Moreover it is homogeneous of degree $12+N(L^\perp)$.  Any $g\in
\Tilde{\Orth}^+(L)$ acts trivially on $L^\perp$ and
$(Z,r)=(Z_1+Z_2,r)=(Z_2,r)$.  Therefore
\begin{equation*}
\left.\frac{\Phi_{12}(g Z_1,Z_2)}
{\prod_{r} (Z_2, r)}\ \right\vert_{\cD_L}=
\left.\frac{\Phi_{12}(\tilde g\cdot (Z_1,Z_2))}{
\prod_{r}(\tilde g\cdot (Z_1,Z_2), r)} \ \right\vert_{\cD_L}
=\left.\det(g)\frac{\Phi_{12}(Z)}{
\prod_{r}(Z, r)}\ \right\vert_{\cD_L}.
\end{equation*}
It follows that $\Phi|_L$ is modular with respect to
$\Tilde\Orth^+(L)$ with character $\det$. We finish the proof using
Koecher's principle.

The zeros of $\Phi|_L$ can be determined using Lemma~\ref{lem:LdivDL}
and the fact that $\Phi_{12}$ vanishes along $\cD_r(II_{2,26})$ with
$r^2=-2$.
\end{proof}

\begin{remark}\label{rem:smalln}
Theorem~\ref{thm:qpb} is still true for $n\le 2$. We can show this by
computing the Fourier expansions of the quasi pull-back. Moreover we
prove in Theorem~\ref{thm:qpbcusp} that the quasi pull-back is always
a cusp form.
\end{remark}
\begin{remark}\label{rem:bigmodulargroup}
The modular group of $\Phi|_L$ might be larger than $\Tilde\Orth^+(L)$
(see, for example, \cite[Lemma 4.4]{GHSdim21}).
\end{remark}
\begin{remark}\label{rem:notonlyBform}
For the applications to the theory of moduli spaces we use the quasi
pull-back of $\Phi_{12}$. It is easy to prove an analogue of
Theorem~\ref{thm:qpb} for an arbitrary modular form whose divisor
consists only of rational quadratic divisors (in the style of
Theorem~\ref{thm:qpbcusp} below).
\end{remark}

In \cite{GHSK3} we showed that some quasi pull-backs for lattices
related to the moduli spaces of polarised $\Kthree$ surfaces are cusp
forms.  In this paper we prove a new, more general result: that the
quasi pull-back construction always gives a cusp form. The main idea of
the proof is to consider the quasi pull-back as a differential
operator (see \cite[Section 6]{GHSK3}).

\subsection{Tube domain realisation}\label{subsect:tube}
We define a Fourier expansion of a modular form $F$ at a
$0$-dimensional cusp.  The Fourier expansion depends on the choice of
affine coordinates at a cusp.  We consider the general case, following
the approach used in \cite[Theorem 5.2 and page 542]{B2}: for more
details see \cite[Section 2.3]{GN3} and \cite{F2}.

A $0$-dimensional cusp of $\cD_L$ is defined by an isotropic
sublattice of rank $1$ or, equivalently, by a primitive isotropic
vector $c\in L$ (up to sign: $c$ and $-c$ define the same cusp). The
choice of $c$ identifies $\cD_L$ with an affine quadric:
\begin{equation*}
\cD_{L,c}=\{Z\in \cD_L^\bullet \mid (Z,c)=1\}\cong \cD_L.
\end{equation*}
The lattice
\begin{equation}
L_c:=c^\perp/c= c^\perp_L/\ZZ c
\end{equation}
is an integral lattice of signature $(1,n-1)$.  We fix an element
$b\in L^\vee$ such that $(c,b)=1$. A choice of $b$ gives a realisation
of the hyperbolic lattice $L_c$ as a sublattice in $L$
\begin{equation}
L_{c}\cong L_{c,b}=L\cap c^\perp\cap b^\perp.
\end{equation}
We have
\begin{equation*}
L\otimes \QQ=L_{c,b}\otimes \QQ\oplus (\QQ b+\QQ c).
\end{equation*}
Using the hyperbolic lattice $L_c\otimes \RR$ we define a positive cone
\begin{equation*}
C(L_c)=\{x\in L_c\otimes \RR\mid (x,x)>0\}.
\end{equation*}

We may choose $C^+(L_c)$, one of the two connected components of
$C(L_c)$ so that corresponding tube domain, which is  the  complexification
of $C^+(L_c)$
\begin{equation*}
\cH_c=L_c\otimes \RR+i\, C^+(L_c)
\end{equation*}
has an isomorphism $\cH_c\to \cD_{L,c}\cong \cD_L$ by
\begin{equation}\label{affine-z}
z\mapsto [z]= z\oplus\big(b-\frac {(z,z)+(b,b)}{2}c\big)\quad
(z\in \cH_c,\ [z]\in \cD_{L,c}).
\end{equation}
Using the coordinate $z\in \cH_c$ defined by the choice of $c$ and $b$
we can identify an arbitrary modular form $F$ of weight $k$ with a modular form
$F_{c,b}$ (or simply $F_c$) on the tube domain $\cH_c$:
\begin{equation*}
F(\lambda [z])=\lambda^{-k} F_{c,b}(z).
\end{equation*}

\subsection{Fourier expansion at $0$-dimensional cusps}\label{subsect:fourier}
In order to define the Fourier expansion at the cusp $c$ we consider
an unipotent subgroup of the stabiliser $\Tilde \Orth^+(L)$, the
subgroup of the Eichler transvections. For any $a\in c^\perp_L$ the
map
\begin{equation*}
t'(c,a)\colon v \Mapsto v-(a,v)c \qquad (v\in c_L^\perp)
\end{equation*}
belongs to the orthogonal group $\Orth(c_L^\perp)$. It has the unique
orthogonal extension on $L$ which  is given by the map
\begin{equation}\label{t2}
t(c,a)\colon v\Mapsto v-(a,v)c+(c,v)a-\frac{1}{2}(a,a)(c,v)c.
\end{equation}
This element is called an \emph{Eichler transvection}: see
\cite[Section 3]{E} and \cite{GHScomm}.  We note that $t(c,a)\in
\Tilde\SO^+(L)$ for any $a\in c^\perp_L$, that $t(c,a)(c)=c$, and that
for $a,\ a'\in c_L^\perp$
\begin{equation*}
t(c,a)t(c,a')=t(c,a+a')\quad\text{and}\quad
t(c,a)^{-1}=t(c,-a).
\end{equation*}
We can identify the lattice $L_{c,b}$ with the corresponding group of
transvections $E_c(L)=\latt{t(c,a)\mid a\in L_c }$. The group $E_c(L)$
is the unipotent radical of the parabolic subgroup associated to $c$.
A direct calculation shows that $t(c, a)$ acts as linear translation
in the affine coordinates~\eqref{affine-z}:
\begin{equation*}
t(c,a)([z])=[z+a].
\end{equation*}
Let $F\in M_k(\Tilde\SO^+(L))$.  Then $F_c(z+a)=F_c(z)$ for all $a\in
L_{c,b}$ and we obtain the Fourier expansion of $F$ at the cusp $c$:
\begin{equation}\label{F-exp}
F_c(z)=\sum_{l\in L_{c,b}^\vee} f(l)\exp{(2\pi i \,(l,z))}.
\end{equation}
The function $F_c(z)$ is holomorphic at the cusp $c$ if the Fourier
coefficient $f(l)$ can be different from $0$ only if its index $l\in
L_{c,b}^\vee$ belongs to the closure of the positive cone $C^+(L_c)$.
Another formulation of this fact is $(l,Y)>0$ for any $Y$ in the
positive cone $C^+(L_c)$.

\begin{remark}\label{rem:changeb}
In general the Fourier expansion~\eqref{F-exp} depends on the choice
of $b$.  For another $b'$ the Fourier coefficients will be different by
a factor $\exp(2\pi i (l,b-b'))$ which is a root of unity (see details
in \cite[\S{2.3}]{GN3}). In particular the Fourier coefficient $f_0$
(the value of $F$ at the cusp $c$) is well-defined.
\end{remark}

\begin{remark}\label{rem:cuspstabiliser}
The stabiliser of the cusp $c$ in $\Tilde\Orth^+(L)$ is isomorphic
to the semi-direct product of $\Tilde\Orth^+(L_c)$ and $E_c(L)$.
\end{remark}

\begin{remark}\label{rem:simplestcusp}
Let $\div_L (c)=N$. Then $\det L=N^2\det L_c$. If $\div_L(c)=1$ then
$c$ can be completed to a hyperbolic plane in $L$ and $L=U\oplus
L_c$. This cusp is called the \emph{simplest cusp}.  If $\div(c)>1$ then
$c/\div(c)+L$ is an isotropic element of the discriminant group $D(L)$
and $|D(L)|=|\det L|$ is divisible by $\div(c)^2$.  If $D(L)$ contains
no non-trivial isotropic elements (in this case the lattice is
maximal) then all $0$-dimensional cusps are equivalent to the simplest cusp.
\end{remark}

\begin{remark}
If $L$ contains two hyperbolic planes then one can use the Eichler
criterion (see Lemma~\ref{lem:eichler}) in order to classify the
$0$-dimensional cusps with respect to the action of
$\Tilde\SO^+(L)$. The orbit of $c$, in this case, is uniquely
determined by the isotropic class $c/\div (c)+L$ in the discriminant
group $D(L)$.

If $D(L)$ does not contain isotropic elements (in particular if
$\det(L)$ is square free) then the modular variety
$\Tilde\Orth^+(L)\setminus \cD_L$ has only one $0$-dimensional cusp.
\end{remark}

\begin{remark}\label{rem:sublattice}
If $[\Orth(L):\Gamma]<\infty$ and $F\in M_k(\Gamma, \chi)$ then we can define
the Fourier expansion using a sublattice $L_c(\Gamma)$ of finite index in $L_c$
corresponding to the group $E_c(L)\cap \ker \chi$.
\end{remark}

\subsection{Properties of quasi pull-back}\label{subsect:qpbproperties}

We may consider the quasi pull-back of other modular forms, not only
$\Phi_{12}$. If $L$ is of signature $(2,n)$ and $r\in L$ is primitive
with $(r,r)<0$ then $\cD_r(L)=\cD_{tr}(L)$ for any $t\in \QQ^\times$,
and we write $r^\perp=r_L^\perp$ if there is no ambiguity.

\begin{theorem}\label{thm:qpbcusp}
Let $L$ be an integral lattice of signature $(2,n)$.  Suppose that the
modular form $F\in M_k(\Tilde{\SO}^+(L))$ vanishes with order $m>0$ on
the rational quadratic divisor $\cD_r(L)$ where $r$ is a primitive
vector in $L$ with $(r,r)<0$.  We define the quasi pull-back of $F$ on
the domain $\cD_{r^\perp}$ of complex dimension $n-1$ by
\begin{equation}\label{quasi-r}
F|_{r^\perp}=\left.{\frac{F(Z)}{(Z,r)^m}} \right\vert_{\cD_{r^\perp}}.
\end{equation}
Then the quasi pull-back is a cusp form of weight $k+m$
\begin{equation*}
F|_{r^\perp}\in S_{k+m}(\Tilde{\SO}^+(r^\perp)).
\end{equation*}
\end{theorem}

\begin{proof}
There are two parts to the assertion: that $F|_{r^\perp}$ is a modular
form, and that it vanishes at every cusp.

For the first part, we have the inclusions
\begin{equation*}
r^\perp\oplus \ZZ r\subset L\subset L^\vee\subset
(r^\perp)^\vee \oplus \ZZ \frac{r}{(r,r)},
\end{equation*}
and $r^\perp$ has signature $(2,n-1)$.  We consider the two embeddings
$\cD_{r^\perp}\emb \cD_L$ and $\Tilde{\SO}^+(r^\perp)\emb
\Tilde{\SO}^+(L)$.  We note that any element $g\in
\Tilde{\SO}^+(r^\perp)$ extends to $\tilde g\in \Tilde{\SO}^+(L)$ by
acting trivially on $r$ (see Lemma~\ref{lem:stable}). Therefore any
$\tilde g$ preserves $\cD_{r^\perp}$ and $\tilde g\cdot r=r$.  The
function $F|_{r^\perp}$ is a holomorphic function on $\cD_{r^\perp}$
and it is homogeneous of degree $k+m$.  For any $g\in
\Tilde{\SO}^+(r^\perp)$ we have
\begin{equation*}
\left.\frac{F(\tilde g Z)}{(\tilde gZ,r)^m}\right\vert_{\cD_{r^\perp}}
= \left.\frac{F(Z)}{(Z,r)^m}\right\vert_{\cD_{r^\perp}}.
\end{equation*}
Therefore the quasi pull-back $F|_{r^\perp}$ is
$\Tilde{\SO}^+(r^\perp)$-invariant. If instead $F$ has character $\chi$ then
$F_{r^\perp}$ tranforms according to the induced character
$\chi|_{\Tilde{\SO}^+(r^\perp)}$.
If $n>3$ then using Koecher's principle we conclude that
\begin{equation*}
F|_{r^\perp}\in M_{k+m}(\Tilde{\SO}(r^\perp)).
\end{equation*}
If $n\le 3$ we cannot apply Koecher's principle and instead we must
use the Fourier expansion to check that the quasi pull-back is a
modular form as well as to show the vanishing at the cusps.

We calculate the Fourier expansion of the quasi pull-back at an
arbitrary $0$-dimensional cusp of $r^\perp$.  Let $c\in r^\perp$ be a
primitive isotropic vector in $r^\perp$ (if there are any: if not,
there is nothing more to prove).  We fix a vector $b\in
(r^\perp)^\vee$. Then we can define the two homogeneous domains
$\cH_{c}(r^\perp)$ and $\cH_c(L)$ because the vector $c$ defines also
a $0$-dimensional cusp of $L$.  We write $Z\in \cH_c(L)$ in the form
$Z=Z_1+zr$ where
\begin{equation*}
Z_1\in r^\perp\otimes \CC,\ z=x+iy\in \CC, \ \ (\Im Z_1,\, \Im Z_1)+(r,r)y^2>0.
\end{equation*}
If $[Z]$ is the image of $Z$ in $\cD_c(L)$ (see \eqref{affine-z}) then
$([Z], r)=(r,r)z$.  Therefore the equation of the divisor $\cD_r(L)$
in the affine coordinates $Z=Z_1\oplus zr\in \cH_c$ is $z=0$. The
quasi pull-back $F|_{r^\perp}$ is equal, up to a constant, to the
first non-zero coefficient in the Taylor expansion of $F_{c,b}(Z)$ in
$z$.

Consider the Fourier expansion
\begin{equation*}
F_{c,b}(Z)=\sum_{l\in L_{c,b}^\vee} f(l)\exp{(2\pi i \,(l,Z))}.
\end{equation*}
The modular form $F_{c,b}(Z)$ is holomorphic at the boundary.
Therefore $l$ belongs to the closure of the dual cone, in particular,
$(l,l)\ge 0$.  The vectors $c$ and $b$ are orthogonal to $r$ and the
lattice $(r^\perp)_{c,b}\oplus \ZZ r$ is a sublattice of finite index
in the lattice of translations $L_{c,b}$.  For any $Z_1\in
\cH_{c}(r^\perp)\subset r^\perp\otimes \CC$ we consider $z=x+iy$ such
that $Z=Z_1\oplus zr\in \cH_c(L)$.  (If $y$ is small enough then $(\Im
Z_1, \Im Z_1)>-(r,r)y^2>0$.)  Therefore we can rewrite the Fourier
expansion using this parametrisation
\begin{equation}\label{fouriertube}
F_{c,b}(Z_1+zr)=\sum_{l_1\in (r^\perp)_{c,b}^\vee,\ l_2\in \ZZ r/(r,r)}
f(l_1+l_2)\exp{(2\pi i \,(l_1,Z_1)+ z(l_2, r))}
\end{equation}
where $(l,l)=(l_1+l_2, l_1+l_2)=(l_1,l_1)+(l_2,l_2)\ge 0$.
The Fourier coefficients  of $(F|_{r^\perp})_{c,b}$ are proportional
to the Fourier coefficients of the $m$-th Taylor coefficient
\begin{equation*}
\left.\frac{\partial^m F_{c,b}(Z_1+zr)}{(\partial z)^m}\right \vert_{z=0}.
\end{equation*}
The derivatives of the terms in the Fourier
expansion~\eqref{fouriertube} vanish if $l_2=0$. If $(l_2,l_2)<0$ then
$(l_1,l_1)>0$, so nonzero Fourier coefficients occur only when the
index $l_1$ has positive square.  We have proved this, which is much
stronger than just the vanishing of the zeroth coefficient, for an arbitrary
$0$-dimensional cusp of $r^\perp$, so we have shown that
$F|_{r^\perp}$ is holomorphic at the boundary even for $n=1$ or $n=2$.
Moreover the value of $F|_{r^\perp}$ at an arbitrary $0$-dimensional
cusp is zero.

The Baily-Borel compactification of $\Tilde \SO^+(r^\perp)\setminus
\cD_{r^\perp}$ contains only boundary components of dimension $0$ and
$1$ (the latter only if $r^\perp$ contains a totally isotropic
sublattice of rank two). The Fourier expansion at a $1$-dimensional
cusp $E$ is called Fourier--Jacobi expansion (see \cite{Ba},
\cite{P-S}, \cite{G2}).  The value of a modular form $G$ on the
boundary component $E$ is given by the Siegel operator $\Phi_E(G)$
(see \cite{BB}). This is the zeroth coefficient of the Fourier--Jacobi
expansion which is a modular form with respect to a subgroup of
$\SL(2,\ZZ)$.  

The boundary of a $1$-dimensional cusp $E$ is a union
of some $0$-dimensional cusps.  We consider the Fourier expansion of
$\Phi_E(F|_{r^\perp})$ at a $0$-dimensional cusp $c$ associated to $E$
as a part of the Fourier expansion of $(F|_{r^\perp})_{c,h}$ (see
\cite[Section 5.2]{Ko1} for the Fourier expansion of Fourier--Jacobi
coefficients of modular forms).  The indices of the Fourier
coefficients of $\Phi_E(F|_{r^\perp})$ are of hyperbolic norm
$0$. Therefore $\Phi_E(F|_{r^\perp})\equiv 0$ because, as shown above, all such
coefficients in $(F|_{r^\perp})_{c,h}$ are equal to zero.  Therefore
the quasi pull-back $F|_{r^\perp}$ is a cusp form.
\end{proof}

Using Theorem~\ref{thm:qpbcusp} we prove that the quasi pull-back
defined in Theorem~\ref{thm:qpb} is a cusp form.
\begin{corollary}\label{cor:qpb}
Let $L\emb II_{2,26}$ be a nondegenerate sublattice of signature
$(2,n)$, $n\ge 1$. We assume that the set $R_{-2}(L^\perp)$ of
$(-2)$-roots in $L^\perp$ is non-empty. Then the quasi pull-back
$\Phi|_L\in S_{12+N(L^\perp)}(\Tilde{\Orth}(L),\, \det)$ of the
Borcherds form $\Phi_{12}$ is a cusp form.
\end{corollary}
\begin{proof}
To prove the corollary we divide the procedure of the quasi pull-back
of Theorem~\ref{thm:qpb} into finitely many steps.

First, we take a root $r_1\in R_{-2}(L^\perp)\ne \emptyset$ and define
$M_1=(r_1)^\perp_{II_{2,26}}$. The form $\Phi_{12}$ has a zero of order
$1$ on $\cD_{r_1}(II_{2,26})$.  According to Theorem~\ref{thm:qpbcusp}
we have
\begin{equation*}
\Phi|_{M_1}=
\left.\frac{\Phi_{12}(Z)}{(r_1,Z)}\right\vert_{M_1}\in S_{13}(\Tilde{\Orth}^+(M_1),\det).
\end{equation*}
We note that the cusp form $\Phi|_{M_1}$ might have divisors different
from the $(-2)$-divisors of $\Phi_{12}$. If $s\in R_{-2}(L^\perp)$
such that $(r_1,s)\ne 0$ then $s_1=\pr_{M_1^\vee}(s)$ has negative
norm $-2<(s_1,s_1)<0$ and $\Phi|_{M_1}$ vanishes along
$\cD_{s_1}(M_1)$ according to Lemma~\ref{lem:LdivDL}.  Therefore the
divisors of $\Phi|_{M_1}$ are rational quadratic divisors defined by
some vectors in $v\in M_1^\vee$.  But $\cD_v(M_1)=\cD_{tv}(M_1)$ and
we can fix $t$ in order to have a primitive vector $tv\in M_1$.

Second, we consider the lattice $L^\perp\cap M_1$. Suppose first that
there is no vector $v\in L^\perp\cap M_1$ such that $\Phi|_{M_1}$
vanishes on $\cD_v(M_1)$.  In this case $\Phi|_L$ is equal to the
pull-back of $\Phi|_{M_1}$ on $\cD_L$.  The pull-back of a cusp form
is a cusp form and the proof is finished.

Otherwise, if $r_2 \in L^\perp\cap M_1$ is a vector such that
$\Phi|_{M_1}$ vanishes on $\cD_{r_2}(M_1)$, then we define
$M_2=(r_2)^\perp_{M_1}=\latt{r_1,r_2}^\perp_{II_{2,26}}$.  As in
Theorem~\ref{thm:qpbcusp} we have
\begin{equation*}
\Phi|_{M_2}=
\left.\frac{\Phi|_{M_1}}{(r_2,Z)^m}\right\vert_{M_2}\in S_{13+m}(\Tilde{\Orth}^+(M_2),\det)
\end{equation*}
where $m\ge 1$ is the degree of zero of $\Phi|_{M_1}$ on
$\cD_{r_2}(M_1)$.

The function $\Phi|_{M_2}$ is a modular form vanishing along some
rational quadratic divisors.  We can repeat the procedure described
above for $M_2$.  After a finite number of steps we get the cusp form
$\Phi|_L$.
\end{proof}

\begin{proposition}\label{prop:qpbvanishing}
Let $L$ be one of the lattices $L_{2d}$, $L_{2,2d}$ or $L_{Q(d)}$
defined in \eqref{L2d}--\eqref{L2d-nonsplit}. We assume that there
exists an embedding $L\emb II_{2,26}$ such that the weight of the
quasi pull-back $\Phi|_L$ is smaller than the dimension of
$\cD_L$. Then $\Phi|_L$ vanishes along the ramification divisor of
the modular projection
\begin{equation*}
\pi\colon  \cD_L \to \Tilde\Orth^+(L)\setminus \cD_L.
\end{equation*}
\end{proposition}
\begin{proof}
We give here a proof which works for all cases including the moduli
spaces of polarised holomorphic symplectic O'Grady varieties (see
\cite[Corollary 4.6]{GHSdim21}).

The components of the branch divisor of $\pi$ are $\cD_r(L)$ where
$r\in L$ and $\sigma_r$ or $-\sigma_r$ is in $\Tilde\Orth^+(L)$ (see
\cite[Corollary 2.13]{GHSK3} and Equation~\eqref{br-div} above).  If
$\sigma_r\in \Tilde\Orth^+(L)$, then $\Phi|_L$ vanishes along
$\cD_r(L)$ because $\Phi|_L$ is modular with character $\det$.  We
have to prove that $\Phi|_L$ vanishes also on $\cD_r(L)$ with
$-\sigma_r\in \Tilde\Orth^+(L)$.  To prove this we use the
transitivity of the quasi pull-back construction.  Let $r\in L\emb
II_{2,26}$ and $L_r=r^\perp_L$. Then
\begin{equation*}
(\Phi|_L)|_{L_r}=\Phi|_{L_r}.
\end{equation*}
We have to consider three cases.
\smallskip

1) Let $r\in L=L_{2d}$. Then $\rank L_r=18$.  According to
Corollary~\ref{cor:reflK3} $\det |(L_r)^{\perp}_{II_{2,26}}|=1$ or
$4$.  In \cite[Table I]{CS} one can find all classes of the
indecomposable lattices of small rank and determinant. Analysing that
table we find three classes of lattices of rank $8$ of determinant
$1$, $2$ or $4$:
\begin{equation*}
E_8\ \  (|R(E_8)|=240), \quad E_7\oplus A_1\ \  (|R(E_7\oplus
A_1)|=128),\quad D_8\ \ (|R(D_8)|=112).
\end{equation*}
Therefore $\Phi_{12}$ has a zero of order at least $56$ along
$\cD_r(L)$.
  The modular form $\Phi|_L$ is of weight $k<19$. Therefore
$R(L^\perp_{II_{2,26}})$ has at most $12$ roots. Therefore $\Phi|_L$
vanishes on $\cD_r(L_{2d})$ with order at least $50$.
\smallskip

2) Let $r\in L=L_{2,2d}$. Then $\rank L_r=19$.  We described
reflective vectors $r$ with $-\sigma_r\in \Tilde\Orth^+(L)$ in
Proposition~\ref{prop:reflvectors}.  There are three possible cases:
\begin{equation*}
r^2=2d, \ \div(r)=d\  {\rm or}\  2d
\quad{\rm and }\quad r^2=d=\div(r)\ (d\ {\rm  is\  odd}).
\end{equation*}
According to Lemma~\ref{lem:hperp}, $\det L_r=2$, $4$ or $8$.
Therefore the rank $7$ lattice $(L_r)^{\perp}_{II_{2,26}}$ has the
same determinant.  According to \cite[Table I]{CS} there are six
possible classes of such lattices:
\begin{equation*}
E_7,\ D_7,\ D_6\oplus A_1,\ A_7,\  [D_6\oplus \latt{8}]_2,\
[E_6\oplus \latt{24}]_3,
\end{equation*}
where $[M]_n$ denotes an overlattice of order $n$ of $M$.  Any of
these lattices contains at least $60$ roots ($|R(D_6)|=60$).  The
modular form $\Phi|_L$ is of weight $k<20$. Therefore
$R(L^\perp_{II_{2,26}})$ has at most $14$ roots and $\Phi|_L$ vanishes
on $\cD_r(L_{2, 2d})$ with order at least $23$.
\smallskip

3) Let $r\in L=L_{Q(d)}$. In this case $d\equiv 3\mod 4$, and the
discriminant group $D(L_{Q(d)})$ is cyclic of order $d$.  Using
Proposition~\ref{prop:reflvectors} and Lemma~\ref{lem:hperp} we obtain
that only one class for $(L_r)^{\perp}_{II_{2,26}}$ is possible.  This
is the lattice $E_7$.  We finish the proof as above.
\end{proof}

\section{Arithmetic of root systems}\label{sect:rootsystems}

In this section we finish the proof of Theorem~\ref{thm:K3gt}. To prove
it we use the low weight cusp form trick (Theorem~\ref{thm:gt}) by
using the quasi pull-back of the Borcherds modular form $\Phi_{12}$ to
construct cusp forms of small weight with large divisor.

According to Theorem~\ref{thm:qpb}, Theorem~\ref{thm:qpbcusp} and
Proposition~\ref{prop:qpbvanishing} the main point for us is the
following.  We want to know for which $2d>0$ there exists a vector
\begin{equation}\label{orth2}
  l\in E_8,\ l^2=2d,\ l\ \text{ is orthogonal to at least $2$ and at
    most $12$ roots.}
\end{equation}
To solve this problem we use the combinatorial geometry of the root
system $E_8$ together with the theory of quadratic forms.  First we
give some properties of the lattice $E_8$ and we show how one can
construct the first polarisations of general type in
Theorem~\ref{thm:K3gt}.  After that we outline the answer to the
question in~\eqref{orth2}.

\subsection{Vectors in $E_8$ and $E_7$}\label{subsect:E8}

By definition, the lattice $D_8$ is an even sublattice of the
Euclidean lattice~$\ZZ^8$
\begin{equation*}
D_8=\{l=(x_1,\dots,x_8)\in \ZZ^8\,|\, x_1+\dots+x_8\in 2\ZZ\}.
\end{equation*}
The determinant of $D_8$ is equal to $4$.
We denote by $e_1,\dots,e_8$ the Euclidean basis of $\ZZ^8$
($(e_i, e_j)=\delta_{ij}$).
The lattice $E_8$ is the double  extension of $D_8$:
\begin{equation*}
E_8=D_8\cup (\frac{e_1+\dots+e_8}{2}+D_8).
\end{equation*}
We consider the Coxeter basis of simple roots in $E_8$ (see
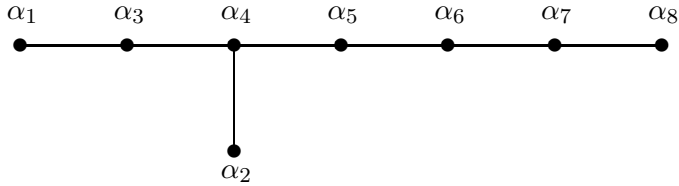
\cite{Bou})

\noindent\hskip-.5cm\begin{picture}(300,10)(55,10)
\put(100,0){\circle*{5}}
\put(95,10){$\alpha_1$}
\put(100,0){\vector(1,0){42}}
\put(140,0){\circle*{5}}
\put(135,10){$\alpha_3$}
\put(140,0){\vector(1,0){42}}
\put(180,0){\circle*{5}}
\put(175,10){$\alpha_4$}
\put(180,1){\vector(0,-1){43}}
\put(180,-40){\circle*{5}}
\put(175,-50){$\alpha_2$}
\put(180,0){\vector(1,0){42}}
\put(220,0){\circle*{5}}
\put(215,10){$\alpha_5$}
\put(220,0){\vector(1,0){42}}
\put(260,0){\circle*{5}}
\put(255,10){$\alpha_6$}
\put(260,0){\vector(1,0){42}}
\put(300,0){\circle*{5}}
\put(295,10){$\alpha_7$}
\put(300,0){\vector(1,0){42}}
\put(340,0){\circle*{5}}
\put(335,10){$\alpha_8$}
\end{picture}
\vskip2cm
\noindent
where
\begin{gather*}
\alpha_1=\frac 1{2}(e_1+e_8)-\frac 1{2}(e_2+e_3+e_4+e_5+e_6+e_7),\\
\alpha_2=e_1+e_2,\quad \alpha_k=e_{k-1}-e_{k-2}\ \ (3\le k\le 8)
\end{gather*}
and $E_8=\latt{\alpha_1,\dots \alpha_8}_\ZZ$.  The lattice $E_8$
contains $240$ roots.  We recall that any root is a sum of simple
roots with integral coefficients of the same sign.  The fundamental
weights $\omega_j$ of $E_8$ form the dual basis, so
$(\alpha_i,\omega_j)=\delta_{ij}$. The formulae for the weights are
given in \cite[Tabl. VII]{Bou}. We shall use the Cartan matrix of the
dual basis
\begin{equation}\label{Cartan-w}
((\omega_i,\omega_j))=\left(
\begin{matrix}
4&5&7&10&8&6&4&2\\
5&8&10&15&12&9&6&3\\
7&10&14&20&16&12&8&4\\
10&15&20&30&24&18&12&6\\
8&12&16&24&20&15&10&5\\
6&9&12&18&15&12&8&4\\
4&6&8&12&10&8&6&3\\
2&3&4&6&5&4&3&2
\end{matrix}
\right).
\end{equation}
Let us assume that $l_{2d}\in E_8$ and $(l_{2d})_{E_8}^\perp$ contains
exactly $12$ roots. We consider two cases
\begin{equation*}
R((l_{2d})_{E_8}^\perp)=A_2\oplus 3A_1\quad{\rm or }\quad
R((l_{2d})_{E_8}^\perp)= A_2\oplus A_2.
\end{equation*}
There are four possible choices of the subsystem $A_2\oplus 3A_1$
inside the Dynkin diagram of $E_8$. There are four choices for a copy
of $A_2$:
\begin{equation*}
A_2=\latt{\alpha_1, \alpha_3},\ \
\latt{\alpha_2, \alpha_4},\ \
\latt{\alpha_5, \alpha_6},\ \
\latt{\alpha_7, \alpha_8}.
\end{equation*}
Then the three pairwise orthogonal copies of $A_1$ are defined
automatically.  For example, if $A_2^{(1)}=\latt{\alpha_5, \alpha_6}$
then $3A_1^{(1)}=\latt{\alpha_2}\oplus \latt{\alpha_3}\oplus
\latt{\alpha_8}$.  The sum $A_2^{(1)}\oplus 3A_1^{(1)}$ is the root
system of the orthogonal complement of the vector
$l_{5,6}=\omega_1+\omega_4+\omega_7\in E_8$.  In fact, if
$r=\sum_{i=1}^8 x_i \alpha_i$ is a positive root ($x_i\ge 0$) then
$(r,l_{5,6})=x_1+x_4+x_7=0$. Therefore $x_1=x_4=x_7=0$ and $r$ belongs
to $A_2^{(1)}\oplus 3A_1^{(1)}$ (see the Dynkin diagram of $E_8$
above). Using the matrix~\eqref{Cartan-w} we find $l_{5,6}^2=46$.
Doing similar calculations with
\begin{equation*}
l_{1,3}=\omega_4+\omega_6+\omega_8,\quad
l_{2,4}=\omega_3+\omega_5+\omega_7,\quad
l_{7,8}=\omega_1+\omega_4+\omega_6
\end{equation*}
we obtain polarisations for $d=50$, $54$ and $57$
\begin{equation*}
l_{1,3}^2=2\cdot 50, \qquad l_{2,4}^2=2\cdot 54, \qquad l_{7,8}^2=2\cdot 57.
\end{equation*}
To get a good vector for $d=52$ we consider the lattice $M=A_2\oplus
A_2=\latt{\alpha_3, \alpha_4}\oplus \latt{\alpha_6, \alpha_7}$ in
$E_8$. Then $M$ is the root system of the orthogonal complement of
$l_{104}=\omega_1+\omega_2+\omega_5+\omega_8$ and $l_{104}^2=2\cdot
52$.

In this way we construct the first five polarisations of general
type from the list of Theorem~\ref{thm:K3gt}.  These elementary
arguments, similar to the arguments of Kondo in \cite{Ko2}, do not
give the whole list but only a few early cases.  A sufficient
condition for existence of vectors satisfying~\eqref{orth2} is
given in the theorem below.

\begin{theorem}\label{thm:mainineq}
A vector $l$ satisfying~\eqref{orth2} does exist if the inequality
\begin{equation}\label{mineq}
4N_{E_7}(2d)>28N_{E_6}(2d)+63N_{D_6}(2d)
\end{equation}
is valid, where $N_L(2d)$ denotes the number of representations of
$2d$ by the lattice $L$.
\end{theorem}
\begin{proof}
We use bouquets of copies of $A_2$ in $E_8\setminus E_7$.  The root
system $R(E_8)$ is a disjoint union of $126$ roots of $E_7$ and the
bouquet of $28$ copies of $A_2$ centred in $A_1$.  This fact explains
the coefficients $28$ and $63$ in the right hand side
of~\eqref{mineq}. See \cite[Section 7]{GHSK3} for more details.
\end{proof}
In \cite{GHSK3} we found explicit formulae for the numbers of
representations and we proved that
\begin{equation*}
N_{E_7}(2m)>\frac{24\pi^3}{5\zeta(3)}, \qquad
21m^2>N_{D_6}(2d),\qquad 103.69m^2> N_{E_6}(2d).
\end{equation*}
These inequalities give a finite set of $d$ for which \eqref{mineq} is
not valid. Analysing the corresponding theta series we found the set
of all such $d$, containing $131$ numbers.  The five tables in
\cite[Section 7]{GHSK3} and the argument above for $d=52$ give
the list of polarisations of general type of Theorem~\ref{thm:K3gt}.

The geometric genus of $\cF_{2d}$ is positive if there exists a cusp
form of canonical weight $19$.  For each $d$ not in the general type
list but satisfying $d\ge 40$ and $d\ne 41$, $44$, $45$ or $47$, there
exists a vector $h_{2d}\in E_8$ of length $2d$ orthogonal exactly to
$14$ roots.  The corresponding quasi pull-back is a cusp forms of
canonical weight.  For $d=42$, $43$, $51$, $53$ and $55$ such cusp
forms were constructed by Kondo in \cite{Ko2}. For other $d$
see~\cite{GHSK3}.

A similar arithmetic method applied to $E_7$ (see \cite[Section
  4]{GHSsymp}) gives the proof of Theorem~\ref{thm:splitgt}. There is
a significant technical difficulty in the case of $E_7$.  The proof
involves estimating the number of ways of representing certain
integers by various root lattices of odd rank.  In \cite[Section
  5]{GHSsymp}) we gave a new, clear, explicit version of Siegel's
formula for this number in the odd rank case. It may be expressed
either in terms of Zagier $L$-functions or in terms of the H.~Cohen
numbers. For example we obtained a new, very short formula for the
number of representations by $5$ squares (see \cite[Section
  4]{GHSsymp} for the details).

\subsection{Binary quadratic forms in $E_8$}\label{subsect:QE8}

We saw in Example~\ref{ex:debarrevoisin} that the results of
Debarre and Voisin \cite{DV} imply that the moduli space
$\cM^{[2],\text{non-split}}_{2\cdot 11}$ of Beauville degree $d=11$ is
unirational. We note that $\cM^{[2],\text{split}}_{2\cdot 11}$, which
is a finite covering of $\cM^{[2],\text{non-split}}_{2\cdot 11}$ by
\eqref{2d-Q(d)}, has non-negative Kodaira dimension.
Theorem~\ref{thm:splitgt} shows that there can be at most $11$
exceptional split polarisations of non general type.
Proposition~\ref{prop:nonsplitgt} hints that one can expect a theorem
for the non-split case, which we hope to prove in the future, similar
to Theorem~\ref{thm:splitgt}.

We recall that the Beauville degree $d\equiv -1 \mod 4$ if the
polarisation $h_{2d}$ is of non-split type.  For small $d$ we can
calculate the class of the orthogonal complement of $Q(d)$.  According
to \cite[Table $0$]{CS} the rank $6$ lattice $Q(d)^\perp_{E_8}$ of
determinant $d$ contains at least $24$ roots for $d=3$, $7$, $11$ and
$15$.  One can continue this analysis but we propose below a simple
algorithm which gives us the first ``good" embeddings of $Q(d)$ in
$E_8$.

\begin{proposition}\label{prop:nonsplitgt}
The moduli spaces $\cM^{[2],\text{non-split}}_{2\cdot 39}$
and $\cM^{[2],\text{non-split}}_{2\cdot 47}$  are of general type.
\end{proposition}

\begin{proof}
Let $Q(d)(-1)= \left(\begin{matrix}2&-1\\-1&2c\end{matrix}\right)$,
  where $c=(d+1)/4$, be the binary quadratic form associated to a
  non-split polarisation of degree $2d$ (see \eqref{L2d-nonsplit}). To
  make the notation simpler we denote this binary form by $Q(d)$.  We
  have to embed $Q(d)$ in $E_8$ so as to satisfy
\begin{equation*}
2\le |R(Q(d)^\perp_{E_8})|\le 14.
\end{equation*}
We describe below a method (we call it $\frac{1}2(++)$-algorithm)
which gives many such embeddings.  We take the following realisation
of the binary quadratic form of determinant $d$ in $E_8$:
\begin{equation*}
Q(d)=\latt{e_2-e_1,\,v_{2c}}=
\latt{e_2-e_1,\ \frac{1}{2}\,(e_1-e_2+x_3e_3+\dots+x_8e_8)},
\end{equation*}
where  $x_i\ge 0$ and
$2c=(2+x_3^2+\dots+x_8^2)/4$.
The second generator $v_{2c}$ of the binary quadratic form
is a half-integral element of $E_8$.
According to the definition  of $E_8$ given above
\begin{equation*}
v_{2c}+\frac 1{2}(e_1+\dots+e_8)=e_1+\frac{x_3+1}2\, e_3+\dots
+\frac{x_8+1}2\, e_8\in D_8.
\end{equation*}
Therefore
$x_3\equiv \dots \equiv x_8\mod 2$ and  $x_3+\dots+x_8\equiv 0\mod 4$.
Now we can find all roots orthogonal to $Q(d)$.

First, we consider the integral roots of $E_8$.  The roots
$\pm(e_1+e_2)$ are orthogonal to $Q(d)$ therefore
$R(Q(d)^\perp_{E_8})$ is not empty.  Then the integral roots
$\pm(e_i\pm e_j)$ ($i,j>2$) are orthogonal to $Q(d)$ if and only if
$x_i=\mp x_j$.

Second, a half-integral root orthogonal to $Q(d)$ has the form
\begin{equation*}
\pm \frac1{2}(e_1+e_2+\sum_{i=3}^{8}(-1)^{\varepsilon_i}e_i)
\end{equation*}
where the number of $-$ signs in the sum is even.  Therefore $v_{2c}$
is orthogonal to a half-integral root if and only if we can divide the
vector of coefficients $(x_3,\dots,x_8)$ in two parts containing even
number of terms and with the same sum.  If this is possible there are
two pairs of roots orthogonal to $v_{2c}$.  For example, if there is a
root of type $\pm (++;++++--)$ then we also have $\pm(++;----++)$.  The
rules described above give the number (it is always positive) of roots
orthogonal to $Q(d)$.

To finish the proof of Proposition~\ref{prop:nonsplitgt} we consider
two vectors
\begin{equation*}
v_{20}=\frac{1}2(1,1;5,5,3,3,3,1)\quad{\rm and}\quad
v_{24}=\frac{1}2(1,1;7,5,3,3,1,1).
\end{equation*}
The quadratic forms $Q(39)$ and $Q(47)$ are orthogonal to exactly $14$
roots in $E_8$.  As above, using Theorem~\ref{thm:qpb},
Theorem~\ref{thm:qpbcusp} and Proposition~\ref{prop:qpbvanishing} we
see that the quasi pull-back of $\Phi_{12}$ on the modular variety
defined by the lattice $L_{Q(d)}$ is a cusp form of weight $19$
vanishing on the ramification divisor. We finish the proof using
Theorem~\ref{thm:gt}.
\end{proof}

\bigskip
\noindent
V.~Gritsenko\\
Universit\'e Lille 1\\
Laboratoire Paul Painlev\'e\\
F-59655 Villeneuve d'Ascq, Cedex\\
France\\
{\tt valery.gritsenko@math.univ-lille1.fr}
\bigskip

\noindent
K.~Hulek\\
Institut f\"ur Algebraische Geometrie\\
Leibniz Universit\"at Hannover\\
D-30060 Hannover\\
Germany\\
{\tt hulek@math.uni-hannover.de}
\bigskip

\noindent
G.K.~Sankaran\\
Department of Mathematical Sciences\\
University of Bath\\
Bath BA2 7AY\\
England\\
{\tt g.k.sankaran@bath.ac.uk}
\end{document}